\documentclass[a4, 12pt]{amsart}
\usepackage{amssymb,amstext,amscd,amsthm,amsfonts,mathdots,mathrsfs}
\usepackage{pdflscape}
\usepackage{hyperref}
\usepackage{pdflscape}
\usepackage{setspace}
\usepackage[all]{xy}
\usepackage{enumerate}
\usepackage{indentfirst}
\usepackage{ytableau}
\usepackage{color}
\usepackage{colortbl}
\usepackage{amsmath}
\usepackage{tikz}
\usetikzlibrary{shapes.geometric, arrows}
\usetikzlibrary{calc}
\numberwithin{equation}{section}
\usepackage{graphicx}
\usepackage[thicklines]{cancel}
\usepackage{multicol}\multicolsep=0pt
\usepackage{cleveref}
\crefname{theorem}{Theorem}{Theorems}
\crefname{definition}{Theorem}{Definitions}
\crefname{proposition}{Theorem}{Propositions}
\crefname{alphatheorem}{Theorem}{Theorems}
\newtheorem{alphatheorem}{Theorem}

\newtheorem{alphaproposition}[alphatheorem]{Proposition}

\newtheorem{theorem}{Theorem}[section]
\newtheorem{lemma}[theorem]{Lemma}
\newtheorem{proposition}[theorem]{Proposition}
\newtheorem{definition}[theorem]{Definition}

\newtheorem{corollary}[theorem]{Corollary}
\newtheorem{example}[theorem]{Example}
\newtheorem{remark}[theorem]{Remark}

\theoremstyle{remark}
\newtheorem{rem}[theorem]{Remark}

\newcommand{\clr}{rgb:black,2;blue,2;red,0}

\tikzset{anchorbase/.style={baseline={([yshift=-0.5ex]current bounding box.center)}}}
\tikzset{wipe/.style={white,line width=4pt}}

\DeclareFontFamily{OT1}{pzc}{}
\DeclareFontShape{OT1}{pzc}{m}{it}{ <-> s*[1.2] pzcmi7t }{}
\DeclareMathAlphabet{\mathpzc}{OT1}{pzc}{m}{it}

\newcommand{\qW}{\mathpzc{QWeb}}
\newcommand{\W}{\mathpzc{Web}}

 \newcommand{\QAW}{\mathpzc{QWeb}^\bullet}
 \newcommand{\Qwb}{\mathfrak{q}\text{-}\mathbf{Web}}
  
\newcommand{\QAWC}{\mathpzc{QWeb}^{\bullet\prime}}

\newcommand{\Mat}{\text{Mat}}
\newcommand{\PMat}{\text{SParMat}}
\newcommand{\Par}{\text{Par}}
\newcommand{\SPar}{\text{SPar}}

\newcommand\glv{\mathfrak{gl}(V)}
\newcommand\C{\mathbb{C}}
\newcommand\Z{\mathbb{Z}}

\newcommand\N{\mathbb{N}}
\newcommand\kk{\Bbbk}
\newcommand\la{\lambda}
\newcommand{\Hom}{{\rm Hom}}
\newcommand{\End}{{\rm End}}
\newcommand{\rot}{\rotatebox[origin=c]{180}}
\newcommand{\arxiv}[1]{\href{http://arxiv.org/abs/#1}{\tt arXiv:\nolinkurl{#1}}}

\def\t{\mathfrak t}

\newcommand{\bdot}{ node[circle, draw, fill=\clr, thick, inner sep=0pt, minimum width=3.5pt]{}}

\newcommand{\wdot}{ node[circle, draw, color=\clr, fill=white, thick, inner sep=0pt, minimum width=3.5pt]{}}

\newcommand{\xdot}{
	\begin{tikzpicture}[baseline = 3pt, scale=0.5, color=\clr]
		
		\draw[-,thick] (0,0) to[out=up, in=down] (0,1);
		\draw(0,0.5) \bdot;
	\end{tikzpicture}
}

\newcommand{\wxdot}{
	\begin{tikzpicture}[baseline = 3pt, scale=0.5, color=\clr]
		\draw[-,thick] (0,0) to[out=up, in=down] (0,1);
		\draw(0,0.5) \wdot;
	\end{tikzpicture}
}

\newcommand{\wlambdadota}{\begin{tikzpicture}[baseline = 3pt, scale=0.5, color=\clr]
		\draw[-,line width=2pt] (0,0) to[out=up, in=down] (0,1.4);
		\draw(0,0.6) \bdot; 
		\draw (0.7,0.6) node {$\scriptstyle \omega_{\lambda}$};
		\node at (0,-.3) {$\scriptstyle a$};
\end{tikzpicture} }

\newcommand{\merge}
{\begin{tikzpicture}[baseline = -.5mm, scale=0.7,color=\clr]
		\draw[-,line width=1pt] (0.28,-.3) to (0.08,0.04);
		\draw[-,line width=1pt] (-0.12,-.3) to (0.08,0.04);
		\draw[-,line width=2pt] (0.08,.4) to (0.08,0);
		\node at (-0.22,-.4) {$\scriptstyle a$};
		\node at (0.35,-.4) {$\scriptstyle b$};\node at (0,.55){$\scriptstyle a+b$};\end{tikzpicture} }

\newcommand{\splits}
{\begin{tikzpicture}[baseline = -.5mm, scale=0.7,color=\clr]
		\draw[-,line width=2pt] (0.08,-.3) to (0.08,0.04);
		\draw[-,line width=1pt] (0.28,.4) to (0.08,0);
		\draw[-,line width=1pt] (-0.12,.4) to (0.08,0);
		\node at (-0.22,.5) {$\scriptstyle a$};
		\node at (0.36,.6) {$\scriptstyle b$};
		\node at (0.1,-.45){$\scriptstyle a+b$};
\end{tikzpicture}}

\newcommand{\dotgen}
{\begin{tikzpicture}[baseline = 3pt, scale=0.5, color=\clr]
		\draw[-,thick] (0,0) to[out=up, in=down] (0,1.4);
		\draw(0,0.6) \bdot;
		\node at (0,-.3) {$\scriptstyle 1$};
\end{tikzpicture}}

\newcommand{\odota}
{\begin{tikzpicture}[baseline = 3pt, scale=0.5, color=\clr]
		\draw[-,thick] (0,0) to[out=up, in=down] (0,1.4);
		\draw(0,0.6) \wdot;
		\node at (0,-.3) {$\scriptstyle a$};
\end{tikzpicture}}

\newcommand{\crossing}{\begin{tikzpicture}[baseline=-.5mm, scale=0.7,color=\clr]
		\draw[-,line width=1.2pt] (-0.3,-.3) to (.3,.4);
		\draw[-,line width=1.2pt] (0.3,-.3) to (-.3,.4);
		\node at (0.3,-.5) {$\scriptstyle b$};
		\node at (-0.3,-.45) {$\scriptstyle a$};
		\node at (0.3,.55) {$\scriptstyle a$};
		\node at (-0.3,.6) {$\scriptstyle b$};
\end{tikzpicture}}

\newcommand{\wkdota}{\begin{tikzpicture}[baseline = 3pt, scale=0.5, color=\clr]
     \draw[-,line width=1.2pt] (0,-.4) to[out=up, in=down] (0,1.2);
		\draw(0,0.4) \bdot; 
		\draw (0.65,0.4) node {$\scriptstyle \omega_r$};
		\node at (0,-.8) {$\scriptstyle a$};
\end{tikzpicture} }

\newcommand{\bdota}{
\begin{tikzpicture}[baseline = 3pt, scale=0.5, color=\clr]
	\draw[-,line width=1.2pt] (0,0) to[out=up, in=down] (0,1);
	\draw(0,0.6) \bdot; 
	\node at (0,-.3) {$\scriptstyle a$};
\end{tikzpicture}
}
\newcommand{\dotgenC}{
\begin{tikzpicture}[baseline = 3pt, scale=0.5, color=\clr]
	\draw[-,line width=1.2pt] (0,0) to[out=up, in=down] (0,1);
	\draw(0,0.6) \bdot; 
	\node at (0,-.3) {$\scriptstyle 1$};
\end{tikzpicture}
}

\newcommand{\wdota}{
\begin{tikzpicture}[baseline = 3pt, scale=0.5, color=\clr]
	\draw[-,line width=1.2pt] (0,0) to[out=up, in=down] (0,1);
	\draw(0,0.6) \wdot; 
	\node at (0,-.3) {$\scriptstyle a$};
\end{tikzpicture}
}

\newcommand{\genoba}{\begin{tikzpicture}[baseline = 8pt, scale=0.5, color=\clr]
\draw[-,line width=1pt] (0,0.5)to[out=up,in=down](0,1.5);
\draw (0,0.2) node{$\scriptstyle a$};
\end{tikzpicture} \;}

\newcommand{\rcircdot}{\begin{tikzpicture}[baseline = 3pt, scale=0.5, color=\clr]
		\draw[-,line width=1.5pt] (0,-.4) to[out=up, in=down] (0,1.2);
		\draw(0,0.4) \bdot; 
		\draw (0.65,0.4) node {$\scriptstyle \omega^\circ_r$};
		\node at (0,-.8) {$\scriptstyle a$};
\end{tikzpicture}}

\newcommand{\wxdota}{\begin{tikzpicture}[baseline = 3pt, scale=0.5, color=\clr]
		\draw[-,line width=1.5pt] (0,0) to[out=up, in=down] (0,1.4);
		\draw(0,0.6) \bdot; 
		\draw (0.7,0.6) node {$\scriptstyle \omega_1$};
		\node at (0,-.3) {$\scriptstyle a$};
\end{tikzpicture}}

\newcommand{\p}[1]{|#1|}

\begin{document}
	\setlength{\baselineskip}{17pt}
	\title{Affine web of type q}
	
	\author{Linliang Song}
	\address{School of Mathematical Science, Tongji University, Shanghai, 200092, China}\email{llsong@tongji.edu.cn}
	
	\author{Xingyu Wang}
	\address{School of Mathematical Science, Tongji University, Shanghai}\email{2410288@tongji.edu.cn}

	
	\begin{abstract}
   We introduce a new diagrammatic $\kk$-linear monoidal supercategory $\QAW$, the affine web supercategory of type $Q$, where $\kk$ is a commutative ring of characteristic not two. This category is the affinization of the web category of type $Q$, originally introduced by Brown and Kujawa. It serves as the type $Q$ analog of the affine web category introduced by Davidson, Kujawa, Muth and Zhu,  and independently by Wang and one of the authors. We obtain diagrammatic integral bases for the Hom-spaces of this category.
We show that $\QAW$ provides a combinatorial model for a natural monoidal supercategory of endosuperfunctors for Lie superalgebras of type $Q$. 
     \end{abstract}
	
	\maketitle
	\setcounter{tocdepth}{1}
	\tableofcontents
    
    \section{Introduction}
\subsection{Webs of type $Q$ }
The web category $\W$ (of type $A$) was first introduced by Cautis, Kamnitzer, and Morrison \cite{CKM} to provide a diagrammatic description of the category generated by the fundamental representations of $U_q(\mathfrak{sl}_n)$. A related (polynomial) version of the $\mathfrak{gl}$-web category, the infinite version, is discussed in \cite{BEEO}. Subsequently, the web category has been generalized to other types, including types $B$, $C$, and $D$ \cite{BERT21,BT24,BWu23,ST19}, as well as types $P$ and $Q$ \cite{DKM21},  \cite{BKu21}, among others.

Brown and Kujawa \cite{BKu21} introduced the category $\mathfrak{q}$-$\textbf{Web}$ of webs of type $Q$ over $\mathbb{C}$ and used it to describe the full subcategory of supermodules for the Lie superalgebra of type $Q$, consisting of tensor products of supersymmetric tensor powers of the natural supermodules.  This category is a $\C$-linear strict monoidal supercategory with generating objects denoted $\genoba$ for $a \in \mathbb{Z}_{\geq 1}$, and generating morphisms $\merge$, $\splits$, and $\wdota$, subject to the relations \eqref{webassoc}, \eqref{equ:r=0}, \eqref{doublewdot}, and \eqref{wdot-move-split and merge over C}--\eqref{wdot-rung-relation_2}.

The category $\mathfrak{q}$-$\textbf{Web}$ can be slightly generalized to the category $\qW$ (see Definition \ref{Def:webQ}), defined over any commutative ring $\kk$ of characteristic not two. The generators of $\qW$ are $\merge$, $\splits$, $\wdota$, and $\crossing$, subject to slightly different relations \eqref{webassoc}--\eqref{equ:r=0} and \eqref{wdotsmovecrossing}--\eqref{equ:onewdot}. In fact, $\mathfrak{q}$-$\textbf{Web}$ is isomorphic to $\qW$ when $\kk = \mathbb{C}$. See Theorem \ref{isom-qWeb-overc}.

\subsection{Affine web   of type $Q$}
The affinization of (diagrammatic) categories or algebras has proven to be significant in the study of their representations and related categorifications. A key feature of this process is that the new generator (often represented as a black dot) introduced in the affinization provides the Jucys-Murphy elements, which in turn lead to the required categorical actions in the categorification. For example, the polynomial generators of the (degenerate) affine Hecke algebra correspond to the Jucys-Murphy elements of the (cyclotomic) Hecke algebras.

The affinization of web categories for type $A$ was introduced by Davidson, Kujawa, Muth, and Zhu \cite{DKMZ}, and independently by Wang and one of the authors \cite{SW1}.
Notably, the affine web category \cite{SW1} plays a fundamental role as a building block for the affine Schur category, as established in \cite{SW2}. Furthermore, these categories, along with their cyclotomic quotients, provide the first diagrammatic presentations of several important algebras \cite{SW1,SW2}, including affine Schur algebras, higher-level affine Schur algebras, and the  degenerate analog of Dipper-James-Mathas' cyclotomic Schur algebras.

The goal of this paper is to extend this process of affinization to the web category of type $Q$.
We introduce a new diagrammatic $\kk$-linear monoidal supercategory $\QAW$, which we refer to as the affine web category of type $Q$. It is obtained from $\qW$ by adjoining additional (even) morphisms (referred to as black dots) $\bdota$, for $a \in \mathbb{Z}_{\geq 1}$, subject to the additional relations \eqref{dotmovecrossing}--\eqref{wdotbdot}.

\subsection{Main results}
When $\kk$ has characteristic zero (e.g., $\mathbb{C}$), the presentation of $\QAW$ simplifies significantly. In fact, it can be obtained by combining the generators and relations of $\qW$ with those of the affine Sergeev superalgebras. This observation was first made for the affine web category of type $A$ in \cite{SW1} and may extend to the affinization of web categories of other types.

We denote the $\mathbb{C}$-linear monoidal supercategory with these simplified relations by $\QAWC$ (see Definition \ref{def-QAW-overC}).

\begin{alphatheorem}  [\cref{equiv-over-C}]
  \label{thm:AffineC}
  The $\C$-linear monoidal supercategories $\QAWC$  and $\QAW$ are isomorphic. 
 \end{alphatheorem} 
 
The functor connecting $\Qwb$ (over $\mathbb{C}$) and the representations of the Lie superalgebra $\mathfrak{q}_n$ of type $Q$ (cf. \cite{BKu21}) can be naturally lifted to a functor from $\QAW$ to the monoidal supercategory of endosuperfunctors for $\mathfrak{q}_n$. In particular, the objects $a$, for $a \in \mathbb{Z}{\geq 1}$, correspond to the functor $-\otimes S^a(V)$, where $S^a(V)$ is the supersymmetric power of the natural $\mathfrak{q}_n$-supermodule $V$. This also shows that $\QAW$ naturally arises in the study of the representation theory of $\mathfrak{q}_n$.

 \begin{alphaproposition} 
[Proposition \ref{functorofaff}]
	Suppose $\kk=\C$. There is a strict monoidal functor $\mathcal F:\QAW\rightarrow \End(\mathfrak {q}_n\text{-smod})$.
\end{alphaproposition}

The above functor also plays a key role in proving  the linear independence of the expected integral diagrammatic basis for the morphism spaces in $\QAW$ (and $\qW$).
The sets of objects of $\QAW$ and $\qW$ are both given by $\Lambda_{\text{st}}=\cup_{m}\Lambda_{\text{st}}(m)$,
the set of all strict compositions. This is the same as the web category $\W$ of type $A$.
For any $\lambda,\mu\in \Lambda_{\text{st}}(m)$, the morphism space $\Hom_{\W}(\mu,\lambda)$ has a basis given by reduced chicken foot diagrams (CFD) $\text{Mat}_{\lambda,\mu}$   \cite{BEEO}. In particular,  a matrix entry represents the thickness of a leg in a reduced CFD. 

Recall a partition $\lambda$ is called strict if $\lambda_i>\lambda_{i+1}$ for all $i$. Note that this is different from the notion of strict compositions.  
Let $\Par_m$ (resp. $\SPar_m$) be the set of all  partitions (resp. strict partitions) with all parts $\le m$.  
An elementary CFD (of type $Q$) means a reduced CFD with each of its thin strands (or legs) of thickness $m$ decorated by an elementary dot packet $g_{\eta,\nu}$ defined in \eqref{dots_data_simplify}, for $(\nu,\eta)\in \Par_m\times \SPar_m$. 
The set of elementary CFDs from $\mu$ to $\la$ can be identified with  $\PMat_{\lambda,\mu}$ defined in \eqref{dottedreduced}.

 \begin{alphatheorem}  [\cref{basis-theorem,thm:basisofqweb}]
  \label{thm:BasesAffine}
    Suppose that  $\lambda,\mu\in \Lambda_{\text{st}}(m)$. 
 Then 
 \begin{enumerate}
     \item $\Hom_{\QAW}(\mu,\lambda)$ has a $\kk$-basis
    given by $\PMat_{\lambda,\mu}$.
    \item $\Hom_\qW(\mu,\lambda)$ has a $\kk$-basis given by $\PMat^{1}_{\lambda,\mu}$ (see \eqref{equ:parmatr1}).
    \item $\qW$ is a subcategory of $\QAW$.
 \end{enumerate}    
  \end{alphatheorem}
The basis for the finite web category $\qW$ coincides with the diagrammatic basis for $\Qwb$ given in \cite{Br19} when $\kk = \mathbb{C}$. In ongoing work by Davidson, Kujawa, Muth, and Zhu \cite{DKMZ}, a category closely related to $\QAW$ is introduced alongside an analogous basis theorem

The main challenge in proving the basis theorem for the affine web category  $\QAW$ is to establish the linear independence of the elementary dot packet $g_{\eta, \nu}$ in $\End_{\QAW}(m)$, for $(\eta, \nu) \in \Par_m \times \SPar_m$.

This challenge arises from a new phenomenon: when acting on the tensor product of the generic Verma module of $\mathfrak{q}_n$ and $S^m(V)$, the element $g_{\eta, \nu}$ corresponds to partially symmetric polynomials indexed by $(\eta, \nu)$, rather than symmetric polynomials indexed by partitions, as in the affine web of type $A$ case in \cite{SW1}. As a result the new phenomenon requires a very different treatment and we develop a new strategy to demonstrate the linear independence of these partially symmetric polynomials. See Appendix \S \ref{subse:partiallysym}.

It is interesting to explore the explicit connection between  cyclotomic quotients of $\QAW$  and  finite $W$-algebra of type $Q$. 
Furthermore, 
using $\QAW$ as a building block, it is natural to expect that  there is a type $Q$ analog of the 
affine Schur category \cite{SW2}  and  its cyclotomic quotients, where the latter is expected to give a diagrammatic presentation of the higher level of the queer Schur superalgebras.

After the initial release of this work on arXiv, Alistair Savage brought to our attention an interesting connection: the quantum analogue of the affine web category of type 
$Q$ proposed in this paper is expected to arise as a partial idempotent completion of the quantum isomeric supercategory introduced in \cite{Savage_2024}. This observation opens a promising direction for further research.

\subsection{The organization}
The paper is organized as follows. In Section 2, we introduce the web category $\qW$ of type $Q$ and its affinization $\QAW$ over any commutative ring $\kk$. Many implicit relations in $\QAW$ are also established for later use.

In Section 3, for $\kk$ a field of characteristic zero, we show that $\qW$ is isomorphic to the category $\qW'$ introduced in \cite{BKu21} and provide a simplified presentation of $\QAW$.

In Section 4, we present the functor from $\QAW$ to $\End(\mathfrak{q}_n\text{-smod})$ (over $\C$) and establish the basis theorem for both $\QAW$ and $\qW$.

Finally, the appendix includes additional $\equiv$-relations in $\QAW$ and a useful result on the linear independence of certain partially symmetric polynomials. These results are used to prove both the spanning and linear independence aspects of the basis theorem for $\QAW$.

\noindent {\bf Acknowledgement.} 
LS is partially supported by NSFC (Grant No. 12071346). We thank Weiqiang Wang for encouraging us to study this research topic. We thank Alistair Savage for highlighting the potential link to \cite{Savage_2024}  and for pointing out the typos of the first version of the paper.
\section{The affine web category of type $Q$}
\label{sec:q-AWeb}
 Let $\kk$ be a commutative ring with $2$ is invertible throughout this paper. All categories and functors will be $\kk$-linear without further mention.
\subsection{The  web category of type $A$}
We first recall the definition of the web category $\W$ (of type $A$).
 \begin{definition}\cite{CKM,BEEO}
 \label{def-web}
    The web category $\W$ is the strict monoidal category generated by objects $a\in \mathbb Z_{\ge1}$. The object $a$ and its identity morphism  will be drawn as a vertical strand labeled by $a$:
$\genoba$.
The generating morphisms are the merges, splits, and (thick) crossings depicted as
\begin{align}
\label{merge+split+crossing}
\merge
			&:(a,b) \rightarrow (a+b),&
\splits
			&:(a+b)\rightarrow (a,b),&
\crossing
			&:(a,b) \rightarrow (b,a),
\end{align}
for $a,b \in \Z_{\ge 1}$,
		subject to the following relations \eqref{webassoc}--\eqref{equ:r=0}, for $a,b,c,d \in \Z_{\ge 1}$ with $d-a=c-b$:
\begin{align}
\label{webassoc}
. 
\end{align} 
\end{definition}

	We list some well-known implied relations in $\W$ (see \cite{CKM} and \cite[(4.30)--(4.33)]{BEEO}),  which shall be used frequently:
\begin{align}
\label{swallows}
%
		\:.\end{align}

\subsection{The web category of type $Q$}
We extend Brown-Kujawa's web category of type $Q$ over $\C$ \cite{BKu21} to  any commutative ring  $\kk$. While the two constructions differ slightly, we will show that they are isomorphic  when $\kk=\C$ in the next section.
\begin{definition}
\label{Def:webQ}
The web category $\qW$ of type $Q$ is the strict monoidal supercategory obtained from  $\W$ by adjoining additional generating morphisms (called white dots) $\wdota$, $a\in \Z_{\ge1}$,
, a\in \mathbb{Z}_{\ge 1}, 
subject to the following additional relations \eqref{wdotsmovecrossing}--\eqref{equ:onewdot} for $a,b\in\mathbb{Z}_{\ge 1}$:
\begin{equation}
\label{wdotsmovecrossing}
%
,
\end{equation}
		
\begin{align}
\label{wdot-move-split and merge}
	%
.
\end{equation}
The $\mathbb{Z}_2$-grading is given by declaring $	
%
  
$ to have parity $\bar{1}$ and other generating morphisms to have  parity $\bar{0}$.
\end{definition}

\subsection{The affine web category of type $Q$} 	
\begin{definition}
\label{def-affine-web}
The affine web category $\QAW$ of type $Q$ is the strict monoidal supercategory obtained from $\qW$ by adjoining (even) additional morphisms (called black dots) $\bdota$, $a\in\Z_{\ge1}$,
	subject to  the following additional relations \eqref{dotmovecrossing}--\eqref{wdotbdot}, for $a,b \in \Z_{\ge 1}$ :

\begin{equation}
\label{dotmovecrossing}
\begin{split}	
%
.
\end{equation}

\end{definition}

Note that the $a$-fold merges and splits in \eqref{intergralballon} are defined by using \eqref{webassoc} via compositions of 2-fold merges and splits. For example, the following are 3-fold merge and split.
\begin{align}
%
\:.
\end{align}

	We define the following morphisms in $\QAW$
\begin{equation}
\label{equ:r=0andr=a}
%
 \qquad\quad (1\le r\le a).
\end{equation}

   Throughout this paper, we adopt the convention that
\[ 
%
.
\]
	 The following relations can be derived from \eqref{equ:onewdot} and \eqref{equ:r=0andr=a}:
\begin{equation}
%
.
\end{equation}
To simplify notation, for any $0\le r\le a-1$
we define 
\begin{equation}
\label{bw-dot}    
\rcircdot
:=
%
. 
\end{equation}
 Hence, we also adopt the convention that  $%
=0$ if $r<0$ or $r\ge a$.
 We will also use $\omega_{a,r}$   (resp. $\omega^\circ_{a,r}$) to abbreviate $\wkdota$ (resp. $\rcircdot$).
 
For any $ m\in \mathbb N$, a composition $\mu=(\mu_1,\mu_2,\ldots,\mu_n)$ of $m$ is a sequence of non-negative integers such that $\sum_{i\ge 1} \mu_i=m$. The length $l(\mu)$ of $\mu$ is the total number $n$ of parts. Let $\Lambda(m)$ denote the set of all compositions of $m$, and $\Lambda=\bigcup_m \Lambda(m)$ denote the set of all compositions. A composition $\mu$ is called strict if  $\mu_i>0$ for all $i\ge 1$.  
We denote by  $\Lambda_{\text{st}}$ (and respectively, $\Lambda_{\text{st}}(m)$) the set of all strict compositions (and respectively, of $m$)

The set of objects in $\W $ ($\qW$ and $\QAW$) is $\Lambda_{\text{st}}$
with the empty sequence as the unit object. By definition, any morphisms in $\QAW$ are a linear combination of diagrams obtained from the tensor product and compositions of generators together with identity morphism. Such a diagram is called a dotted web diagram. Moreover, for any $\lambda,\mu\in \Lambda_{\text{st}}$ and $\mathcal C\in\{\qW,\QAW\}$, we have $\Hom_{\mathcal C}(\lambda,\mu)\neq0$ only if $\lambda,\mu\in \Lambda_{st}(m)$ for some $m$.

\subsection{More implied relations}	 
 Recall the  \emph{super-interchange law} in any supermonoidal category:
\begin{equation}
\label{super-interchange}
\begin{tikzpicture}[baseline = 19pt,scale=0.4,color=\clr,inner sep=0pt, minimum width=11pt]
   		\draw[-,thick] (0,0) to (0,3);
   		\draw[-,thick] (2,0) to (2,3);
   		\draw (0,2) node[circle,draw,thick,fill=white]{$f$};
   		\draw (2,1) node[circle,draw,thick,fill=white]{$g$};
\end{tikzpicture}
   	~=~
\begin{tikzpicture}[baseline = 19pt,scale=0.4,color=\clr,inner sep=0pt, minimum width=11pt]
   		\draw[-,thick] (0,0) to (0,3);
   		\draw[-,thick] (2,0) to (2,3);
   		\draw (0,1.5) node[circle,draw,thick,fill=white]{$f$};
   		\draw (2,1.5) node[circle,draw,thick,fill=white]{$g$};
\end{tikzpicture}
   	~=(-1)^{\p{f}\p{g}}~
\begin{tikzpicture}[baseline = 19pt,scale=0.4,color=\clr,inner sep=0pt, minimum width=11pt]
   		\draw[-,thick] (0,0) to (0,3);
   		\draw[-,thick] (2,0) to (2,3);
   		\draw (0,1) node[circle,draw,thick,fill=white]{$f$};
   		\draw (2,2) node[circle,draw,thick,fill=white]{$g$};
\end{tikzpicture}
   	~.
\end{equation}
It is easy to check that there exists an automorphism $\div:\QAW\rightarrow {\QAW}^{\text{op}}$ that fixes objects and rotates morphism diagrams by $180^\circ$ around a horizontal axis. For example, by \eqref{super-interchange} we have
$$
\begin{tikzpicture}[baseline = 3pt, scale=0.5, color=\clr]
	\draw[-,line width=1.5pt] (0,-.5) to (0,1.3);
        \draw[-,line width=1.5pt] (.8,-.5) to (.8,1.3);
	\draw(0,0.4) \wdot;
        \draw(.8,0.4) \wdot;
\end{tikzpicture}^\div
=
\begin{tikzpicture}[baseline = 3pt, scale=0.5, color=\clr]
	\draw[-,line width=1.5pt] (0,-.5) to (0,1.3);
        \draw[-,line width=1.5pt] (.8,-.5) to (.8,1.3);
	\draw(0,.6) \wdot;
        \draw(.8,0.2) \wdot;
\end{tikzpicture}^\div
=
\begin{tikzpicture}[baseline = 3pt, scale=0.5, color=\clr]
	\draw[-,line width=1.5pt] (0,-.5) to (0,1.3);
        \draw[-,line width=1.5pt] (.8,-.5) to (.8,1.3);
	\draw(0,0.2) \wdot;
        \draw(.8,0.6) \wdot;
\end{tikzpicture}
=-\:
\begin{tikzpicture}[baseline = 3pt, scale=0.5, color=\clr]
	\draw[-,line width=1.5pt] (0,-.5) to (0,1.3);
        \draw[-,line width=1.5pt] (.8,-.5) to (.8,1.3);
	\draw(0,0.4) \wdot;
        \draw(.8,0.4) \wdot;
\end{tikzpicture}.
$$
and hence \eqref{dotmovecrossing} is preserved by $\div$. We will frequently use this $\div$ later to derive many equalities by symmetry. 

\begin{lemma}
       The following relations hold in $\QAW$:
\begin{equation}
\label{dots-ball}
\begin{tikzpicture}[baseline = 1.5mm, scale=0.8, color=\clr]
	\draw[-,line width=1.5pt] (0,0.7) to(0,.3);
	\draw[-,line width=1.5pt] (0,0.3) to(0,-.4);
	\draw(0,0.4) \wdot; 
	\draw(0,-0.1) \bdot; 
	\draw (-0.4,-0.1) node {$\scriptstyle \omega^\circ_a$};
	\draw (-0,-0.6) node {$\scriptstyle k$};
\end{tikzpicture}
			=
\begin{tikzpicture}[baseline = 1.5mm, scale=0.8, color=\clr]
	\draw[-,line width=1.5pt] (0,0.7) to(0,.5);
	\draw[-,line width=1.5pt] (0,0.5) to(0,-.4);
	\draw(0,-0.1) \wdot; 
	\draw(0,0.4) \bdot; 
	\draw (-0.4,0.4) node {$\scriptstyle \omega^\circ_a$};
	\draw (-0,-0.6) node {$\scriptstyle k$};
\end{tikzpicture}\: , 
         \quad
\begin{tikzpicture}[baseline = 1.5mm, scale=.8, color=\clr]
	\draw[-,line width=1.5pt] (0,0.7) to(0,.3);
	\draw[-,line width=1.5pt] (0,0.3) to(0,-.4);
	\draw(0,0.4) \bdot; 
	\draw(0,-0.1) \bdot; 
	\draw (-0.4,-0.1) node {$\scriptstyle \omega^\circ_a$};
	\draw (-0,-0.6) node {$\scriptstyle k$};
\end{tikzpicture}
			= -
\begin{tikzpicture}[baseline = 1.5mm, scale=0.8, color=\clr]
	\draw[-,line width=1.5pt] (0,0.7) to(0,.5);
	\draw[-,line width=1.5pt] (0,-.4) to(0,.5);
	\draw(0,-0.1) \bdot; 
	\draw(0,0.4) \bdot; 
	\draw (-0.4,0.4) node {$\scriptstyle \omega^\circ_a$};
	\draw (-0,-0.6) node {$\scriptstyle k$};
\end{tikzpicture}\: , 
            \quad
\begin{tikzpicture}[baseline = 1.5mm, scale=0.8, color=\clr]
	\draw[-,line width=1.5pt] (0,0.7) to(0,-.4);
	\draw(0,-0.1) \bdot; 
	\draw (-.3,-0.1) node {$\scriptstyle \omega_a$};
	\draw(0,0.4) \wdot; 
	\draw (-0,-0.6) node {$\scriptstyle k$};
\end{tikzpicture}
			=-
\begin{tikzpicture}[baseline = 1.5mm, scale=0.8, color=\clr]
	\draw[-,line width=1.5pt] (0,0.7) to(0,-.4);
	\draw(0,-0.1) \wdot; 
	\draw (-.3,0.4) node {$\scriptstyle \omega_a$};
	\draw(0,0.4) \bdot; 
	\draw (-0,-0.6) node {$\scriptstyle k$};
\end{tikzpicture}
			+2\:
\begin{tikzpicture}[baseline = 1.5mm, scale=0.8, color=\clr]
	\draw[-,line width=1.5pt] (0,0.7) to(0,.5);
	\draw[-,line width=1.5pt] (0,-.4) to(0,.5);
	\draw(0,0.15) \bdot; 
	\draw (-0.4,0.15) node {$\scriptstyle \omega^\circ_a$};
	\draw (-0,-0.6) node {$\scriptstyle k$};
\end{tikzpicture},
\quad
\begin{tikzpicture}[baseline = 1.5mm, scale=0.8, color=\clr]
	\draw[-,line width=1.5pt] (0,0.7) to(0,-.4);
	\draw(0,-0.1) \bdot; 
	\draw (-.3,-0.1) node {$\scriptstyle \omega_a$};
	\draw(0,0.4) \wdot; 
	\draw (-0,-0.6) node {$\scriptstyle k$};
\end{tikzpicture}
	=
\begin{tikzpicture}[baseline = 1.5mm, scale=0.8, color=\clr]
	\draw[-,line width=1.5pt] (0,0.7) to(0,-.4);
	\draw(0,-0.1) \wdot; 
	\draw (-.3,0.4) node {$\scriptstyle \omega_a$};
	\draw(0,0.4) \bdot; 
	\draw (-0,-0.6) node {$\scriptstyle k$};
\end{tikzpicture}
			+2\:
\begin{tikzpicture}[baseline = 1.5mm, scale=0.7, color=\clr]
	\draw[-,line width=1.5pt] (0,0.7) to (0,.9);
	\draw[-,line width=1.5pt] (0,-0.2) to (0,-.4);
	\draw[-,line width=1pt] (0,-.2) to[out=180, in=180] (0,.7);
	\draw[-,line width=1pt] (0,-.2) to[out=0, in=0] (0,.7);
	\draw(-.2,.5) \wdot; 
	\draw(-.2,0) \bdot;  
	\node at (-.55,0) {$\scriptstyle a $};
	\node at (0,-.6) {$\scriptstyle k$};
\end{tikzpicture}
\end{equation}
\end{lemma}

\begin{proof}
	    This follows from \eqref{wdot-move-split and merge}, \eqref{bdotmove},\eqref{wdotbdot} and \eqref{super-interchange}.
\end{proof}

\begin{lemma}
\label{killwdot}
 	The following relation holds in $\QAW$:
 \begin{equation}
  \label{equ:twowhite=0}
 		\begin{tikzpicture}[baseline = -.5mm, scale=1.2, color=\clr]
 			\draw[-,line width=1.5pt](0,-.45) to (0,-0.25);
 			\draw[-,line width=1pt] (0,-.25) to [out=180,in=-180] (0,.25);
 			\draw[-,line width=1pt] (0,-.25) to [out=0,in=0] (0,.25);
 			\draw[-,line width=1.5pt](0,.45) to (0,0.25);
 			\draw (-0.15,0) \wdot;
 			\draw (0.15,0) \wdot;
 			\draw (-0.35,0) node{$\scriptstyle a$};
 			\node at (0.35,0) {$\scriptstyle b$};
 		\end{tikzpicture}
 		=0.    	
 \end{equation}
\end{lemma}
\begin{proof}
       First, note that
\begin{equation}
\label{equ:fourwdot}
\begin{tikzpicture}[baseline = 7.5pt,scale=.7, color=\clr]
	\draw[-,line width=1pt] (-.2,1.1) to (.08,.64) to (.34,1.1) to (.08,.64);
	\draw[-,line width=1.5pt] (.08,.66) to (.08,0.08);
	\draw[-,line width=1pt] (-.2,-.4) to (.08,0.1) to (.34,-.4) to (.08,0.1);
	\draw (-.06,.85) \wdot;
	\draw (.21,.85) \wdot;
	\draw (-.06,-.15) \wdot;
	\draw (.21,.-.15) \wdot;
	\node at (-0.2,-.6) {$\scriptstyle 1$};
	\node at (0.34,-.6) {$\scriptstyle 1$};
\end{tikzpicture}
		\overset{\eqref{mergesplit}}{=}          
\begin{tikzpicture}[baseline = 7.5pt,scale=.7, color=\clr]
	\draw[-,line width=1pt] (-.2,.7) to (.34,0);
	\draw[-,line width=1pt] (-.2,0) to (.34,0.7);
	\draw[-,line width=1pt] (-.2,0) to (-.2,-.4);
	\draw[-,line width=1pt] (-.2,.7) to (-.2,1.1);
	\draw[-,line width=1pt] (.34,0) to (.34,-.4);
	\draw[-,line width=1pt] (.34,.7) to (.34,1.1);
	\draw (-.2,-.2) \wdot;
	\draw (.34,-.2) \wdot;
	\draw (-.2,.9) \wdot;
	\draw (.34,.9) \wdot;
	\node at (-0.2,-.6) {$\scriptstyle 1$};
	\node at (0.34,-.6) {$\scriptstyle 1$};
\end{tikzpicture}
		+
\begin{tikzpicture}[baseline = 7.5pt,scale=.7, color=\clr]
	\draw[-,line width=1pt] (-.2,1.1) to (-.2,-.4);
	\draw[-,line width=1pt] (.34,1.1) to (.34,-.4); 			
	\draw (-.2,0) \wdot;
	\draw (.34,0) \wdot;
	\draw (-.2,.7) \wdot;
	\draw (.34,.7) \wdot;
	\node at (-0.2,-.6) {$\scriptstyle 1$};
	\node at (0.34,-.6) {$\scriptstyle 1$};
\end{tikzpicture}
		\overset{\eqref{wdotsmovecrossing}}{\underset{\eqref{super-interchange}}{=}}
\begin{tikzpicture}[baseline = 7.5pt,scale=.7, color=\clr]
	\draw[-,line width=1pt] (-.2,.3) to (.34,-.4);
	\draw[-,line width=1pt] (-.2,-.4) to (.34,0.3);
	\draw[-,line width=1pt] (-.2,.3) to (-.2,1.1);
	\draw[-,line width=1pt] (.34,.3) to (.34,1.1);
	\draw (-.2,.4) \wdot;
	\draw (.34,.6) \wdot;
	\draw (-.2,.9) \wdot;
	\draw (.34,.8) \wdot;
	\node at (-0.2,-.6) {$\scriptstyle 1$};
	\node at (0.34,-.6) {$\scriptstyle 1$};
\end{tikzpicture}
		-
\begin{tikzpicture}[baseline = 7.5pt,scale=.7, color=\clr]
	\draw[-,line width=1pt] (-.2,1.1) to (-.2,-.4);
	\draw[-,line width=1pt] (.34,1.1) to (.34,-.4); 			
	\draw (-.2,0.6) \wdot;
	\draw (.34,-.1) \wdot;
	\draw (-.2,.9) \wdot;
	\draw (.34,.2) \wdot;
	\node at (-0.2,-.6) {$\scriptstyle 1$};
	\node at (0.34,-.6) {$\scriptstyle 1$};
\end{tikzpicture}
		\overset{\eqref{doublewdot}}=
\begin{tikzpicture}[baseline = 7.5pt,scale=1, color=\clr]
	\draw[-,line width=1pt] (-.2,.7) to (.34,0);
	\draw[-,line width=1pt] (-.2,0) to (.34,0.7);
	\node at (-0.2,-.2) {$\scriptstyle 1$};
	\node at (0.34,-.2) {$\scriptstyle 1$};
\end{tikzpicture}
		-
\begin{tikzpicture}[baseline = 7.5pt,scale=1, color=\clr]
	\draw[-,line width=1pt] (-.2,0) to (-.2,.7);
	\draw[-,line width=1pt] (.25,0) to (.25,.7); 	\node at (-0.2,-.2) {$\scriptstyle 1$};
	\node at (0.25,-.2) {$\scriptstyle 1$};
\end{tikzpicture}
		\overset{\eqref{mergesplit}}=   
\begin{tikzpicture}[baseline = 7.5pt,scale=.7, color=\clr]
	\draw[-,line width=1pt] (-.2,1.1) to (.08,.64) to (.34,1.1) to (.08,.64);
	\draw[-,line width=1.5pt] (.08,.66) to (.08,0.08);
	\draw[-,line width=1pt] (-.2,-.4) to (.08,0.1) to (.34,-.4) to (.08,0.1);
	\node at (-0.2,-.6) {$\scriptstyle 1$};
	\node at (0.34,-.6) {$\scriptstyle 1$};
\end{tikzpicture}
		-2
\begin{tikzpicture}[baseline = 7.5pt,scale=1, color=\clr]
	\draw[-,line width=1pt] (-.2,0) to (-.2,.7);
	\draw[-,line width=1pt] (.25,0) to (.25,.7);	
        \node at (-0.2,-.2) {$\scriptstyle 1$};
	\node at (0.25,-.2) {$\scriptstyle 1$};
\end{tikzpicture}\:.		
\end{equation}
	Suppose $a=b=1$. Then
\[2\:\:
\begin{tikzpicture}[baseline = -.5mm, scale=1.2, color=\clr]
		\draw[-,line width=1.5pt](0,-.45) to (0,-0.25);
		\draw[-,line width=1pt] (0,-.25) to [out=180,in=-180] (0,.25);
		\draw[-,line width=1pt] (0,-.25) to [out=0,in=0] (0,.25);
		\draw[-,line width=1.5pt](0,.45) to (0,0.25);
		\draw (-0.15,0) \wdot;
		\draw (0.15,0) \wdot;
		\node at (0,-.6) {$\scriptstyle 2$};
\end{tikzpicture}
	\overset{\eqref{equ:r=0}}{=}
\begin{tikzpicture}[baseline = -.5mm, scale=1.2, color=\clr]
		\draw[-,line width=1.5pt](0,.6) to (0,0.7);
		\draw[-,line width=1pt] (0,.1) to [out=180,in=-180] (0,.6);
		\draw[-,line width=1pt] (0,.1) to [out=0,in=0] (0,.6);
		\draw[-,line width=1.5pt](0,0) to (0,0.1);
		\draw[-,line width=1pt] (0,-.5) to [out=180,in=-180] (0,0);
		\draw[-,line width=1pt] (0,-.5) to [out=0,in=0] (0,0);
		\draw (-0.15,-.25) \wdot;
		\draw (0.15,-.25) \wdot;
		\draw[-,line width=1.5pt](0,-.5) to (0,-0.6);
		\node at (0,-.75) {$\scriptstyle 2$};
\end{tikzpicture}
	\overset{\eqref{equ:fourwdot}}{=}
\begin{tikzpicture}[baseline = -.5mm, scale=1.2, color=\clr]
		\draw[-,line width=1.5pt](0,.6) to (0,0.7);
		\draw[-,line width=1pt] (0,.1) to [out=180,in=-180] (0,.6);
		\draw[-,line width=1pt] (0,.1) to [out=0,in=0] (0,.6);
		\draw[-,line width=1.5pt](0,0) to (0,0.1);
		\draw[-,line width=1pt] (0,-.5) to [out=180,in=-180] (0,0);
		\draw[-,line width=1pt] (0,-.5) to [out=0,in=0] (0,0);
		\draw (-0.15,.35) \wdot;
		\draw (0.15,.35) \wdot;
		\draw (-0.12,-.1) \wdot;
		\draw (0.12,-.1) \wdot;
		\draw (-0.12,-.4) \wdot;
		\draw (0.12,-.4) \wdot;
		\draw[-,line width=1.5pt](0,-.5) to (0,-0.6);
		\node at (0,-.75) {$\scriptstyle 2$};
\end{tikzpicture}
	+2\:
\begin{tikzpicture}[baseline = -.5mm, scale=1.2, color=\clr]
		\draw[-,line width=1.5pt](0,-.45) to (0,-0.25);
		\draw[-,line width=1pt] (0,-.25) to [out=180,in=-180] (0,.25);
		\draw[-,line width=1pt] (0,-.25) to [out=0,in=0] (0,.25);
		\draw[-,line width=1.5pt](0,.45) to (0,0.25);
		\draw (-0.15,0) \wdot;
		\draw (0.15,0) \wdot;
		\node at (0,-.6) {$\scriptstyle 2$};
\end{tikzpicture}
	\overset{\eqref{super-interchange}}{=}-
\begin{tikzpicture}[baseline = -.5mm, scale=1.2, color=\clr]
		\draw[-,line width=1.5pt](0,.6) to (0,0.7);
		\draw[-,line width=1pt] (0,.1) to [out=180,in=-180] (0,.6);
		\draw[-,line width=1pt] (0,.1) to [out=0,in=0] (0,.6);
		\draw[-,line width=1.5pt](0,0) to (0,0.1);
		\draw[-,line width=1pt] (0,-.7) to [out=180,in=-180] (0,0);
		\draw[-,line width=1pt] (0,-.7) to [out=0,in=0] (0,0);
		\draw (-0.15,.35) \wdot;
		\draw (0.15,.35) \wdot;
		\draw (-0.14,-.1) \wdot;
		\draw (-0.18,-.3) \wdot;
		\draw (0.18,-.4) \wdot;
		\draw (0.14,-.6) \wdot;
		\draw[-,line width=1.5pt](0,-.7) to (0,-0.8);
		\node at (0,-.95) {$\scriptstyle 2$};
\end{tikzpicture}
	+2\:
\begin{tikzpicture}[baseline = -.5mm, scale=1.2, color=\clr]
		\draw[-,line width=1.5pt](0,-.45) to (0,-0.25);
		\draw[-,line width=1pt] (0,-.25) to [out=180,in=-180] (0,.25);
		\draw[-,line width=1pt] (0,-.25) to [out=0,in=0] (0,.25);
		\draw[-,line width=1.5pt](0,.45) to (0,0.25);
		\draw (-0.15,0) \wdot;
		\draw (0.15,0) \wdot;
		\node at (0,-.6) {$\scriptstyle 2$};
\end{tikzpicture}
	\overset{\eqref{equ:r=0}}{\underset{\eqref{doublewdot}}{=}}
	0 \:,
\]
 yields \eqref{equ:twowhite=0} since $2$ is invertible.
	For $a+b>2$, we have
\[
\begin{tikzpicture}[baseline = -.5mm, scale=1, color=\clr]
		\draw[-,line width=1.5pt](0,-.45) to (0,-0.25);
		\draw[-,line width=1pt] (0,-.25) to [out=180,in=-180] (0,.25);
		\draw[-,line width=1pt] (0,-.25) to [out=0,in=0] (0,.25);
		\draw[-,line width=1.5pt](0,.45) to (0,0.25);
		\draw (-0.15,0) \wdot;
		\draw (0.15,0) \wdot;
		\draw (-0.35,0) node{$\scriptstyle a$};
		\node at (0.35,0) {$\scriptstyle b$};
\end{tikzpicture}
	\overset{\eqref{equ:onewdot}}{=}
\begin{tikzpicture}[baseline = -.5mm, scale=1, color=\clr]
		\draw[-,line width=1.2pt](0,-.35) to (0,-0.25);
		\draw[-,line width=1pt] (0,-.25) to [out=180,in=-180] (0,.25);
		\draw[-,line width=1pt] (0,-.25) to [out=0,in=0] (0,.25);
		\draw[-,line width=1.2pt](0,.35) to (0,0.25);
		\draw (0.15,0) \wdot;
		\draw (-0.45,0) node{$\scriptstyle a-1$};
		\draw[-,line width=1.2pt](0.6,-.35) to (0.6,-0.25);
		\draw[-,line width=1pt] (0.6,-.25) to [out=180,in=-180] (0.6,.25);
		\draw[-,line width=1pt] (0.6,-.25) to [out=0,in=0] (0.6,.25);
		\draw[-,line width=1.2pt](0.6,.35) to (0.6,0.25);
		\draw (0.45,0) \wdot;
		\draw (1.05,0) node{$\scriptstyle b-1$};
		\draw[-,line width=1pt] (0,.35) to [out=90,in=-180] (0.3,.5);
		\draw[-,line width=1pt] (0.6,.35) to [out=90,in=0] (0.3,.5);
		\draw[-,line width=1pt] (0,-.35) to [out=-90,in=-180] (0.3,-.5);
		\draw[-,line width=1pt] (0.6,-.35) to [out=-90,in=0] (0.3,-.5);
		\draw[-,line width=1.2pt](0.3,-.5) to (0.3,-0.6);
		\draw[-,line width=1.2pt](0.3,.5) to (0.3,0.6);
\end{tikzpicture}
	\overset{\eqref{webassoc}}{=}
\begin{tikzpicture}[baseline = -.5mm, scale=.8, color=\clr]
		\draw[-,line width=1.2pt](0,-.35) to (0,-0.25);
		\draw[-,line width=1pt] (0,-.25) to [out=180,in=-180] (0,.25);
		\draw[-,line width=1pt] (0,-.25) to [out=0,in=0] (0,.25);
		\draw[-,line width=1.2pt](0,.35) to (0,0.25);
		\draw (-0.15,0) \wdot;
		\draw (0.15,0) \wdot;
		\draw[-,line width=1pt](0.4,-.35) to (0.4,0.35);
		\draw[-,line width=1pt] (0,.35) to [out=90,in=-180] (0.2,.5);
		\draw[-,line width=1pt] (0.4,.35) to [out=90,in=0] (0.2,.5);
		\draw[-,line width=1pt] (0,-.35) to [out=-90,in=180] (0.2,-.5);
		\draw[-,line width=1pt] (0.4,-.35) to [out=-90,in=0] (0.2,-.5);
		\draw[-,line width=1.2pt](0.2,-.5) to (0.2,-0.6);
		\draw[-,line width=1.2pt](0.2,.5) to (0.2,0.6);
		\draw[-,line width=1pt](-0.3,-.6) to (-0.3,0.6);
		\draw[-,line width=1pt] (-0.3,.6) to [out=90,in=180] (-0.05,.75);
		\draw[-,line width=1pt] (0.2,.6) to [out=90,in=0] (-0.05,.75);
		\draw[-,line width=1pt] (-.3,-.6) to [out=-90,in=180] (-0.05,-.75);
		\draw[-,line width=1pt] (0.2,-.6) to [out=-90,in=0] (-0.05,-.75);
		\draw[-,line width=1.2pt](0.05,.75) to (0.05,0.85);
		\draw[-,line width=1.2pt](0.05,-.75) to (0.05,-0.85);
\end{tikzpicture}
	=0 .\]
 The lemma is proved.
\end{proof}
	By assigning degree $a$ to $\bdota$ and $0$ to other generating morphisms, we define the degree of any dotted web diagram $D$ 
    as the sum of degrees of its local diagrams of $D$. 
    For any $\lambda,\mu\in\Lambda_{\text{st}}(m)$,  let $\Hom_{\QAW}(\lambda,\mu)_{\le k}$ (and respectively, $\Hom_{\QAW}(\lambda,\mu)_{< k}$) be the $\kk$-span of all dotted web diagram of type $\lambda\rightarrow \mu$ with degree $\le k$ (and respectively, $<k$). For any two dotted web diagrams $D$ and $D'$  in  $\Hom_{\QAW}(\lambda,\mu)_{\le k}$, we write 
\begin{align}  
\label{modLOT}
                	D\equiv D' \quad \text{ if } D=D'  \mod \Hom_{\QAW}(\lambda,\mu)_{< k}.
\end{align}

\begin{lemma}
\label{dotmovefreely1}
		The following $\equiv$-relations hold in $\QAW$
\begin{equation}
\label{dotmovecrossing+high}
.
\end{equation}
\end{lemma}

\begin{proof}
It suffices to prove half of the equalities by the symmetry $\div$. For the first equalities in \eqref{dotmovecrossing+high} and \eqref{ballcross}, we have
$$
%
,
\end{aligned}
\end{equation}
where $r_{c,d,t}=\scalebox{1}{$t!\binom{a-c}{t}\binom{b-d}{t}$}, r'_{c,d,t}=\scalebox{1}{$(t-1)!\binom{a-c-1}{t-1}\binom{b-d-1}{t-1}$} $.
Thus the first one in \eqref{dotmovemerge+high} holds. By symmetry, we also have:

\begin{equation}
\label{dotsmovesplits1}
%
	\overset{\eqref{wdotsmovecrossing}}{\underset{\eqref{wdot-move-split and merge}}{=}}\sum_{k\ge 0}
%
.
\end{align*}
This lemma is proved.
\end{proof}

\begin{lemma}
\label{dotmovefreely2}        
		The following $\equiv$- relations hold in $\QAW$
\begin{equation}
\label{extrarelation}
%
    ).
\end{equation}
\end{lemma}
\begin{proof}For the first one in \eqref{extrarelation}, we have 
\begin{align*}
%
,
\end{align*}
where $r_{c,d,t}=\scalebox{1}{$t!\binom{a-c}{t}\binom{k-a-d}{t}$}, r'_{c,d,t}=\scalebox{1}{$(t-1)!\binom{a-c-1}{t-1}\binom{k-a-d-1}{t-1}$} $.
 For the second  identity, we have
\begin{align*}
%
\big)
, \text{ by repeating the previous step.}
\end{align*}
      
This lemma is proved.
	\end{proof}

\section{affine web category of type $Q$ over $\mathbb{C}$}
	\label{affineweboverC}	

In this section, $\kk$
 denotes an arbitrary field of characteristic zero, such as $\C$. Under this assumption on $\kk$, we will
show that $\qW$ is isomorphic to the web category of type $Q$ introduced by Brown and Kujawa \cite{BKu21} and 
significantly simplify the presentation of $\QAW$.

\subsection{$\qW$ over $\C$}
We recall the web category $\Qwb$ of type $Q$ over $\C$ \cite{BKu21}.
\begin{definition}
\label{defofbku}
\cite{BKu21}
 Suppose $\kk=\C$.   The web category $\mathfrak{q}\text{-}\mathbf{Web}$ of type $Q$ is the strict monoidal supercategory with generating  objects $a\in\mathbb{Z}_{\ge 1}$, and  generating morphisms 
\begin{align}
\label{merge+split+wdots over C}
\merge 
    		&:(a,b) \rightarrow (a+b),&
\splits
    &:(a+b)\rightarrow (a,b),&
\wdota
    &:(a)\rightarrow (a),      		
\end{align}
 subject to the relations \eqref{webassoc},\eqref{equ:r=0},\eqref{doublewdot} and   \eqref{wdot-move-split and merge over C}--\eqref{wdot-rung-relation_2}, for $a,b\in\mathbb{Z}_{\ge 1}$:
\begin{align} 
\label{wdot-move-split and merge over C}
\begin{tikzpicture}[baseline = -.5mm,scale=0.8,color=\clr]
    			\draw[-,line width=2pt] (0.08,-.5) to (0.08,0.04);
    			\draw[-,line width=1pt] (0.34,.5) to (0.08,0);
    			\draw[-,line width=1pt] (-0.2,.5) to (0.08,0);
    			\node at (-0.22,.7) {$\scriptstyle 1$};
    			\node at (0.36,.7) {$\scriptstyle b$};
    			\draw (0.08,-.2) \wdot;
\end{tikzpicture} 
    		=
\begin{tikzpicture}[baseline = -.5mm,scale=0.8,color=\clr]
    			\draw[-,line width=2pt] (0.08,-.5) to (0.08,0.04);
    			\draw[-,line width=1pt] (0.34,.5) to (0.08,0);
    			\draw[-,line width=1pt] (-0.2,.5) to (0.08,0);
    			\node at (-0.22,.7) {$\scriptstyle 1$};
    			\node at (0.36,.7) {$\scriptstyle b$};
    			\draw (-.05,.24) \wdot;
\end{tikzpicture}
    		+
\begin{tikzpicture}[baseline = -.5mm,scale=0.8,color=\clr]
    			\draw[-,line width=2pt] (0.08,-.5) to (0.08,0.04);
    			\draw[-,line width=1pt] (0.34,.5) to (0.08,0);
    			\draw[-,line width=1pt] (-0.2,.5) to (0.08,0);
    			\node at (-0.22,.7) {$\scriptstyle 1$};
    			\node at (0.36,.7) {$\scriptstyle b$};
    			\draw (.22,.24) \wdot;
\end{tikzpicture}, 
    \quad 
\begin{tikzpicture}[baseline = -.5mm, scale=0.8, color=\clr]
    			\draw[-,line width=1pt] (0.3,-.5) to (0.08,0.04);
    			\draw[-,line width=1pt] (-0.2,-.5) to (0.08,0.04);
    			\draw[-,line width=2pt] (0.08,.6) to (0.08,0);
    			\node at (-0.22,-.7) {$\scriptstyle 1$};
    			\node at (0.35,-.7) {$\scriptstyle b$};
    			\draw (0.08,.2) \wdot;
\end{tikzpicture}
    		=
\begin{tikzpicture}[baseline = -.5mm,scale=0.8, color=\clr]
    			\draw[-,line width=1pt] (0.3,-.5) to (0.08,0.04);
    			\draw[-,line width=1pt] (-0.2,-.5) to (0.08,0.04);
    			\draw[-,line width=2pt] (0.08,.6) to (0.08,0);
    			\node at (-0.22,-.7) {$\scriptstyle 1$};
    			\node at (0.35,-.7) {$\scriptstyle b$};
    			\draw (-.08,-.3) \wdot;
\end{tikzpicture}
    		+
\begin{tikzpicture}[baseline = -.5mm,scale=0.8, color=\clr]
    			\draw[-,line width=1pt] (0.3,-.5) to (0.08,0.04);
    			\draw[-,line width=1pt] (-0.2,-.5) to (0.08,0.04);
    			\draw[-,line width=2pt] (0.08,.6) to (0.08,0);
    			\node at (-0.22,-.7) {$\scriptstyle 1$};
    			\node at (0.35,-.7) {$\scriptstyle b$};
    			\draw (.22,-.3) \wdot;
    		\end{tikzpicture}\:,
\end{align} 
    	
\begin{align}
\label{mergesplit over C}
\begin{tikzpicture}[baseline = 7.5pt,scale=0.7, color=\clr]
    \draw[-,line width=1pt] (-.2,1.1) to (.08,.64) to (.34,1.1) to (.08,.64);
    \draw[-,line width=2pt] (.08,.66) to (.08,0.08);
    \draw[-,line width=1pt] (-.2,-.4) to (.08,0.1) to (.34,-.4) to (.08,0.1);
    \node at (-0.2,-.6) {$\scriptstyle 1$};
    \node at (0.34,-.6) {$\scriptstyle 1$};
\end{tikzpicture}
    		-
\begin{tikzpicture}[baseline = 7.5pt,scale=0.7, color=\clr]
    \draw[-,line width=1pt] (-.2,1.1) to (.08,.64) to (.34,1.1) to (.08,.64);
    \draw[-,line width=2pt] (.08,.66) to (.08,0.08);
    \draw[-,line width=1pt] (-.2,-.4) to (.08,0.1) to (.34,-.4) to (.08,0.1);
    \draw (-.06,.85) \wdot;
    \draw (.21,.85) \wdot;
    \draw (-.06,-.15) \wdot;
    \draw (.21,.-.15) \wdot;
    \node at (-0.2,-.6) {$\scriptstyle 1$};
    \node at (0.34,-.6) {$\scriptstyle 1$};
\end{tikzpicture}
    		=2
\begin{tikzpicture}[baseline = 7.5pt,scale=0.7, color=\clr]
    \draw[-,line width=1pt] (-.04,1.1) to (-.04,-.4) ;
    \draw[-,line width=1pt] (.3,1.1) to (.3,-.4) ;
    \node at (-.04,-.6) {$\scriptstyle 1$};
    \node at (.3,-.6) {$\scriptstyle 1$};
\end{tikzpicture}\:,
\end{align}	
    	
\begin{equation}
\label{rung-sawp one}
\begin{tikzpicture}[baseline = 2mm,scale=0.6,color=\clr]
    \draw[-,thick] (0,-.4) to (0,1.2);
    \draw[-,thick] (.024,-0.4) to (0.024,0) to (1.02,.25) to (1.02,.55)to (.024,.8) to (.024,1.2);
    \draw[-,line width=1.2pt] (1,-0.4) to (1,1.2);
    \node at (1,-.6) {$\scriptstyle b$};
    \node at (0,-.6) {$\scriptstyle a$};
    \node at (0.5,.9) {$\scriptstyle 1$};
    \node at (0.5,-.1) {$\scriptstyle 1$};
\end{tikzpicture}
    		-
\begin{tikzpicture}[baseline = 2mm,scale=0.6, color=\clr]
    \draw[-,line width=1.2pt] (0,-.4) to (0,1.2);
    \draw[-,thick] (0.98,-.4) to (0.98,0) to (-.02,.25) to (-.02,.55)to (.98,.8) to (.98,1.2);
    \draw[-,thin] (1,-.4) to (1,1.2);
    \node at (1,-.6) {$\scriptstyle b$};
    \node at (0,-.6) {$\scriptstyle a$};
    \node at (0.5,.9) {$\scriptstyle 1$};
    \node at (0.5,-.1) {$\scriptstyle 1$}; 
\end{tikzpicture}
    		=
    		(a-b)
\begin{tikzpicture}[baseline = 2mm,scale=0.6, color=\clr]
    \draw[-,line width=1.pt] (0,-.4) to (0,1.2);
    \draw[-,line width=1.pt] (.7,-.4) to (.7,1.2);
    \node at (.7,-.6) {$\scriptstyle b$};
    \node at (0,-.6) {$\scriptstyle a$};  
\end{tikzpicture}\:,
\end{equation} 
    	
\begin{equation}
\label{rung-sawp with wdots}
\begin{tikzpicture}[baseline = 2mm,scale=0.6,color=\clr]
    \draw[-,thick] (0,-.4) to (0,1.2);
    \draw[-,thick] (.024,-0.4) to (0.024,0) to (1.02,.25) to (1.02,.55)to (.024,.8) to (.024,1.2);
    \draw[-,line width=1.2pt] (1,-0.4) to (1,1.2);
    \node at (1,-.6) {$\scriptstyle b$};
    \node at (0,-.6) {$\scriptstyle a$};
    \node at (0.5,1) {$\scriptstyle 1$};
    \node at (0.5,-.1) {$\scriptstyle 1$};
    \draw (.5,.675) \wdot;
\end{tikzpicture}
    		-
\begin{tikzpicture}[baseline = 2mm,scale=0.6, color=\clr]
    \draw[-,line width=1.2pt] (0,-.4) to (0,1.2);
    \draw[-,thick] (0.98,-.4) to (0.98,0) to (-.02,.25) to (-.02,.55)to (.98,.8) to (.98,1.2);
    \draw[-,thin] (1,-.4) to (1,1.2);
    \node at (1,-.6) {$\scriptstyle b$};
    \node at (0,-.6) {$\scriptstyle a$};
    \node at (0.5,.9) {$\scriptstyle 1$};
    \node at (0.5,-.2) {$\scriptstyle 1$};
    \draw (.5,.1) \wdot;    
\end{tikzpicture}
    		=
\begin{tikzpicture}[baseline = 2mm,scale=0.6, color=\clr]
    \draw[-,line width=1.pt] (0,-.4) to (0,1.2);
    \draw[-,line width=1.pt] (.7,-.4) to (.7,1.2);
    \node at (.7,-.6) {$\scriptstyle b$};
    \node at (0,-.6) {$\scriptstyle a$};  
    \draw (0,.5) \wdot;     
\end{tikzpicture}
    		-
\begin{tikzpicture}[baseline = 2mm,scale=0.6, color=\clr]
    \draw[-,line width=1.pt] (0,-.4) to (0,1.2);
    \draw[-,line width=1.pt] (.7,-.4) to (.7,1.2);
    \node at (.7,-.6) {$\scriptstyle b$};
    \node at (0,-.6) {$\scriptstyle a$};   
    \draw (0.7,.5) \wdot; 
\end{tikzpicture}
    		=
\begin{tikzpicture}[baseline = 2mm,scale=0.6,color=\clr]
    \draw[-,thick] (0,-.4) to (0,1.2);
    \draw[-,thick] (.024,-0.4) to (0.024,0) to (1.02,.25) to (1.02,.55)to (.024,.8) to (.024,1.2);
    \draw[-,line width=1.2pt] (1,-0.4) to (1,1.2);
    \node at (1,-.6) {$\scriptstyle b$};
    \node at (0,-.6) {$\scriptstyle a$};
    \node at (0.5,1) {$\scriptstyle 1$};
    \node at (0.5,-.2) {$\scriptstyle 1$};
    \draw (.5,.1) \wdot;
\end{tikzpicture}
    		-
\begin{tikzpicture}[baseline = 2mm,scale=0.6, color=\clr]
    \draw[-,line width=1.2pt] (0,-.4) to (0,1.2);
    \draw[-,thick] (0.98,-.4) to (0.98,0) to (-.02,.25) to (-.02,.55)to (.98,.8) to (.98,1.2);
    \draw[-,thin] (1,-.4) to (1,1.2);
    \node at (1,-.6) {$\scriptstyle b$};
    \node at (0,-.6) {$\scriptstyle a$};
    \node at (0.5,1) {$\scriptstyle 1$};
    \node at (0.5,-.1) {$\scriptstyle 1$};  
    \draw (.5,.675) \wdot;  
\end{tikzpicture}\:,
\end{equation} 
    	
\begin{equation}
\label{wdot-rung-relation_1}
\begin{tikzpicture}[baseline = -1mm,scale=0.6,color=\clr] 
    \draw[-,line width=1.2pt] (0,-0.6) to (0,1);
    \draw[-,line width=1.2pt] (0.8,-0.6) to (0.8,1.);
    \draw[-,line width=1.2pt] (1.6,-0.6) to (1.6,1.);
    \draw[-,line width=1.2pt] (0,.8) to (.8,.4);
    \draw[-,line width=1.2pt] (.8,0) to (1.6,-.4);
    \node at (0,-.8) {$\scriptstyle a$};
    \node at (.8,-.8) {$\scriptstyle b$};
    \node at (1.6,-.8) {$\scriptstyle c$};
    \node at (.4,.2) {$\scriptstyle 1$};
    \node at (1.2,.1) {$\scriptstyle 1$};
\end{tikzpicture}
    		-~
\begin{tikzpicture}[baseline = -1mm,scale=0.6,color=\clr]
    \draw[-,line width=1.2pt] (0,-0.6) to (0,1);
    \draw[-,line width=1.2pt] (0.8,-.6) to (0.8,1);
    \draw[-,line width=1.2pt] (1.6,-.6) to (1.6,1);
    \draw[-,line width=1.2pt] (0,0) to (.8,-.4);
    \draw[-,line width=1.2pt] (.8,.8) to (1.6,.4);
    \node at (0,-.8) {$\scriptstyle a$};
    \node at (.8,-.8) {$\scriptstyle b$};
    \node at (1.6,-.8) {$\scriptstyle c$};
    \node at (.4,.1) {$\scriptstyle 1$};
    \node at (1.2,.2) {$\scriptstyle 1$};
\end{tikzpicture}
    		=~
\begin{tikzpicture}[baseline = -1mm,scale=0.6,color=\clr]
    \draw[-,line width=1.2pt] (0,-0.6) to (0,1);
    \draw[-,line width=1.2pt] (0.8,-0.6) to (0.8,1);
    \draw[-,line width=1.2pt] (1.6,-0.6) to (1.6,1);
    \draw[-,line width=1.2pt] (0,.8) to (.8,.4);
    \draw[-,line width=1.2pt] (.8,0) to (1.6,-.4);
    \draw(.4,.6) \wdot;
    \draw(1.2,-.2) \wdot;
    \node at (0,-.8) {$\scriptstyle a$};
    \node at (.8,-.8) {$\scriptstyle b$};
    \node at (1.6,-.8) {$\scriptstyle c$};
    \node at (1.2,.1) {$\scriptstyle 1$};
    \node at (.4,.2) {$\scriptstyle 1$};
\end{tikzpicture}
    		+~
\begin{tikzpicture}[baseline = -1mm,scale=0.6,color=\clr]
    \draw[-,line width=1.2pt] (0,-0.6) to (0,1);
    \draw[-,line width=1.2pt] (0.8,-.6) to (0.8,1);
    \draw[-,line width=1.2pt] (1.6,-.6) to (1.6,1);
    \draw[-,line width=1.2pt] (0,0) to (.8,-.4);
    \draw[-,line width=1.2pt] (.8,.8) to (1.6,.4);
    \draw(.4,-.2) \wdot;
    \draw(1.2,.6) \wdot;
    \node at (0,-.8) {$\scriptstyle a$};
    \node at (.8,-.8) {$\scriptstyle b$};
    \node at (1.6,-.8) {$\scriptstyle c$};
    \node at (1.2,.2) {$\scriptstyle 1$};
    \node at (.4,.1) {$\scriptstyle 1$};
\end{tikzpicture}\:,
\end{equation}
    	
\begin{equation}
\label{wdot-rung-relation_2}
\begin{tikzpicture}[baseline = -1mm,scale=0.6,color=\clr] 
    \draw[-,line width=1.2pt] (0,-0.6) to (0,1);
    \draw[-,line width=1.2pt] (0.8,-0.6) to (0.8,1.);
    \draw[-,line width=1.2pt] (1.6,-0.6) to (1.6,1.);
    \draw[-,line width=1.2pt] (0,.8) to (.8,.4);
    \draw[-,line width=1.2pt] (.8,0) to (1.6,-.4);
    \draw(1.2,-.2) \wdot;
    \node at (0,-.8) {$\scriptstyle a$};
    \node at (.8,-.8) {$\scriptstyle b$};
    \node at (1.6,-.8) {$\scriptstyle c$};
    \node at (1.2,.1) {$\scriptstyle 1$};
    \node at (.4,.4) {$\scriptstyle 1$};
\end{tikzpicture}
    		-~
\begin{tikzpicture}[baseline = -1mm,scale=0.6,color=\clr]
    \draw[-,line width=1.2pt] (0,-0.6) to (0,1);
    \draw[-,line width=1.2pt] (0.8,-.6) to (0.8,1);
    \draw[-,line width=1.2pt] (1.6,-.6) to (1.6,1);
    \draw[-,line width=1.2pt] (0,0) to (.8,-.4);
    \draw[-,line width=1.2pt] (.8,.8) to (1.6,.4);
    \draw(1.2,.6) \wdot;
    \node at (0,-.8) {$\scriptstyle a$};
    \node at (.8,-.8) {$\scriptstyle b$};
    \node at (1.6,-.8) {$\scriptstyle c$};
    \node at (1.2,.2) {$\scriptstyle 1$};
    \node at (.4,.1) {$\scriptstyle 1$};
\end{tikzpicture}
    		=~
\begin{tikzpicture}[baseline = -1mm,scale=.6,color=\clr]
    \draw[-,line width=1.2pt] (0,-0.6) to (0,1);
    \draw[-,line width=1.2pt] (0.8,-0.6) to (0.8,1);
    \draw[-,line width=1.2pt] (1.6,-0.6) to (1.6,1);
    \draw[-,line width=1.2pt] (0,.8) to (.8,.4);
    \draw[-,line width=1.2pt] (.8,0) to (1.6,-.4);
    \draw(.4,.6) \wdot;			
    \node at (0,-.8) {$\scriptstyle a$};
    \node at (.8,-.8) {$\scriptstyle b$};
    \node at (1.6,-.8) {$\scriptstyle c$};
    \node at (1.2,.1) {$\scriptstyle 1$};
    \node at (.4,.2) {$\scriptstyle 1$};
\end{tikzpicture}
    		-~
\begin{tikzpicture}[baseline = -1mm,scale=0.6,color=\clr]
    \draw[-,line width=1.2pt] (0,-0.6) to (0,1);
    \draw[-,line width=1.2pt] (0.8,-.6) to (0.8,1);
    \draw[-,line width=1.2pt] (1.6,-.6) to (1.6,1);
    \draw[-,line width=1.2pt] (0,0) to (.8,-.4);
    \draw[-,line width=1.2pt] (.8,.8) to (1.6,.4);
    \draw(.4,-.2) \wdot;
    \node at (0,-.8) {$\scriptstyle a$};
    \node at (.8,-.8) {$\scriptstyle b$};
    \node at (1.6,-.8) {$\scriptstyle c$};
    \node at (.4,.2) {$\scriptstyle 1$};
    \node at (1.2,.3) {$\scriptstyle 1$};
\end{tikzpicture}\:.
\end{equation}    	
The $\mathbb{Z}_2$-grading is given by declaring $	
\begin{tikzpicture}[baseline = 3pt, scale=0.5, color=\clr]
    \draw[-,line width=2pt] (0,0) to[out=up, in=down] (0,1.4);
    \draw(0,0.6) \wdot; 
    \node at (0,-.3) {$\scriptstyle a$};
\end{tikzpicture}  
    	$ to have parity $\bar{1}$ and other generating morphisms to have  parity $\bar{0}$.
\end{definition}

\begin{theorem}
\label{isom-qWeb-overc}
Suppose $\kk=\C$. Then $\qW$ is isomorphic to $\mathfrak{q}\text{-}\mathbf{Web}$.
\end{theorem}
\begin{proof}
We will construct two natural functors $\alpha: \Qwb\rightarrow \qW$ and its inverse $\beta$.
The functor $\alpha$ is defined by mapping the generating objects and generating morphisms to the same notation in $\qW$.
To show $\alpha$ is well-defined, it suffices to verify 
the relations \eqref{wdot-move-split and merge over C}--\eqref{wdot-rung-relation_2} hold in $\qW$. Specifically, the relation \eqref{wdot-move-split and merge over C} is the special case of \eqref{wdot-move-split and merge}, while  the  relations 
\eqref{mergesplit over C}--\eqref{wdot-rung-relation_2} are verified in the lemma \ref{remaining relation in QWeb}. 

To define $\beta$, we need the  crossings defined in $\Qwb$ \cite{BKu21}.
For the thin crossing, we define
\begin{equation}
\label{crossing-11}
\begin{tikzpicture}[baseline = 2mm,scale=0.6, color=\clr]
    \draw[-,thick] (0,0) to (.6,1);
    \draw[-,thick] (0,1) to (.6,0);
    \node at (0,-.2) {$\scriptstyle 1$};
    \node at (0.6,-0.2) {$\scriptstyle 1$};
\end{tikzpicture}:
    	 =
\begin{tikzpicture}[baseline = 2mm,scale=0.6, color=\clr]
    \draw[-,thick] (0,0) to (.28,.3) to (.28,.7) to (0,1);
    \draw[-,thick] (.6,0) to (.31,.3) to (.31,.7) to (.6,1);
    \node at (0,-.2) {$\scriptstyle 1$};
    \node at (0.63,-.2) {$\scriptstyle 1$};
\end{tikzpicture}
    	 	-
\begin{tikzpicture}[baseline = 2mm, scale=0.6, color=\clr]
    \draw[-,thick] (0,0)to[out=up,in=down](0,1);
    \draw[-,thick] (0.5,0)to[out=up,in=down](0.5,1); 
    \node at (0,-.2) {$\scriptstyle 1$};
    \node at (0.5,-0.2) {$\scriptstyle 1$};
\end{tikzpicture}.	 	  
\end{equation}
The thick crossing is defined as 
\cite[Definition 4.5.1]{BKu21}:
\begin{equation}
\label{crossgen}
\begin{tikzpicture}[baseline = 2mm,scale=0.8, color=\clr]
    \draw[-,line width=1.2pt] (0,-.2) to (.6,.8);
    \draw[-,line width=1.2pt] (0,.8) to (.6,-.2);
    \node at (0,-.3) {$\scriptstyle a$};
    \node at (0.6,-0.3) {$\scriptstyle b$};
    \node at (0,.9) {$\scriptstyle b$};
    \node at (0.6,.9) {$\scriptstyle a$};
\end{tikzpicture}
:=\frac{1}{a!b!}
\begin{tikzpicture}[baseline = 2mm,scale=.25, color=\clr]
    \draw[-,line width=1.2pt] (0, 3.75) to (0,4.5);
    \draw[-,line width=1.2pt] (1,2) to [out=90,in=330] (0,3.75);
    \draw[-,line width=1.2pt] (-1,2) to [out=90,in=210] (0,3.75);		
    \draw[-,line width=1.2pt] (4, 3.75) to (4,4.5);
    \draw[-,line width=1.2pt] (5,2) to [out=90,in=330] (4,3.75);
    \draw[-,line width=1.2pt] (3,2) to [out=90,in=210] (4,3.75);
    \draw[-,line width=1.2pt] (0,-3.5) to (0,-2.75);
    \draw[-,line width=1.2pt] (0,-2.75) to [out=30,in=270] (1,-1);
    \draw[-,line width=1.2pt] (0,-2.75) to [out=150,in=270] (-1,-1); 
    \draw[-,line width=1.2pt] (4,-3.5) to (4,-2.75);
    \draw[-,line width=1.2pt] (4,-2.75) to [out=30,in=270] (5,-1);
    \draw[-,line width=1.2pt] (4,-2.75) to [out=150,in=270] (3,-1); 
    \draw[-,line width=1.2pt] (5,2) to (1,-1);
    \draw[-,line width=1.2pt] (3,2) to (-1,-1);
    \draw[-,line width=1.2pt] (1,2) to (5,-1);
    \draw[-,line width=1.2pt] (-1,2) to (3,-1);
    \node at (2.6, 2.7) {\scriptsize $1$};
    \node at (5.4, 2.7) {\scriptsize $1$};
    \node at (1.4, 2.7) {\scriptsize $1$};
    \node at (-1.4, 2.7) {\scriptsize $1$};
    \node at (2.6, -1.7) {\scriptsize $1$};
    \node at (5.4, -1.7) {\scriptsize $1$};
    \node at (1.4, -1.7) {\scriptsize $1$};
    \node at (-1.4, -1.7) {\scriptsize $1$};
    \node at (0.1, 2.35) { $\cdots$};
    \node at (4.1, 2.35) { $\cdots$};
    \node at (0.1, -1.6) { $\cdots$};
    \node at (4.1, -1.6) { $\cdots$};
    \node at (0,5) {\scriptsize $b$};
    \node at (4,5) {\scriptsize $a$};
    \node at (0,-4) {\scriptsize $a$};
    \node at (4,-4) {\scriptsize $b$};
\end{tikzpicture}.
\end{equation}
We define the functor $\beta: \qW\rightarrow\Qwb$ by sending the generating objects, merges, splits, and white dots to the same notation in $\Qwb$.
Then $\alpha\circ \beta$ and $\beta\circ \alpha$ fix all generator and it follows that   $\alpha\circ\beta=id_{\Qwb}, \beta\circ\alpha=id_{\qW}$. To complete the proof, 
it remains to verify that  $\beta$ is well-defined, which is provided in Lemma \ref{lem:wellofbeta}.
\end{proof}
\begin{lemma}
\label{remaining relation in QWeb}
    The relations \eqref{mergesplit over C}--\eqref{wdot-rung-relation_2} hold in $\qW$.
    \end{lemma}
\begin{proof}
    	The relations \eqref{webassoc}--\eqref{equ:r=0} implies the rung relation \cite[(4.26)]{BEEO}: 
\begin{equation}\label{rung-sawp}
,
\end{equation} 
    	and \eqref{rung-sawp one} holds since it is a special case of \eqref{rung-sawp}.
  Next, we see that \eqref{mergesplit over C} holds since   	
\begin{align*}
%
\:.
\end{align*}
For \eqref{rung-sawp with wdots} we have
\begin{align*}
%
&\overset{\eqref{wdot-move-split and merge}}= 
%
  
\overset{\eqref{wdot-move-split and merge}}= 	
%
 
\overset{\eqref{wdot-move-split and merge}}=  
%
\:.
\end{align*}
     Finally, \eqref{wdot-rung-relation_2} can be checked in a similar way.  
    \end{proof}

Using \eqref{webassoc}, \eqref{equ:r=0}, \cite[Lemma 4.7]{BKu21}, and the definition of crossings in \eqref{crossing-11}--\eqref{crossgen}, 
we can prove that  \eqref{swallows} holds in $\Qwb$  by induction on $a+b$. Additionally, the relations \eqref{sliders} and \eqref{rung-sawp} hold in $\Qwb$ by \cite[Lemma 4.12, Lemma 4.4]{BKu21}. 
 In $\W$,  crossings can be expressed with merges and splits by \cite[(4.36)]{BEEO}, the following lemma shows we have the same formula in $\Qwb$.

\begin{lemma} 
In $\Qwb$, we have
\begin{equation}
\label{def-crossing-by-rung}
\begin{tikzpicture}[baseline = 2mm,scale=1, color=\clr]
    \draw[-,line width=1.2pt] (0,0) to (.6,1);
    \draw[-,line width=1.2pt] (0,1) to (.6,0);
    \node at (0,-.1) {$\scriptstyle a$};
    \node at (0.6,-0.1) {$\scriptstyle b$};
\end{tikzpicture}
=\sum_{t=0}^{\min(a,b)}
    		(-1)^t
\begin{tikzpicture}[baseline = 2mm,scale=1, color=\clr]
    \draw[-,line width=1.2pt] (0,0) to (0,1);
    \draw[-,thick] (0.8,0) to (0.8,.2) to (.03,.4) to (.03,.6)to (.8,.8) to (.8,1);
    \draw[-,thin] (0.82,0) to (0.82,1);
    \node at (0.81,-.1) {$\scriptstyle b$};
    \node at (0,-.1) {$\scriptstyle a$};
    \node at (1,.5) {$\scriptstyle t$};
\end{tikzpicture}.
\end{equation}
\end{lemma}
\begin{proof}
We proceed     by induction on $ a+b$. 
    	Suppose $a=1$. The base case for  $a=b=1$ is just \eqref{crossing-11}. For $b\ge 2$, we have
\begin{equation}
\label{crossrung1b}
\begin{aligned}
    		b
\begin{tikzpicture}[baseline = 7.5pt, scale=0.4, color=\clr]
    \draw[-,line width=1.2pt] (0,-.2) to  (1,2.2);
    \draw[-,line width=1.2pt] (1,-.2) to  (0,2.2);
    \node at (0, -.5) {$\scriptstyle 1$};
    \node at (1, -.5) {$\scriptstyle b$};
\end{tikzpicture}
&\overset{\eqref{equ:r=0}}{\underset{\eqref{sliders}}{=}}
\begin{tikzpicture}[baseline = 7.5pt, scale=0.4, color=\clr]
    \draw[-,line width=1.2pt](0,2.2) to (0.3,1.9);
    \draw[-,thick](0.3,1.9) to[out=-80,in=135](1.5,.5) to[out=-45,in=45](2,.5);
    \draw[-,thick](0.3,1.9) to[out=45,in=135](1.5,1.5) to[out=-45,in=45] (2,.5);
    \draw[-,line width=1.2pt](2,.5) to (2.5,0);
    \draw[-,line width=1pt](0,0) to (2.5,2.2);
    \node at (1.5, .2) {$\scriptstyle 1$};
    \node at (0, -.5) {$\scriptstyle 1$};
    \node at (2.5, -.5) {$\scriptstyle b$};
\end{tikzpicture}
    		=
\begin{tikzpicture}[baseline = 7.5pt, scale=0.4, color=\clr]
    \draw[-,line width=1.2pt](0,-1) to (0,-.5);
    \draw[-,line width=1.2pt](1.8,-1) to (1.8,2.2);
    \draw[-,line width=1pt](0,-.5) to (1.8,.5);
    \draw[-,line width=1.2pt](0,.5) to (1.8,-.5);
    \draw[-,line width=1.2pt](0,.5) to (0,2.2);
    \draw[-,line width=1.2pt](1.8,1) to (0,1.8);
    \node at (2.7, 0) {$\scriptstyle b-1$};
    \node at (1, 2) {$\scriptstyle b-1$};
    \node at (0, -1.3) {$\scriptstyle 1$};
    \node at (1.8, -1.3) {$\scriptstyle b$};
\end{tikzpicture}
    		-
\begin{tikzpicture}[baseline = 7.5pt, scale=0.4, color=\clr]
    	\draw[-,line width=1.2pt](0,-1) to (0,0);
    	\draw[-,line width=1.2pt](1.5,-1) to (1.5,-.5);
                \draw[-,line width=1.2pt](1.5,-.5) to[out=135,in=135] (1,0);
                \draw[-,line width=1.2pt](1.5,-.5) to[out=45,in=45] (2,0);
                \draw[-,line width=1.2pt](0,0) to(1,.6);
                \draw[-,line width=1.2pt](1,0) to (0,0.6);
                \draw[-,line width=1.2pt](2,0) to (2,2.2);
                \draw[-,line width=1.2pt](1,0.6) to (1,1.8);
                \draw[-,line width=1.2pt](0,0.6) to (0,1.8);
                \draw[-,line width=1.2pt](2,1) to (1,1.6);
                \draw[-,line width=1.2pt](0,1.8) to (.5,2.2);
                \draw[-,line width=1.2pt](1,1.8) to (0.5,2.2);
                \draw[-,line width=1.2pt](.5,2.2) to (.5,2.4);
    			\node at (0, -1.3) {$\scriptstyle 1$};
    			\node at (1.5, -1.3) {$\scriptstyle b$};
    			\node at (1, -.5) {$\scriptstyle 1$};
\end{tikzpicture}
    		\quad \text{by induction hypothesis}
    		\\
    		&\overset{\eqref{equ:r=0}\eqref{swallows}}{\underset{\eqref{webassoc}}{=}}
\begin{tikzpicture}[baseline = 7.5pt, scale=0.35, color=\clr]
    			\draw[-,line width=1.2pt](0,-1) to (0,2.2);
    			\draw[-,line width=1.2pt](1.8,-1) to (1.8,2.2);
    			\draw[-,line width=1.2pt](1.8,.-.7) to (0,-.2);
    			\node at (0.9, -.8) {$\scriptstyle 1$};
    			\draw[-,line width=1.2pt](0,0.2) to (1.8,.7);
    			\node at (0.9, .8) {$\scriptstyle 1$};
    			\draw[-,line width=1.2pt](1.8,1.1) to (0,1.6);
    			\node at (0.9, 1.9) {$\scriptstyle b-1$};
    			\node at (0, -1.3) {$\scriptstyle 1$};
    			\node at (1.8, -1.3) {$\scriptstyle b$};
\end{tikzpicture}
    		-b
\begin{tikzpicture}[baseline = 7.5pt, scale=0.35, color=\clr]
    			\draw[-,line width=1.2pt](0,-1) to (0,2.2);
    			\draw[-,line width=1.2pt](1.8,-1) to (1.8,2.2);
    			\draw[-,line width=1.2pt](1.8,0) to (0,.5);
      			\node at (0.9, 1) {$\scriptstyle b-1$};
    			\node at (0, -1.3) {$\scriptstyle 1$};
    			\node at (1.8, -1.3) {$\scriptstyle b$};
\end{tikzpicture}
    		-(b-1)
\begin{tikzpicture}[baseline = 7.5pt, scale=0.35, color=\clr]
    			\draw[-,line width=1.2pt](0,-1) to (0,2.2);
    			\draw[-,line width=1.2pt](1.8,-1) to (1.8,2.2);
    			\draw[-,line width=1.2pt](1.8,0) to (0,.5);
    			\node at (0.9, 1) {$\scriptstyle b-1$};
    			\node at (0, -1.3) {$\scriptstyle 1$};
    			\node at (1.8, -1.3) {$\scriptstyle b$};
\end{tikzpicture}
    	\overset{\eqref{rung-sawp}\eqref{equ:r=0}}{\underset{\eqref{webassoc}}{=}}b
\begin{tikzpicture}[baseline = 7.5pt,scale=.35, color=\clr]
    \draw[-,line width=1.2pt](0,-1) to (0.9,0);
    \draw[-,line width=1.2pt](1.8,-1) to (.9,0);
    \draw[-,line width=1.2pt](0,2.2) to (0.9,1.2);
    \draw[-,line width=1.2pt](1.8,2.2) to (.9,1.2);
    \draw[-,line width=1.2pt](.9,0) to (.9,1.2);
    \node at (0, -1.4) {$\scriptstyle 1$};
    \node at (1.8, -1.4) {$\scriptstyle b$};
\end{tikzpicture}
    		-b
\begin{tikzpicture}[baseline = 7.5pt, scale=0.35, color=\clr]
    \draw[-,line width=1.2pt](0,-1) to (0,2.2);
    \draw[-,line width=1.2pt](1.8,-1) to (1.8,2.2);
    \draw[-,line width=1.2pt](1.8,0) to (0,.5);
    \node at (0.9, 1) {$\scriptstyle b-1$};
    \node at (0, -1.4) {$\scriptstyle 1$};
    \node at (1.8, -1.4) {$\scriptstyle b$};
\end{tikzpicture}.
\end{aligned}
\end{equation}
    This completes the proof of the case $a=1$. Suppose $a\ge 2$. We have 
\begin{align*}
    		a
\begin{tikzpicture}[baseline = 7.5pt, scale=0.4, color=\clr]
    \draw[-,line width=1.2pt] (0,-.2) to  (1,2.2);
    \draw[-,line width=1.2pt] (1,-.2) to  (0,2.2);
    \node at (0, -.5) {$\scriptstyle a$};
    \node at (1, -.5) {$\scriptstyle b$};
\end{tikzpicture}
    		&\overset{\eqref{equ:r=0}}{\underset{\eqref{sliders}}{=}}
\begin{tikzpicture}[baseline = 7.5pt, scale=0.4, color=\clr]
    \draw[-,line width=1.2pt](0,1.7) to (2.3,-0.2);
    \draw[-,line width=1.2pt](.5,.2) to[out=75,in=left] (2,1.7);
    \draw[-,line width=1.2pt](2,1.7) to [out=-100,in=right](.5,.2);
    \draw[-,line width=1.2pt](2,1.7)to(2.2,1.9);
    \draw[-,line width=1.pt](0,-.2) to (.5,.2);
    \node at (0.2, .5) {$\scriptstyle 1$};
    \node at (0, -.5) {$\scriptstyle a$};
    \node at (2.3, -.5) {$\scriptstyle b$};
\end{tikzpicture}
    		\overset{\eqref{crossrung1b}}{=}
\begin{tikzpicture}[baseline = 7.5pt, scale=0.4, color=\clr]
    \draw[-,line width=1.2pt](0,-1.5) to (0,-1);
    \draw[-,line width=1.2pt](0,-1) to (-.4,-.5);
    \draw[-,line width=1.2pt](0,-1) to (0.4,-.5);
    \draw[-,line width=1.2pt](-.4,-.5) to (-.4,.5);
    \draw[-,line width=1.2pt](.4,-.5) to (1.4,.5);
    \draw[-,line width=1.2pt](1.4,-.5) to (.4,.5);
    \draw[-,line width=1.2pt](-.4,.5) to (0,1);
    \draw[-,line width=1.2pt](0.4,.5) to (0,1);
    \draw[-,line width=1.2pt](0,1) to (0,1.5);
    \draw[-,line width=1.2pt](0,1.5) to (-.4,2);
    \draw[-,line width=1.2pt](0,1.5) to (0.4,2);
    \draw[-,line width=1.2pt](1.4,-1.5) to (1.4,-.5);
    \draw[-,line width=1.2pt](1.4,.5) to (1.4,2);
    \draw[-,line width=1.2pt](0.4,2) to (.9,2.4);
    \draw[-,line width=1.2pt](1.4,2) to (0.9,2.4);
    \draw[-,line width=1.2pt](0.9,2.4) to (0.9,2.6);
    \draw[-,line width=1.2pt](-0.4,2) to (-0.4,2.6);
    \node at (0,-1.7) {$\scriptstyle a$};
    \node at (1.4, -1.7) {$\scriptstyle b$};
    \node at (-.4, -1.1) {$\scriptstyle 1$};
    \node at (2.2, 1.5) {$\scriptstyle a-1$};
\end{tikzpicture}
    		-
\begin{tikzpicture}[baseline = 7.5pt, scale=0.4, color=\clr]
    \draw[-,line width=1.2pt](0,-1.5) to (0,-1);
    \draw[-,line width=1.2pt](0,-1) to (-.4,-.5);
    \draw[-,line width=1.2pt](0,-1) to (0.4,-.5);
    \draw[-,line width=1.2pt](-.4,-.5) to (-.4,.5);
    \draw[-,line width=1.2pt](.4,-.5) to (1.4,.5);
    \draw[-,line width=1.2pt](1.4,-.5) to (.4,.5);
    \draw[-,line width=1.2pt](-0.4,.5) to (-0.4,2);
    \draw[-,line width=1.2pt](0.4,.5) to (0.4,2);
    \draw[-,line width=1.2pt](0.4,1) to (-0.4,1.6);
    \draw[-,line width=1.2pt](1.4,-1.5) to (1.4,-.5);
    \draw[-,line width=1.2pt](1.4,.5) to (1.4,2);
    \draw[-,line width=1.2pt](0.4,2) to (.9,2.4);
    \draw[-,line width=1.2pt](1.4,2) to (0.9,2.4);
    \draw[-,line width=1.2pt](0.9,2.4) to (0.9,2.6);
    \draw[-,line width=1.2pt](-0.4,2) to (-0.4,2.6);
    \node at (0,-1.7) {$\scriptstyle a$};
    \node at (1.4, -1.7) {$\scriptstyle b$};
    \node at (-.4, -1.1) {$\scriptstyle 1$};
    \node at (2.2, 1.5) {$\scriptstyle a-1$};
\end{tikzpicture}
    		\\
    		&\overset{\eqref{swallows}}{\underset{\eqref{sliders}}{=}}\sum_{t\ge 0}(-1)^t
\begin{tikzpicture}[baseline = 7.5pt, scale=0.5, color=\clr]
    \draw[-,line width=1.2pt](0,-1.5) to (0,-1);
    \draw[-,line width=1.2pt](0,-1) to (-.4,-.5);
    \draw[-,line width=1.2pt](0,-1) to (0.4,-.5);
    \draw[-,line width=1.2pt](-.4,-.5) to (-.4,.5);
    \draw[-,line width=1.2pt](.4,-.5) to (.4,.5);
    \draw[-,line width=1.2pt](-.4,.5) to (0,1);
    \draw[-,line width=1.2pt](0.4,.5) to (0,1);
    \draw[-,line width=1.2pt](0,1) to (0,1.5);
    \draw[-,line width=1.2pt](0,1.5) to (-.4,2);
    \draw[-,line width=1.2pt](0,1.5) to (0.4,2);
    \draw[-,line width=1.2pt](1.4,-1.5) to (1.4,2);
    \draw[-,line width=1.2pt](1.4,-.5) to (0.4,-.2);
    \draw[-,line width=1.2pt](0.4,.1) to (1.4,.4);
     \draw[-,line width=1.2pt](0.4,2) to (.9,2.4);
    \draw[-,line width=1.2pt](1.4,2) to (0.9,2.4);
    \draw[-,line width=1.2pt](0.9,2.4) to (0.9,2.6);
     \draw[-,line width=1.2pt](-0.4,2) to (-0.4,2.6);
    \node at (0,-1.7) {$\scriptstyle a$};
    \node at (1.4, -1.7) {$\scriptstyle b$};
    \node at (1.6, -.2) {$\scriptstyle t$};
    \node at (-.4, -1.1) {$\scriptstyle 1$};
    \node at (2, 1.5) {$\scriptstyle a-1$};
\end{tikzpicture}
    		+
\begin{tikzpicture}[baseline = 7.5pt, scale=0.5, color=\clr]
    \draw[-,line width=1.2pt](0,-1) to (0,1.8);
    \draw[-,line width=1.2pt](1.4,-1) to (1.4,1.8);
    \draw[-,line width=1.2pt](0,-.5) to (1.4,1);
    \draw[-,line width=1.2pt](1.4,-.5) to (0,1);
    \node at (0,-1.5) {$\scriptstyle a$};
    \node at (1.4, -1.5) {$\scriptstyle b$};
    \node at (1.6, .3) {$\scriptstyle 1$};
    \node at (-.4, .3) {$\scriptstyle 1$};
\end{tikzpicture}
    		\quad \text{ by induction hypothesis}
    		\\
    		&\overset{\eqref{rung-sawp}}{=}\sum_{t\ge 0}(-1)^t\sum_{s\ge 0}\scalebox{1}{$\binom{t+1}{s}$}
\begin{tikzpicture}[baseline = 7.5pt, scale=0.7, color=\clr]
    \draw[-,line width=1.2pt](1.4,-1.3) to (1.4,-1);
    \draw[-,line width=1.2pt](1.4,-1) to (1,-.5);
    \draw[-,line width=1.2pt](1.4,-1) to (1.8,-.5);
    \draw[-,line width=1.2pt](1,-.5) to (1,.9);
    \draw[-,line width=1.2pt](1.8,-.5) to (1.8,.9);
    \draw[-,line width=1.2pt](1,.9) to (1.4,1.4);
    \draw[-,line width=1.2pt](1.8,.9) to (1.4,1.4);
    \draw[-,line width=1.2pt](1.4,1.4) to (1.4,2);
    \draw[-,line width=1.2pt](0,-1.3) to (0,2);
    \draw[-,line width=1.2pt](0,0) to (1,-.3);
    \draw[-,line width=1.2pt](0,0.3) to (1,.6);
    \draw[-,line width=1.2pt](0,1.1) to (1.4,1.7);
    \node at (0,-1.5) {$\scriptstyle a$};
    \node at (1.4, -1.5) {$\scriptstyle b$};
    \node at (2, .2) {$\scriptstyle t$};
    \node at (.4, -.3) {$\scriptstyle b-s$};
    \node at (1.4, .2) {$\scriptstyle s-t$};
    \node at (.4, 1.7) {$\scriptstyle 1$};
\end{tikzpicture}
    		+
    		\sum_{t\ge 0}(-1)^t\:
\begin{tikzpicture}[baseline = 7.5pt, scale=0.4, color=\clr]
    \draw[-,line width=1.2pt](0,-1.5) to (0,-1);
    \draw[-,line width=1.2pt](0,-1) to (-0.4,-.5);
    \draw[-,line width=1.2pt](0,-1) to (0.4,-.5);
    \draw[-,line width=1.2pt](-.4,-.5) to (-0.4,1.5);
    \draw[-,line width=1.2pt](0.4,-.5) to (0.4,1.5);
    \draw[-,line width=1.2pt](-.4,1.5) to (0,2);
    \draw[-,line width=1.2pt](0.4,1.5) to (0,2);
    \draw[-,line width=1.2pt](0,2) to (0,2.2);
    \draw[-,line width=1.2pt](2.2,-1.5) to (2.2,-1);
    \draw[-,line width=1.2pt](2.2,-1) to (2.6,-.5);
    \draw[-,line width=1.2pt](2.2,-1) to (1.8,-.5);
    \draw[-,line width=1.2pt](1.8,-.5) to (1.8,1.5);
    \draw[-,line width=1.2pt](2.6,-.5) to (2.6,1.5);
    \draw[-,line width=1.2pt](1.8,1.5) to (2.2,2);
    \draw[-,line width=1.2pt](2.6,1.5) to (2.2,2);
    \draw[-,line width=1.2pt](2.2,2) to (2.2,2.2);
    \draw[-,line width=1.2pt](1.8,-.22) to (0.4,.3);
    \draw[-,line width=1.2pt](0.4,.8) to (1.8,1.2);
    \node at (0, -1.8) {$\scriptstyle a$};
    \node at (2.2,-1.8) {$\scriptstyle b$};
    \node at (2,.5) {$\scriptstyle t$};
\end{tikzpicture}
    		\quad \text{ by induction hypothesis}
    		\\
&\overset{\eqref{webassoc}\eqref{equ:r=0}}{\underset{\eqref{rung-sawp}}{=}}
    \sum_{t\ge 0}(-1)^t\big(\sum_{s\ge 0}\scalebox{1}{$\binom{t+1}{s}\binom{s}{t}(a-s)$}
\begin{tikzpicture}[baseline = 2mm,scale=0.8, color=\clr]
    \draw[-,line width=1.2pt] (0,0) to (0,1);
    \draw[-,thick] (0.8,0) to (0.8,.2) to (.03,.4) to (.03,.6)to (.8,.8) to (.8,1);
    \draw[-,thin] (0.82,0) to (0.82,1);
    \node at (0.81,-.15) {$\scriptstyle b$};
    \node at (0,-.15) {$\scriptstyle a$};
    \node at (1,.5) {$\scriptstyle s$};
\end{tikzpicture}
   + \sum_{s\ge 0}\scalebox{1}{$\binom{t+1}{s}\binom{s+1}{t+1}(t+1)$ }
\begin{tikzpicture}[baseline = 2mm,scale=0.8, color=\clr]
    \draw[-,line width=1.2pt] (0,0) to (0,1);
    \draw[-,thick] (0.8,0) to (0.8,.2) to (.03,.4) to (.03,.6)to (.8,.8) to (.8,1);
    \draw[-,thin] (0.82,0) to (0.82,1);
    \node at (0.81,-.15) {$\scriptstyle b$};
    \node at (0,-.15) {$\scriptstyle a$};
    \node at (1.2,.5) {$\scriptstyle s+1$};
\end{tikzpicture}\big)
    		\\
    &=\sum_{t=0}^{\min(a,b)}
    		(-1)^t
\begin{tikzpicture}[baseline = 2mm,scale=0.8, color=\clr]
    \draw[-,line width=1.2pt] (0,0) to (0,1);
    \draw[-,thick] (0.8,0) to (0.8,.2) to (.03,.4) to (.03,.6)to (.8,.8) to (.8,1);
    \draw[-,thin] (0.82,0) to (0.82,1);
    \node at (0.81,-.15) {$\scriptstyle b$};
    \node at (0,-.15) {$\scriptstyle a$};
    \node at (1,.5) {$\scriptstyle t$};
\end{tikzpicture}\:.
\end{align*}
  	    	The last equality follows from the fact $0\le s-t \le 1$.  
    \end{proof}

\begin{lemma}
\label{lem:wellofbeta}
    The functor $\beta$ is well-defined.
\end{lemma}
\begin{proof}
By \cite{M}, the relation \eqref{webassoc} and \eqref{rung-sawp} implies the following relation
\begin{equation}
\begin{tikzpicture}[baseline = 2mm,scale=1,color=\clr]
	\draw[-,thick] (0,0) to (0,1);
	\draw[-,thick] (-.015,0) to (-0.015,.2) to (-.57,.4) to (-.57,.6)
        to (-.015,.8) to (-.015,1);
	\draw[-,line width=1.2pt] (-0.6,0) to (-0.6,1);
        \node at (-0.6,-.1) {$\scriptstyle a$};
        \node at (0,-.1) {$\scriptstyle b$};
        \node at (-0.3,.84) {$\scriptstyle d$};
        \node at (-0.3,.19) {$\scriptstyle c$};
\end{tikzpicture}
=
\sum_{t\ge 0}
\binom{b\! -\!a\!+\!d\!-\!c}{t}
\begin{tikzpicture}[baseline = 2mm,scale=1,color=\clr]
	\draw[-,line width=1.2pt] (0,0) to (0,1);
	\draw[-,thick] (-0.8,0) to (-0.8,.2) to (-.03,.4) to (-.03,.6)
        to (-.8,.8) to (-.8,1);
	\draw[-,thin] (-0.82,0) to (-0.82,1);
        \node at (-0.81,-.1) {$\scriptstyle a$};
        \node at (0,-.1) {$\scriptstyle b$};
        \node at (-0.4,.9) {$\scriptstyle c-t$};
        \node at (-0.4,.13) {$\scriptstyle d-t$};
\end{tikzpicture}.        
\end{equation}
    Thus, all relations of $\mathpzc{Web}$ hold in $\Qwb$ by \eqref{def-crossing-by-rung} and \cite[Remark 4.8]{BEEO}.
    In particular, \eqref{mergesplit}   also holds.
    The relations \eqref{wdotsmovecrossing}, \eqref{equ:onewdot} hold in $\Qwb$ by \cite[Lemma 4.3, Lemma 4.12]{BKu21}. 
Furthermore, the relation \eqref{wdot-move-split and merge} can be derived from \eqref{wdot-move-split and merge over C} by induction. This completes the proof that $\beta$ is well-defined.
\end{proof}

\subsection{$\QAW$ over $\C$}    
\begin{definition} 
\label{def-QAW-overC}Suppose $\kk=\C$.
   Let $\QAWC$ be the strict monoidal supercategory obtained from  $\qW$ by adjoining additional even generating morphisms $\dotgenC$,
   subject to the following additional relations \eqref{dotmovecrossingC}--\eqref{wdotbdot over C} for $a,b\in\mathbb{Z}_{\ge 1}$:
    	
\begin{equation}
\label{dotmovecrossingC}
\begin{tikzpicture}[baseline = 7.5pt, scale=0.3, color=\clr]
    \draw[-,thick] (0,-.2) to  (1,2.2);
    \draw[-,thick] (1,-.2) to  (0,2.2);
    \draw(0.2,1.6)\bdot;
    \node at (0, -.6) {$\scriptstyle 1$};
    \node at (1, -.6) {$\scriptstyle 1$};
\end{tikzpicture} 
    		=
\begin{tikzpicture}[baseline = 7.5pt, scale=0.3, color=\clr]
    			\draw[-, thick] (0,-.2) to (1,2.2);
    			\draw[-,thick] (1,-.2) to(0,2.2);
    			\draw(.8,0.3)\bdot;
    			\node at (0, -.6) {$\scriptstyle 1$};
    			\node at (1, -.6) {$\scriptstyle 1$};
\end{tikzpicture} 
    		+
\begin{tikzpicture}[baseline = 7.5pt, scale=0.3, color=\clr]
    \draw[-, thick] (0,.5) to (0,2.2);
    \draw[-, thick] (0,-.2) to (0,.5);
    \draw[-, thick]   (1,1.8) to (1,2.2); 
    \draw[-, thick] (1,1.8) to (1,-.2); 
    \node at (0, -.6) {$\scriptstyle 1$};
    \node at (1, -.6) {$\scriptstyle 1$};
\end{tikzpicture}
            -    		
\begin{tikzpicture}[baseline = 7.5pt, scale=0.3, color=\clr]
    \draw[-, thick] (0,.5) to (0,2.2);
    \draw[-, thick] (0,-.2) to (0,.5);
    \draw[-, thick]   (1,1.8) to (1,2.2); 
    \draw[-, thick] (1,1.8) to (1,-.2); 
    \node at (0, -.6) {$\scriptstyle 1$};
    \node at (1, -.6) {$\scriptstyle 1$};
    \draw(.,0.8)\wdot;
    \draw(1,0.8)\wdot;
\end{tikzpicture},
    		\qquad
\begin{tikzpicture}[baseline = 7.5pt, scale=0.3, color=\clr]
    \draw[-, thick] (0,-.2) to (1,2.2);
    \draw[-,thick] (1,-.2) to(0,2.2);
    \draw(.2,0.2)\bdot;
    \node at (0, -.6) {$\scriptstyle 1$};
    \node at (1, -.6) {$\scriptstyle 1$};
\end{tikzpicture}
    		= 
\begin{tikzpicture}[baseline = 7.5pt, scale=0.3, color=\clr]
    \draw[-,thick] (0,-.2) to  (1,2.2);
    \draw[-,thick] (1,-.2) to  (0,2.2);
    \draw(0.8,1.6)\bdot;
    \node at (0, -.6) {$\scriptstyle 1$};
    \node at (1, -.6) {$\scriptstyle 1$};
\end{tikzpicture}
    		+
\begin{tikzpicture}[baseline = 7.5pt, scale=0.3, color=\clr]
    \draw[-,thick] (0,-.2) to (0,2.2);
    \draw[-,thick] (1,2.2) to (1,-.2); 
    \node at (0, -.6) {$\scriptstyle 1$};
    \node at (1, -.6) {$\scriptstyle 1$};
\end{tikzpicture}
    		+
\begin{tikzpicture}[baseline = 7.5pt, scale=0.3, color=\clr]
    \draw[-, thick] (0,.5) to (0,2.2);
    \draw[-, thick] (0,-.2) to (0,.5);
    \draw[-, thick]   (1,1.8) to (1,2.2); 
    \draw[-, thick] (1,1.8) to (1,-.2); 
    \node at (0, -.6) {$\scriptstyle 1$};
    \node at (1, -.6) {$\scriptstyle 1$};
    \draw(.,0.8)\wdot;
    \draw(1,0.8)\wdot;
\end{tikzpicture}\:.
\end{equation}
    	 
\begin{equation}
\label{wdotbdot over C}
\begin{tikzpicture}[baseline = 1.5mm, scale=0.6, color=\clr]
    \draw[-,line width=1.2pt] (0,-0.1) to[out=up, in=down] (0,1.2);
    \draw(0,0.3) \bdot; 
    \draw(0,0.8) \wdot;
    \node at (0,-.35) {$\scriptstyle 1$};
\end{tikzpicture}
    =-~
\begin{tikzpicture}[baseline = 1.5mm, scale=0.6, color=\clr]
    \draw[-,line width=1.2pt] (0,-0.1) to[out=up, in=down] (0,1.2);
    \draw(0,0.3) \wdot; 
    \draw(0,0.8) \bdot;
    \node at (0,-.35) {$\scriptstyle 1$};
\end{tikzpicture}.
\end{equation}
    \end{definition}

\begin{theorem}
\label{equiv-over-C}
Suppose $\kk=\C$. Then
     $\QAW$ is isomorphic to $\QAWC$.
\end{theorem}

\begin{proof}
We define the following morphisms in $\QAWC$ which correspond to the generator 
$
\begin{tikzpicture}[baseline = 3pt, scale=0.5, color=\clr]
    \draw[-,line width=2pt] (0,0) to[out=up, in=down] (0,1.4);
    \draw(0,0.6) \bdot; 
    \draw (0.7,0.6) node {$\scriptstyle \omega_a$};
    \node at (0,-.3) {$\scriptstyle a$};
\end{tikzpicture} 
    	$, for $a\in\mathbb{N}$, in $\QAW$:
\begin{equation}
\label{intergralballon over C}
\begin{tikzpicture}[baseline = 3pt, scale=0.5, color=\clr]
    \draw[-,line width=2pt] (0,0) to[out=up, in=down] (0,1.4);
    \draw(0,0.6) \bdot; 
    \draw (0.7,0.6) node {$\scriptstyle \omega_a$};
    \node at (0,-.3) {$\scriptstyle a$};
\end{tikzpicture} 
    :=\frac{1}{a!} \ 
\begin{tikzpicture}[baseline = 1.5mm, scale=.7, color=\clr]
    \draw[-, line width=2pt] (0.5,2) to (0.5,2.5);
    \draw[-, line width=2pt] (0.5,0) to (0.5,-.4);
    \draw[-,thick]  (0.5,2) to[out=left,in=up] (-.5,1)to[out=down,in=left] (0.5,0);
    \draw[-,thick]  (0.5,2) to[out=left,in=up] (0,1)to[out=down,in=left] (0.5,0);      
    \draw[-,thick] (0.5,0)to[out=right,in=down] (1.5,1)to[out=up,in=right] (0.5,2);
    \draw[-,thick] (0.5,0)to[out=right,in=down] (1,1) to[out=up,in=right] (0.5,2);
    \node at (0.5,.5){$\scriptstyle \cdots\cdots$};
    \draw (-0.5,1) \bdot; 
    \draw (0,1) \bdot; 
    \node at (0.5,-.6) {$\scriptstyle a$};
    \draw (1,1) \bdot;
    \draw (1.5,1) \bdot; 
    \node at (-.22,0) {$\scriptstyle 1$};
    \node at (1.2,0) {$\scriptstyle 1$};
    \node at (.3,0.3) {$\scriptstyle 1$};
    \node at (.7,0.3) {$\scriptstyle 1$};
\end{tikzpicture}.
\end{equation}
 To show the isomorphism, we define two natural functors $\Psi:\QAWC\rightarrow\QAW$ and $\Phi:\QAW\rightarrow\QAWC$, matching the generating objects and generating morphisms in the same notation. Thus $\Phi\circ\Psi=id_{\QAW},\Psi\circ\Phi=id_{\QAWC}$.  Note that \eqref{dotmovecrossingC},\eqref{wdotbdot over C} are the special cases of \eqref{dotmovecrossing}, \eqref{wdotbdot}, thus $\Psi$ is well-defined. For the existence of $\Phi$, it suffices to verify that the relations \eqref{dotmovecrossing}, \eqref{bdotmove} and \eqref{wdotbdot} hold for $\QAWC$, which is given in Lemmas \ref{lem:2.12re}--\ref{lem:2.13re}.
\end{proof}
The following result is useful.
    \begin{proposition}
    	The following relation holds in $\QAWC$:
    	\begin{equation}			

    		=0,
    	$
    	where the first and third equality follows from \cite[Lemma 4.3]{BKu21}.
    \end{proof}

\begin{lemma}
\label{lem:2.12re}
    	The relation \eqref{dotmovecrossing} holds in $\QAWC$.
\end{lemma}
    
\begin{proof}
    We only prove the first relation in \eqref{dotmovecrossing}, the second holds by the symmetry $\div$. The case for $a=b=1$ is \eqref{dotmovecrossingC}. Suppose $b=1$. By the inductive hypothesis, we have  
\begin{equation}
\label{r=1}
\begin{aligned}
%
\:.
\end{align*}
    	\end{proof}

        \begin{lemma}
        \label{lem:2.13re}
       	The relations \eqref{bdotmove} and \eqref{wdotbdot} hold in $\QAWC$.
\end{lemma}

\begin{proof} 
	  	We proceed by induction on $a+b$. First, we have
\begin{equation}
\label{2-dots-split}
\begin{aligned}
%
.
\end{aligned}
\end{equation}
 Thus for the (first nontrivial) basic case when $a=b=1$, we have 
\begin{align*}
%
.
\end{align*}
	  	Suppose $a+b\ge 2$. Below we simplify the notation $%
   \quad \text{ by induction on $a+b-1$}\\
	  		&
	  		\overset{\eqref{webassoc}\eqref{dotmovecrossing}}{\underset{\eqref{wdotbdot over C}\eqref{intergralballon over C}}{=}}
	  		\frac{b}{r}~
%
 \; .
\end{align*}
	  	where $r=a+b$. This proves the first one in \eqref{bdotmove}, while the second holds by the symmetry $\div$. For \eqref{wdotbdot}, we have
\begin{align*}
%
.
\end{align*}
\end{proof}	
	
	
\section{Basis for the affine web category of type $Q$}
\label{CFD-diagram}	
In this section, we establish a basis theorem for the supercategory $\QAW$ by constructing a basis for each of its morphism spaces.

\subsection{The elementary chicken foot diagram for $\QAW$}
 A partition $\lambda=(\lambda_1,\lambda_2,\ldots,\lambda_k)$ is a composition such that $ \lambda_1\ge\lambda_2 \ge\ldots\ge  \lambda_{k}>0$. Let $\Par(m)$ denote the set of all partitions of $m$. Write
 $\Par=\cup_{m\in \N}\Par(m)$.
 
Let $ \text{Par}_a$ denote the set of partitions $\lambda$ such that $\lambda_i\le a$ for all $i$. For any $\lambda\in \Par_a$ with  length $k$, write 
\begin{equation}
\label{equ:defofbarlambda}
    \bar\lambda=\lambda-(1^k)=(\lambda_1-1,\lambda_2-1,\ldots,\lambda_k-1).
\end{equation}
For    $\lambda=(\lambda_1,\ldots,\lambda_k)$ and $\mu=(\mu_1,\ldots,\mu_t)$ in $\Par_a$, we introduce a shorthand notation for morphisms in $\End_{\QAW}(a)$ of the form, called {\em the elementary dot packet (of thickness $a$)},
\begin{equation}
\label{dots_data_simplify}
g_{\lambda,\mu}= \begin{tikzpicture}[baseline = -1mm,scale=1,color=\clr]
\draw[-,line width=1.5pt] (0.08,-.5) to (0.08,.5);
\node at (.08,-.7) {$\scriptstyle a$};
\draw(0.08,-0.2) \bdot;
\draw(.45,0.2)node {$\scriptstyle \omega^\circ_{\bar\lambda}$};
\draw(0.08,0.2) \bdot;
\draw(.45,-0.2)node {$\scriptstyle \omega_{\mu}$};
\end{tikzpicture}, \text{ where }
\begin{tikzpicture}[baseline = -1mm,scale=1,color=\clr]
\draw[-,line width=1.5pt] (0.08,-.5) to (0.08,.5);
\node at (.08,-.7) {$\scriptstyle a$};
\draw(0.08,0) \bdot;
\draw(.45,0)node {$\scriptstyle \omega^\circ_{\bar\lambda}$};
\end{tikzpicture}:= 
\omega^\circ_{a,\bar\lambda_1}\cdots\omega^\circ_{a,\bar\lambda_k} \text{ and }
\begin{tikzpicture}[baseline = -1mm,scale=1,color=\clr]
\draw[-,line width=1.5pt] (0.08,-.5) to (0.08,.5);
\node at (.08,-.7) {$\scriptstyle a$};
\draw(0.08,0) \bdot;
\draw(.45,0)node {$\scriptstyle \omega_{\mu}$};
\end{tikzpicture}
 := \omega_{a,\mu_1}\cdots \omega_{a,\mu_t}  
.
\end{equation}

A partition $\lambda$ with length $k$ is called strict if $\lambda_i>\lambda_{i+1}$ for all $1\le i\le k-1$.
Let $\SPar_a$ be the subset of $\Par_a$ consists of all strict partitions.



Let $\lambda,\mu\in \Lambda_{\text{st}}(m)$. A $\lambda\times \mu$ \emph{chicken foot diagram} (following \cite{BEEO}) is a  web diagram in $\Hom_{\QAW}(\mu,\lambda)$ consisting of three  horizontal  parts such that
\begin{itemize}
    \item the bottom part consists of only splits,
    \item the top part consists of only merges,
    \item the middle part consists of only crossings of the thinner strands (i.e., legs).
\end{itemize}
A chicken foot diagram (CFD) is called \emph{reduced} if there is at most one
intersection or one join between every pair of the legs.
\begin{definition} \label{def:diagram}
    For $\lambda,\mu\in \Lambda_{\text{st}}(m)$, a $\lambda\times \mu$ elementary chicken foot diagram (of type $Q$) is a $\lambda\times \mu$ reduced chicken foot diagram with an elementary dot packet $g_{\nu,\eta}$, for some $(\nu,\eta)\in \SPar_a\times\Par_a$, attached at the bottom of each leg with thickness $a$. 
\end{definition}

Just as a reduced chicken foot diagram of shape $A$ is encoded by $A\in \Mat_{\lambda,\mu}$, the elementary chicken foot diagrams of shape $A$ are encoded in the matrix $A$ enriched by certain bi-partitions. Denote  
\begin{equation}\label{dottedreduced}
 \PMat_{\lambda,\mu}:=\{(A, P)\mid A=(a_{ij})\in \Mat_{\lambda,\mu}, P=((\nu_{ij},\eta_{i,j})), \nu_{i,j}\in \SPar_{a_{i,j}}, \eta_{i,j}\in\Par_{a_{i,j}}\}.   
\end{equation}
Note that by the braid relation in \eqref{symmetric+braid}, all elementary diagram of the same type $(A,P)$ represent the same morphism in $\QAW$. 
We will identify the set of all elementary chicken foot diagrams from $\mu$ to $\lambda$ with $\PMat_{\lambda,\mu}$. In the following, the $(i,j)$-leg means the strands in the CFD corresponding to $a_{i,j}$. 

\begin{example}
Let $\lambda=(9,7),\mu=(5,5,6)$. The following is a $\lambda\times \mu$  elementary chicken foot diagram:
\begin{align}
    \label{ex-of-dotchickenfootd}
\begin{tikzpicture}[baseline = 19pt,scale=0.5,color=\clr,inner sep=0pt, minimum width=11pt]
   \draw[-,line width=.75mm] (-4,0) to (-4,-.3);
   \draw[-,line width=.75mm] (0,0) to (0,-.3);
   \draw[-,line width=.75mm] (4,0) to (4,-.3);
   \draw[-,line width=.75mm] (-2,5) to (-2,5.3);
   \draw[-,line width=.75mm] (2,5) to (2,5.3);
   \draw[-,line width=.25mm] (-4,0) to (-2,5);
   \draw[-,line width=.25mm] (-4,0) to (2,5);
   \draw[-,line width=.25mm] (0,0) to (-2,5);
   \draw[-,line width=.25mm] (0,0) to (2,5);
   \draw[-,line width=.25mm] (4,0) to (-2,5);
   \draw[-,line width=.25mm] (4,0) to (2,5);
   \draw(-3.5,1.25) \bdot;\node at (-4,1.25) {$\scriptstyle h_1$};
   \node at (-3,4) {$\scriptstyle 2$};
   \draw(-2.5,1.25) \bdot;\node at (-2,1.25) {$\scriptstyle f_1$};
   \node at (0.4,4) {$\scriptstyle 3$};
   \draw(-.5,1.25) \bdot;\node at (-1,1.25) {$\scriptstyle h_2$};
   \node at (-2,4) {$\scriptstyle 3$};
   \draw(.5,1.25) \bdot;\node at (1,1.25) {$\scriptstyle f_2$};
   \node at (1.8,3.8) {$\scriptstyle 2$};
   \draw(2.5,1.25) \bdot;\node at (2,1.25) {$\scriptstyle h_3$};
   \node at (-.4,4) {$\scriptstyle 4$};
   \draw(3.5,1.25) \bdot;\node at (4,1.25) {$\scriptstyle f_3$};
   \node at (2.8,4) {$\scriptstyle 2$};
\end{tikzpicture}
\end{align} 
where $h_i:=g_{\nu^{1,i},\eta^{1,i}},f_j:=g_{\nu^{2,j},\eta^{2,j}}, \nu^{i,j}\in\SPar_a,\eta^{i,j}\in\Par_a$ and the shape of this CFD is
$\begin{pmatrix} 
2&3&4\\
3&2&2
\end{pmatrix}.
$ 
\end{example}

\begin{proposition}
\label{prop:spanofaff}
    For any $\lambda,\mu\in \Lambda_{\text{st}}(m)$, $\Hom_\QAW(\mu,\lambda)$ is spanned by $\PMat_{\lambda,\mu}$ consisting of all $\lambda \times \mu$ elementary chicken foot diagrams from $\mu$ to $\lambda$.
\end{proposition}
\begin{proof}
We prove by induction on the degree $k$ for $\Hom_\QAW(\mu,\lambda)_{\le k}$. Note that the generators of $\QAW$  are themselves  elementary chicken foot diagrams. Moreover,  since the degree $0$ part is generated by those in \eqref{merge+split+crossing} and 
\begin{tikzpicture}[baseline = 3pt, scale=0.5, color=\clr]
\draw[-,line width=1.5pt] (0,0) to[out=up, in=down] (0,1.4);
\draw(0,0.6) \wdot; 
\node at (0,-.3) {$\scriptstyle a$};
\end{tikzpicture}, mimicking the argument as those in \cite{BEEO} for the web category, we have that the degree zero part  is spanned by the required (degree 0) elementary chicken foot diagram from $\mu$ to $\lambda$.

In general, 
it suffices to show that $fg$ can be expressed as a linear combination of elementary chicken foot diagrams up to lower degree terms for any $g \in \PMat_{\lambda,\mu}$ and specific generating morphism $f$ of the following four types:   
\begin{multicols}{2}
\begin{enumerate}
    \item[(1)]Type I:  $1_* \bdota 1_*$, 
       \item[(3)]Type  $\rot{Y}$: $1_*\merge 1_*$,
       \item[(2)]Type II: $1_* \wdota\: 1_*$, 
    \item[(4)]Type Y: $1_*\splits 1_*$.
\end{enumerate}
\end{multicols}
Here $1_*$ denotes suitable identity morphisms. 

In the following proof, we call $\bdota$ (resp. $\wdota$) a black dot 
(resp. a white  dot).
 We divided the proof according to the type of $f$. By Lemma \ref{dotmovefreely1}, we can move black dots and white dots through the crossing and from the thick strand to each of its legs up to lower degrees. Moreover, by Lemma \ref{dotmovefreely2}, we can exchange different $\omega_r$ and $\omega^\circ_r$ separately, and move $\omega_r$ below $\omega^\circ_r$ at the cost of generating additional terms of lower degrees. We will use this observation frequently below. 

(1) 
Suppose that $f$ is of type I with a black dot on the $i$th strand. Then $fg$ is obtained from $g$ by adding a black dot to the thick strand at the $i$th vertex. Suppose that $g$ has a $s$-fold merge at its $i$th vertex. Then we use \eqref{bdotmove} to move the black dot to each leg of the s-fold merge step by step and $fg$ is obtained from $g$ by adding a black dot on the top of each leg of the $s$-fold merge at its $i$th vertex. Note that adding a black dot on the top of the elementary dot packet $g_{\nu,\eta}$ may no longer be the elementary dot packet. By the \eqref{dots-ball} we have
$$
\begin{tikzpicture}[baseline = 1.5mm, scale=0.8, color=\clr]
	\draw[-,line width=1.5pt] (0,-0.4) to[out=up, in=down] (0,.9);
	\draw(0,.-.1) \bdot; 
        \draw(0,.3) \bdot; 
        \draw(0,0.7) \bdot;
	\node at (.4,-.1) {$\scriptstyle \omega_\eta$};
	\node at (.4,.3) {$\scriptstyle \omega^\circ_{\bar\nu}$};
        \node at (0,-.55) {$\scriptstyle r $};
\end{tikzpicture}
=\epsilon\:
\begin{tikzpicture}[baseline = 1.5mm, scale=0.8, color=\clr]
	\draw[-,line width=1.5pt] (0,-0.4) to[out=up, in=down] (0,.9);
        \draw(0,0) \bdot; 
        \draw(0,0.5) \bdot;
	\node at (.4,-.1) {$\scriptstyle \omega_{\eta'}$};
	\node at (.4,.5) {$\scriptstyle \omega^\circ_{\bar\nu}$};
        \node at (0,-.55) {$\scriptstyle r $};
\end{tikzpicture}
$$
where $\epsilon\in \left\{\pm1\right\}$ and $\eta'=(r,\eta_1,\ldots,\eta_n)$ are partitions by adding a $r$ to $\eta$. Then $g_{\nu,\eta'}$ is the elementary dot packet and $fg$ is also a elementary chicken foot diagrams. This completes the proof for $f$ of type I.

(2) 
Suppose that $f$ is of type II with a 
white dot on the $i$th strand. Then $fg$ is obtained from $g$ by adding a white dot to the thick strand at the $i$th vertex. Suppose that $g$ has a $s$-fold merge at its $i$th vertex. By \eqref{wdot-move-split and merge}, we also have 
$$
fg= \sum\limits_{1\le j\le s} g_j, 
$$
where $g_j$ is obtained from g by adding a white dot on the top of $j$th leg of the $s$-fold merge at its $i$th vertex. By \eqref{dots-ball} and \eqref{doublewdot}, adding a white dot to the top of any elementary dot packet is also an elementary dot packet. Thus $fg$ is the sum of the elementary chicken foot diagrams.

(3)
Suppose that $f$ is of type $\rot{Y}$, a two-fold merge joining to the $i$th and $(i+1)$th strands at the top of $g$. Suppose that $f$ is of type $\rot{Y}$, a two-fold merge joining to the $i$th and $(i+1)$th strand at the top of $g$. Then by \eqref{webassoc} $fg$
is a chicken foot diagram obtained from $g$ by merge the $s$-fold merge and the $s$-fold merge to a $(s+t)$-fold merge.
The resulting chicken foot diagram is not generally reduced since two legs may join twice. When this occurs, we can use \eqref{webassoc}, \eqref{swallows} and \eqref{symmetric+braid} to transform it into a CFD-like diagram, but some legs will become double legs (ref. \cite[Proposition 3.6]{SW1}), then by Lemma \ref{doublestring-g} and Lemma \ref{strictpartition}, we can combine the two legs with elementary dot packets and then make $fg$ into a linear combination of elementary chicken foot diagrams up to lower degree terms. 

(4)
Suppose that $f$ is of type Y, a 2-fold split joining to the $i$th vertex at the top of $g$. Suppose $g$ has a $h$-fold merge in the $i$th vertex. Then we first use \eqref{webassoc}, \eqref{mergesplit} and \eqref{sliders} to rewrite the composition of the split in $f$ and the merge of at the $i$th vertex of $g$ as a sum of reduced chicken foot diagrams. The resulting diagrams may not be reduced since there are new splits in the top part. Next, we use \eqref{sliders} to pull each new split down until it meets the elementary dot packet at the bottom. Then we use \eqref{dotmovemerge+high} and \eqref{ballsplit} to move $\wkdota$ and $\rcircdot$ step by step through the split from below (modulo lower degree terms). In the resulting diagrams, each leg contains composites of $\omega_r$ and $\omega^\circ_s$. Finally,  by \eqref{extrarelation}, \eqref{wdotballcommute} and Lemma \ref{strictpartition}, we can express this dot packet as the sum of some elementary dot packet module lower degree terms. This completes the proof for $f$ of type Y.  
\end{proof}

\subsection{A representation of $\QAW$}
In this subsection, we always assume that $\kk=\C$.	We shall construct a representation of $\QAW$ on a module category of the queer Lie superalgebra $\mathfrak{q}_n$.

	Fix $n\in \N$ and let $I=\{1,2,\ldots, n, \bar 1, \bar 2,\ldots,\bar n\}$. 
    Fix a superspace $V=V_{\bar{0}}\oplus V_{\bar{1}}$ with $\dim_\C V_{\bar{0}}= \dim_\C V_{\bar{1}}=n$ and  a homogeneous basis $\{v_i\mid i\in I\}$ with $|v_i|=\bar{0}$ and $|v_{\bar{i}}|=\bar{1}$ for $i\in I_0:=\{1,\ldots,n\}$.  We adopt the convention that $\bar{\bar{i}}=i$ for all $i\in I$. Let $c: V \rightarrow V$ denote the odd linear map given by $c(v_i)=(-1)^{|v_i|}\sqrt{-1}v_{\bar{i}}$ for all $i \in I$.
	
    
	The Lie superalgebra  $\mathfrak{q}_n$ of type $Q$ is the Lie subsuperalgebra of $\mathfrak {gl}_{n|n}=\mathfrak{gl}(V)$ given by $$
	\mathfrak{q}_n=\left\{x\in \mathfrak{gl}(V)|[x,c]=0\right\}.
	$$
	Then $\mathfrak{q}_n$
    consists of matrix of form $$
	\left\{\left(\begin{matrix}
		A&   B\\
		B&   A\\
	\end{matrix}
	\right)	\right\}
	$$
 and $\mathfrak q_n$   has a homogenous basis given by $e^{\bar{0}}_{i,j}:=e_{i,j}+e_{\bar{i},\bar{j}}$ and $e^{\bar{1}}_{i,j}:=e_{\bar{i},j}+e_{i,\bar{j}}$ for $1<i,j<n$, where 
    $e_{a,b}\in\glv$ is the linear map $e_{a,b}(v_k)=\delta_{b,k}v_a$ with  $|e_{a,b}|=|v_a|+|v_b|$, $a,b\in I$. Set $f^{\bar{0}}_{i,j}:=e_{i,j}-e_{\bar{i},\bar{j}}$ and $f^{\bar{1}}_{i,j}:=e_{\bar{i},j}-e_{i,\bar{j}}$ for $1<i,j<n$. These are homogeneous elements of $\glv$ and, together with our basis for $\mathfrak{q}_n$, provide a homogeneous basis for $\glv$.
     Let $U(\mathfrak{gl})$ (resp. $U(\mathfrak{q})$) denote the enveloping superalgebra of $\glv$ (resp. $\mathfrak{q}_n$).
There is a triangular decomposition 
$U(\mathfrak q)=U(\mathfrak n^-)\otimes U(\mathfrak h)\otimes U(\mathfrak n)$, where $U(\mathfrak h)$, $U(\mathfrak n)$ and $U(\mathfrak n^-)$ respectively, are the subsuperalgebras generated by 
\[\{e^\epsilon_{ii}\mid i\in I_0,\epsilon\in \Z_2\}, \quad \{ e^\epsilon_{ij}\mid 1\le i<j\le n, \epsilon\in\Z_2\}, \quad \{e^\epsilon_{ij}\mid 1\le j<i\le n, \epsilon\in\Z_2 \}.\]

	The tensor superalgebra is the tensor algebra $T(V)$ regarded as a superalgebra with the induced grading. As a quotient of $T(V)$, we have the supersymmetric algebra, namely
	$$
	S(V)=T(V)/\langle v\otimes w-(-1)^{|v||w|}w\otimes v | \text{for all homogeneous vectors } v,w\in V\rangle.
	$$
	Let $\mathfrak q_n\text{-Mod}_S$ denote the monoidal supercategory generated by the supersymmetric power  $S^k(V)$ of natural $\mathfrak{q}_n$-supermodule $V$, $k\in N$.
Recall $\mathfrak{q}\text{-}\mathbf{Web}$ 
in Definition \ref{defofbku}.

 \begin{proposition}[{\cite[Proposition 5.3]{BKu21}}]\label{functorofqweb}
		For every $n \geq 1$ there exists an essentially surjective functor of monoidal supercategories, 
		$$
		\Psi: \mathfrak{q}\text{-}\mathbf{Web}\to\mathfrak{q}_n\text{-}\mathbf{Mod_S},
		$$
		given on objects by 
		\[
		\Psi(\lambda)= S^{\lambda}(V) := S^{\lambda_{1}}(V) \otimes\dotsb \otimes S^{\lambda_{t}}(V),
		\] where $\lambda = (\lambda_{1}, \dotsc , \lambda_{t})$.  On morphisms the functor is defined by sending  $\odota$, $\merge$, and $\splits$, respectively, to the following maps:
		\begin{align*}
			S^{a}(V) &\to S^{a}(V), & v_{1}\dotsb v_{a}  \mapsto  \sum_{t=1}^{a}&(-1)^{\p{v_{1}}+\dotsb +\p{v_{t-1}}} v_{1} \dotsb c(v_{t}) \dotsb v_{a}; \\
			S^{a}(V) \otimes S^{b}(V) &\to S^{a+b}(V), & v_{1}\dotsb v_{a}\otimes w_{1}\dotsb w_{b} &\mapsto v_{1}\dotsb v_{a}w_{1}\dotsb w_{b};\\
			S^{a+b}(V) &\to S^{a}(V) \otimes S^{b}(V), & v_{1} \dotsb v_{a+b} \mapsto & \sum (-1)^{\varepsilon_{I,J}} v_{i_{1}}\dotsb v_{i_{a}} \otimes v_{j_{1}}\dotsb v_{j_{b}}.
		\end{align*}
		The last summation is over all $I=\{i_{1}< \dotsb < i_{a} \}$ and $ J=\{j_{1}< \dotsb < j_{b} \}$ such that $I \cup J = \left\{ 1, \dotsb ,a+b  \right\}$.  The element   $\varepsilon_{I,J} \in \Z_{2}$ is defined by the formula $v_{i_{1}}\dotsb v_{i_{a}}v_{j_{1}}\dotsb v_{j_{b}} = (-1)^{\varepsilon_{I,J}} v_{1}\dotsb v_{a+b}$ in $S^{a+b}(V)$.
	\end{proposition}

\begin{rem}
  	\label{rem-crossing}
  	By \ref{crossgen}, one can check that $$\Psi\big(\crossing \big)(v_{i_1}\cdots v_{i_a}\otimes w_{i_1}\cdots w_{i_b})= (-1)^{(\p{v_{i_1}}+\cdots+\p{v_{i_a}})(\p{w_{j_1}}+\cdots+\p{w_{j_b}})} w_{i_1}\cdots w_{i_b}\otimes v_{i_1}\cdots v_{i_a}$$ by induction on $a+b$. 
  \end{rem}
    
For any Lie superalgebra $\mathfrak{g}$, let $\End(\mathfrak{g}\text{-smod})$ be the supercategory whose objects are all endofunctors of $\mathfrak{g}\text{-smod}$ with natural transformations as morphisms. It is a strict monoidal supercategory with identity functor $id$ as the unit object. The tensor product of two functors $F$ and $G$ is given by composition, $i.e., F\otimes G:= G\circ F$.

    Define the following element:
	$$
	\Omega=\sum_{1\le i,j\le n}-e^{\bar{0}}_{i,j}\otimes f^{\bar{0}}_{j,i}+e^{\bar{1}}_{i,j}\otimes f^{\bar{1}}_{j,i} \in U(\mathfrak{q})\otimes U(\mathfrak{gl}).
	$$
	Note that this is $-\Omega$ in \cite[(7.2.1)]{HKS09}. Given a $U(\mathfrak{q})$ -supermodule $M$, we have an even linear map $M\otimes S^\lambda(V) \rightarrow M\otimes S^\lambda(V)$ given by$$
	\Omega.(m\otimes v)=\sum_{1\le i,j\le n}-e^{\bar{0}}_{i,j}m\otimes f^{\bar{0}}_{j,i}v+(-1)^{|m|}e^{\bar{1}}_{i,j}m\otimes f^{\bar{1}}_{j,i}v,
	$$ for all homogeneous $m\in M$ and $v \in S^\lambda(V)$.
	It follows from  \cite[Theorem 7.4.1]{HKS09} that 	 $\Omega$ gives an even $U(\mathfrak{q})$-supermodule homomorphism.

	\begin{proposition}
    \label{functorofaff}
		Suppose $\kk=\C$. There is a strict monoidal functor $\mathcal F:\QAW\rightarrow \End(\mathfrak {q}_n\text{-smod})$ sending the generating object $a$ to the functor $-\otimes S^a(V)$ such that  $\mathcal F(y)_M=\text{Id}_M\otimes\Psi(y)$  for $y\in\{ \splits, \merge, \crossing, \odota\}$, and  $\mathcal F\big(\dotgen\big)_M =\Omega$,
		where $\Psi$ is the functor given in Proposition~\ref{functorofqweb}.   
	\end{proposition}
 \begin{proof}
  It follows from Proposition~\ref{functorofqweb} that $\mathcal F\big(\splits\big), \mathcal F\big(\merge\big), \mathcal F\big(\crossing\big),  \mathcal F\big(\odota\big)$ satisfy  relations  \eqref{webassoc},\eqref{equ:r=0},\eqref{doublewdot} and  \eqref{wdot-move-split and merge over C}--\eqref{wdot-rung-relation_2}. Thus,
 it remains to check the relations \eqref{dotmovecrossingC}--\eqref{wdotbdot over C} according to the equivalent  definition of $\QAW$ over $\kk=\C$ in \S \ref{affineweboverC}.
 By remark \ref{rem-crossing}, we have 
    \[
    \mathcal F\big(
    \begin{tikzpicture}[baseline = 7.5pt, scale=0.4, color=\clr]
    	\draw[-, line width=1pt] (0,-.2) to (1,1.6);
    	\draw[-,thick] (1,-.2) to (0,1.6);
    	\node at (0, -.5) {$\scriptstyle 1$};
    	\node at (1, -.5) {$\scriptstyle 1$};
    \end{tikzpicture}
    \big)_M(m\otimes v_t\otimes v_h)
    =(-1)^{\p{t}\p{h}}m\otimes v_h\otimes v_t
    \text{ and}\]
    \begin{align*}
    	\mathcal F\big(
    	\begin{tikzpicture}[baseline = 7.5pt, scale=0.4, color=\clr]
    		\draw[-, line width=1pt] (0,-.2) to (0,1.6);
    		\draw[-,thick] (.8,-.2) to (.8,1.6);
    		\draw(.8,0.5)\bdot;
    		\node at (0, -.5) {$\scriptstyle 1$};
    		\node at (.8, -.5) {$\scriptstyle 1$};
    	\end{tikzpicture}
    	\big)_M(m\otimes v_t\otimes v_h)
    	&
    	=
    	\sum_{1\le i,j\le n}-e^{\bar{0}}_{i,j}(m\otimes v_t)\otimes f^{\bar{0}}_{j,i}v_h+(-1)^{|m|+|v_t|}e^{\bar{1}}_{i,j}(m\otimes v_t)\otimes f^{\bar{1}}_{j,i}v_h
    	\\
    	&=
    	\sum_{1\le i,j\le n}-e^{\bar{0}}_{i,j}m\otimes v_t\otimes f^{\bar{0}}_{j,i}v_h+(-1)^{|m|+|v_t|}e^{\bar{1}}_{i,j}m\otimes v_t\otimes f^{\bar{1}}_{j,i}v_h
    	\\
    	&~+
    	\sum_{1\le i,j\le n}-m\otimes e^{\bar{0}}_{i,j}v_t\otimes f^{\bar{0}}_{j,i}v_h+(-1)^{|v_t|}m\otimes e^{\bar{1}}_{i,j} v_t\otimes f^{\bar{1}}_{j,i}v_h\\
        &= \mathcal F\big(
      \begin{tikzpicture}[baseline = 7.5pt, scale=0.4, color=\clr]
    	 \draw[-, line width=1pt] (0,-.2) to (1,1.6);
    	 \draw[-,thick] (1,-.2) to (0,1.6);
        	\node at (0, -.5) {$\scriptstyle 1$};
    	  \node at (1, -.5) {$\scriptstyle 1$};
      \end{tikzpicture}\big)_M \circ
      \mathcal F\big(
      \begin{tikzpicture}[baseline = 7.5pt, scale=0.4, color=\clr]
      	\draw[-, line width=1pt] (0,-.2) to (1,1.6);
      	\draw[-,thick] (1,-.2) to (0,1.6);
      	\draw(.25,1.1)\bdot;
      	\node at (0, -.5) {$\scriptstyle 1$};
      	\node at (1, -.5) {$\scriptstyle 1$};
      \end{tikzpicture}
      \big)_M(m\otimes v_t\otimes v_h)\\
      &+\mathcal F\big(
     \begin{tikzpicture}[baseline = 7.5pt, scale=0.4, color=\clr]
     	\draw[-, line width=1pt] (0,-.2) to (1,1.6);
     	\draw[-,thick] (1,-.2) to (0,1.6);
     	\node at (0, -.5) {$\scriptstyle 1$};
     	\node at (1, -.5) {$\scriptstyle 1$};
     \end{tikzpicture}\big)_M \circ
     \mathcal F\big(
     \begin{tikzpicture}[baseline = 7.5pt, scale=0.4, color=\clr]
     	\draw[-, line width=1pt] (0,-.2) to (0,1.6);
     	\draw[-,thick] (.8,-.2) to (.8,1.6);
     	\draw(.8,0.5)\wdot;
     	\draw(.0,0.5)\wdot;
     	\node at (0, -.5) {$\scriptstyle 1$};
     	\node at (.8, -.5) {$\scriptstyle 1$};
     \end{tikzpicture}
     -
     \begin{tikzpicture}[baseline = 7.5pt, scale=0.4, color=\clr]
     	\draw[-, line width=1pt] (0,-.2) to (0,1.6);
     	\draw[-,thick] (.8,-.2) to (.8,1.6);
     	\node at (0, -.5) {$\scriptstyle 1$};
     	\node at (.8, -.5) {$\scriptstyle 1$};
     \end{tikzpicture}
     \big)_M(m\otimes v_t\otimes v_h).
    \end{align*}
  This yields the first relation in  \eqref{dotmovecrossingC}.  The second one   is similar.
   Finally,   the relation \eqref{wdotbdot over C} is already verified in \cite[Theorem 7.4.1]{HKS09}.
 \end{proof}
\begin{rem}
  	\label{remactfundot}
  	One can check that $\mathcal F\big(\wxdota \big)_M= \Omega(M\otimes S^a V)$, for $a\ge 1$. 
  \end{rem}
 
    By evaluation at any module $M\in \mathfrak q_n$-smod, we obtain a functor $\mathcal F_M: \QAW\rightarrow \mathfrak q_n$-smod. We specialize $M$ to be the generic Verma module $M^{\text{gen}}$ below. In this case, we shall work out the leading terms of the action of various morphisms on $M^{\text{gen}}\otimes S^a (V).$
 
    Let $U(\mathfrak h_{\bar{0}})$ be subsuperalgebra of $U(\mathfrak h)$ generated by $\left\{h_i|i=1,\ldots,n\right\}$, where for brevity we set $h_i=e^{\bar{0}}_{i,i}, \bar{h}_i=e^{\bar{1}}_{i,i}$ for $i=1,\ldots,n$. The usual $\Z$-grading on $U(\mathfrak{h}_{\bar{0}})$ is given by putting each $h_i$ in degree 1. Let $U(\mathfrak b)=U(\mathfrak h)\otimes U(\mathfrak n)$, then $U(\mathfrak h)$ is the $U(\mathfrak b)$-supermodule by inflation.
 
   We define the generic Verma module 
 \begin{equation}\label{defofgenericverma}
 	M^{\text{gen}}:= U(\mathfrak q)\otimes _{U(\mathfrak b)}U(\mathfrak h) \cong U(\mathfrak q)\otimes _{U(\mathfrak b)}(U(\mathfrak h)\otimes _{U(\mathfrak h_{\bar{0}})}U(\mathfrak h_{\bar{0}})).
 \end{equation} Let $\{f_1,\ldots,f_N\}$ (resp.$\{\bar{f}_1,\ldots,\bar{f}_N\}$) be the basis for $\mathfrak n^-_{\bar{0}}$(resp.$\mathfrak n^-_{\bar{1}}$) such that
 $$
 \left\{e^{\bar{0}}_{i,j}|1\le j< i\le n\right\}=\left\{f_1,\ldots,f_N\right\} \text{ and } \left\{e^{\bar{1}}_{i,j}|1\le j< i\le n\right\}=\left\{\bar{f}_1,\ldots,\bar{f}_N\right\}.
 $$ 
 By PBW theorem $M^{\text{gen}}$ is a free  right $U(\mathfrak h_{\bar{0}})$-module with basis  
 \begin{equation}
 	f^{a_1}_1\cdots f^{a_N}_N\bar{f}^{b_1}_1\cdots \bar{f}^{b_N}_N\otimes \bar{h}^{c_1}_1\cdots\bar{h}^{c_n}_n\otimes 1
 \end{equation}
    where each $a_k \in \Z_{\ge 0}$ and $b_k,c_k \in \left\{0,1\right\}$. Assigning $\deg h_i=1$ and $\deg f=\deg \bar{h}_i=0$ for $i=1,\ldots,n$ and $f\in \mathfrak n^-$ provides a $\Z$-grading on $M^{\text{gen}}$. This induces a $\Z$-grading on $M^{\text{gen}}\otimes S^\lambda (V)$ by further assigning degree 0 to any element in $S^\lambda (V).$
    
For simplicity, we write 
$v_\mathbf i:= v_{i_1}v_{i_2}\ldots v_{i_a} \in S^a(V)$ for any  $\mathbf i=(i_1,i_2,\ldots,i_a)$.  For any $i\in I$, we define  $i^*:=\bar{i}$ if $i\in \left\{\bar{1},\ldots,\bar{n}\right\}$ and $i^*:=i$ if $i\in \left\{1,\ldots,n\right\}$.
    \begin{lemma}
 	\label{lem:actionofdotlead}
 	The action of $\mathcal F_{M^{\text{gen}}} (\wxdota)$
 	on $M^{\text{gen}}\otimes S^a (V)$ is of filtered degree 1. Moreover, denote $u\otimes h=1\otimes 1\otimes h \in M^{\text{gen}}$ for any $h\in U(\mathfrak h_{\bar{0}})$, up to lower degree terms, we have 
 	\begin{equation*}
 		\wxdota(u\otimes h\otimes v_\mathbf i)
 		\equiv u\otimes\sum_{1\le j\le a}(-1)^{|i_j|+1}h_{i^*_j}h\otimes v_\mathbf i,
 	\end{equation*}  
 	for any $h\in U(\mathfrak h_{\bar{0}})$ and $\mathbf i=(i_1,\ldots,i_a)$.
 \end{lemma}
\begin{proof}
 Note that the linear map defined by left action of  $ e^\epsilon_{ij}, \epsilon\in\Z_2 $ on $M^{\text{gen}}$ is of filtered degree 1 (and respectively, of degree 0) for $i\le j$ (and respectively, $i>j$) by \cite[Lemma 5.1]{BCK19}. This together with Remark \ref{remactfundot} implies the first statement. On the other hand, we have
 	$$
 	\begin{aligned}
 		\wxdota(u\otimes h\otimes v_{\mathbf i})
 		&=\Omega(u\otimes h\otimes v_{\mathbf i} )=\sum_{1\le i,j \le n}-e^{\bar{0}}_{ij}\left(u\otimes h\right)\otimes f^{\bar{0}}_{ji} \left(v_{\mathbf i}\right)
   +(-1)^{\p{u}}e^{\bar{1}}_{ij}\left(u\otimes h\right)\otimes f^{\bar{1}}_{ji}\left(v_{\mathbf i}\right)\\
 		&\equiv \sum_{1\le i\le n} -e^{\bar{0}}_{ii}\left(u\otimes h\right)\otimes f^{\bar{0}}_{ii} \left(v_{\mathbf i}\right)\\
 		&=\sum_{j=1}^{a}\sum_{1\le i\le n}-e^{\bar{0}}_{i,i}\left(u\otimes h\right)\otimes v_{i_1} \ldots \left(\delta_{i_j,i}v_{i_j}-\delta_{i_j,\bar{i}}v_{\bar{i_j}}\right)\ldots v_{i_a}
 		\\
 		&= u\otimes\sum_{1\le j\le a}(-1)^{|i_j|+1}h_{i^*_j}h\otimes v_{\mathbf i}.
 	\end{aligned}  $$
    The lemma is proved.
 \end{proof}

   For brevity, set $y_i=(-1)^{\p{i}+1}h_{i^*}$ for any $i\in I$.
 \begin{corollary}
   	\label{actionofelempoly}
   	For any $\nu\in \Par_a$ and $\mathbf i=(i_1,\ldots,i_a)$, we have 
   	\[
   	\omega_{a,\nu} (u\otimes h\otimes v_{\mathbf i})
   	\equiv
   	\big(u\otimes e_{\nu}(y_{i_1},\ldots, y_{i_a})h \big) \otimes 
   	v_{\mathbf i},
   	\]
   	where $e_{\nu}(x_1,\ldots, x_a)$ is the elementary symmetric polynomial associated with partition $\nu$.
   \end{corollary}
\begin{proof}
		The proof is similar to \cite[Lemma 3.11]{SW1}.
\end{proof}

For any $\mathbf i\in I^a$ and $1\le t\le a$, write 
$\mathbf i_t=(i_1, \ldots, i_{t-1}, \hat i_t, i_{t+1},\ldots,i_a)\in I^{a-1}$, where $\hat i_t$ means $i_t$ is omitted.
\begin{corollary}
\label{action-one-ball} 
For any $0\le r\le a-1$, we have
\begin{equation*}
   		\omega^\circ_{a,r}
   		(u\otimes h\otimes v_{\mathbf i})
   		\equiv 
   		\sqrt{-1}\sum_{t=1}^{a} u\otimes (-1)^{\p{i_1}+\ldots+\p{i_t}}e_r(y_{i_1},\ldots,\widehat{y_{i_t}},\ldots, y_{i_a})h \otimes 
   		v_{i_1} \ldots v_{\bar{i_t}} \ldots v_{i_a},
   	\end{equation*}
   	where $\widehat{y_{i_a}}$ means that this variable is missing.
   \end{corollary}
\begin{proof} We have 
$$
\begin{aligned}
\omega^\circ_{a,r}
        (u\otimes h\otimes v_{\mathbf i})
&=
\begin{tikzpicture}[baseline = -.5mm, scale=1.2, color=\clr]
   			\draw[-,line width=1pt](-.15,-.2) to (-.15,0);
                \draw[-,line width=1pt](.15,-.2) to (.15,0);
   			\draw[-,line width=1pt] (-.15,0) to [out=90,in=-180] (0,.25);
   			\draw[-,line width=1pt] (.15,0) to [out=90,in=0] (0,.25);
   			\draw[-,line width=2pt](0,.45) to (0,0.25);
   			\draw (-0.15,0) \bdot;
   			\draw (0.15,0) \wdot;
   			\draw (-0.4,0) node{$\scriptstyle r$};
   			\node at (-.15,-.3) {$\scriptstyle a-1$};
                \node at (.15,-.3) {$\scriptstyle 1$};
\end{tikzpicture}
        (\sum^{a}_{t=1}\epsilon_t u\otimes h\otimes v_{\mathbf i_t}\otimes v_{i_t})
        \\
        & \equiv
\begin{tikzpicture}[baseline = -.5mm, scale=1.2, color=\clr]
   			\draw[-,line width=1pt](-.15,-.2) to (-.15,0);
                \draw[-,line width=1pt](.15,-.2) to (.15,0);
   			\draw[-,line width=1pt] (-.15,0) to [out=90,in=-180] (0,.25);
   			\draw[-,line width=1pt] (.15,0) to [out=90,in=0] (0,.25);
   			\draw[-,line width=2pt](0,.45) to (0,0.25);
   			\draw (0.15,0) \wdot;
   			\node at (-.15,-.3) {$\scriptstyle a-1$};
                \node at (.15,-.3) {$\scriptstyle 1$};
   	\end{tikzpicture}
        (\sum^{a}_{t=1}\epsilon_t u\otimes e_r(y_{i_1},\ldots,\widehat{y_{i_t}},\ldots, y_{i_a})h\otimes v_{\mathbf i_t}\otimes v_{i_t}) \quad \text{ by Cor. \ref{actionofelempoly}} \\
&=
\begin{tikzpicture}[baseline = -.5mm, scale=1.2, color=\clr]
   			\draw[-,line width=1pt](-.15,-.2) to (-.15,0);
                \draw[-,line width=1pt](.15,-.2) to (.15,0);
   			\draw[-,line width=1pt] (-.15,0) to [out=90,in=-180] (0,.25);
   			\draw[-,line width=1pt] (.15,0) to [out=90,in=0] (0,.25);
   			\draw[-,line width=2pt](0,.45) to (0,0.25);
   			\node at (-.15,-.3) {$\scriptstyle a-1$};
                \node at (.15,-.3) {$\scriptstyle 1$};
   	\end{tikzpicture}
        (\sum^{a}_{t=1}\sqrt{-1}(-1)^{\p{i_1}+\ldots+\p{i_a}+\p{i_t}}\epsilon_t u\otimes e_r(y_{i_1},\ldots,\widehat{y_{i_t}},\ldots, y_{i_a})h\otimes v_{\mathbf i_t}\otimes v_{\bar{i_t}})
\\
&=
      \sum^{a}_{t=1}\sqrt{-1}(-1)^{\p{i_1}+\ldots+\p{i_t}}u\otimes e_r(y_{i_1},\ldots,\widehat{y_{i_t}},\ldots, y_{i_a})h\otimes v_{i_1}\ldots v_{\bar{i_t}}\ldots v_{i_a},
\end{aligned}
$$
where $\epsilon_t=(-1)^{\p{i_t}(\p{i_{t+1}}+\ldots+\p{i_a})}$.
\end{proof}
For any $k\in \Z_{\ge1}$, let 
$\SPar_{a,k}=\{\lambda\in \SPar_a\mid l(\lambda)=k\}$. Recall $\bar\lambda$ in \eqref{equ:defofbarlambda}.
Using  Corollary \ref{action-one-ball} repeatedly,   the action of $\omega^\circ _{\bar\lambda} $ for any $\lambda\in \SPar_{a,k}$ can be written of the following form  
\begin{equation}
\label{dots-change +-}
\begin{tikzpicture}[baseline = 3pt, scale=0.5, color=\clr]
		\draw[-,line width=1.5pt] (0,0) to[out=up, in=down] (0,1.4);
		\draw(0,0.6) \bdot; 
		\draw (0.7,0.6) node {$\scriptstyle \omega^\circ_{\bar\lambda}$};
		\node at (0,-.3) {$\scriptstyle a$};
\end{tikzpicture}
        (u\otimes h\otimes v_{\mathbf i})
        \equiv \sum_{\mathbf j }u\otimes F_\mathbf j h\otimes v_\mathbf j,
\end{equation}
where the sum runs over all sequence $\mathbf j=(j_1,\ldots,j_a)\in I^a$ with $j_t\in\{i_t,\bar{i_t}\}$ for all $1\le t\le a$ such that there are  at most $k$ numbers of  $j_t$'s belonging  to $\{\bar{i_1},\ldots,\bar{i_a}\}$  and $F_\mathbf{j}\in \mathbb{Z}[y_1,\ldots,y_a]$ with degree $|\bar\lambda|$.

Define
\begin{equation}
\label{equ:defofrhok}
  \rho_k=(k-1,k-2,\ldots,1,0). 
\end{equation} 

For any subset $S\subseteq I_k:=\{1,\cdots,k\}$ with $|S|=s$, let $S=\{c_1,\cdots,c_s\}$, where $c_1<c_2<\ldots <c_s$, $\mathfrak{S}_S:=\{(\sigma_1,\ldots,\sigma_s)|\{\sigma_1,\ldots,\sigma_s\}=S\}$, $ \epsilon_\sigma:=\sum_{i=1}^{s}\sharp\{t|\sigma_t<\sigma_i,t<i\}$, $y_S=\{y_{i_j}|j\in S\}$ and $y_{S^c}$ is the complementary set of $y_S$ in $\{y_{i_1},\ldots,y_{i_a}\}$. For any $\sigma\in \mathfrak{S}_S$, write   $$
  f_\sigma=\prod_{i=0}^{s-1} e_{\bar\lambda_{k-i}}(-y_{i_{\sigma_1}},\cdots,-y_{i_{\sigma_i}},\widehat{y_{i_{\sigma_{i+1}}}},y_{i_{\sigma_{i+2}}},\ldots,y_{i_{\sigma_{s}}},y_{i'_1},\ldots,y_{i'_{a-s}}),
  $$
where  $\{y_{i'_1},\ldots,y_{i'_{a-s}}\}:=\{y_{i_j}|j=1,\ldots,a\}-y_S$.
\begin{proposition}
\label{algorithm}
   For $\mathbf j=(\bar{i_1},\ldots,\bar{i_k},i_{k+1},\ldots,i_a)$ in \eqref{dots-change +-}  we have
\begin{equation}
\label{action-w0lambda}
F_{\mathbf j}=(\sqrt{-1})^k\big(\prod_{t=1}^k(-1)^{\p{i_1}+\ldots+\p{i_t}}\big)\Delta(y_{i_1},\ldots,y_{i_k})g_\lambda\:,
\end{equation}
where $\Delta(y_{i_1},\ldots,y_{i_k})=\prod_{1\le a<b\le k}(y_{i_a}-y_{i_b})$ and
\begin{equation}
\label{lead-terms-glambda}
   g_\lambda=e_{\bar\lambda-\rho_k}(y_{i_{k+1}},\ldots,y_{i_a})+\sum_{\mu\rhd\bar\lambda-\rho_k}d_\mu e_\mu(y_{i_{k+1}},\ldots,y_{i_a}),  \text{ for some }d_\mu\in\mathbb Z.
\end{equation}
\end{proposition}
\begin{proof}
  
  Using   Corollary \ref{action-one-ball} repeatedly, we have   \[F_\mathbf j=(\sqrt{-1})^k\big(\prod_{t=1}^k(-1)^{\p{i_1}+\ldots+\p{i_t}}\big) \sum_{\sigma\in\mathfrak{S}_{I_k}}(-1)^{\epsilon_\sigma}f_\sigma \:.\]

  Denote $\lambda_{>r}=(\lambda_{r+1},\lambda_{r+2},\ldots,\lambda_k)$. It suffices to show  
  \begin{equation}
  \label{pre-algorithm}
        \sum_{\sigma\in\mathfrak{S}_{S}}(-1)^{\epsilon_\sigma}f_\sigma=\Delta(y_S)g_{\lambda_{>(k-s)}}(y_{S^c})
  \end{equation}
 for any $S=\{c_1<c_2<\ldots,c_s\}\subset I_k$ with $|S|=s$, 
 where 
  \begin{equation}
  \label{bound condition}
  g_{\lambda_{>(k-s)}}(y_{S^c})=e_{\bar\lambda_{>(k-s)}-\rho_s}(y_{S^c})+\sum_{\mu\rhd\bar\lambda_{>(k-s)}-\rho_{s}}d_\mu e_\mu(y_{S^c}), d_\mu\in\mathbb Z.
\end{equation}
  
  We prove \eqref{pre-algorithm} and \eqref{bound condition} by induction on $s=|S|$. The case $s=1$ is known by Corollary \ref{action-one-ball}. 
  In general, for any $x\in S$, set $S(x)=S\setminus \{x\}$. For each $\sigma\in \mathfrak{S}_S$, $\exists \:r(\sigma)$ such that  $\sigma_{s}=c_{r(\sigma)}$. Moreover,  $\sharp\{t|\sigma_t<\sigma_{s},1\le t< s\}=r(\sigma)-1$.
  By induction hypothesis on $|S(x)|=s-1$, we have 
\begin{align*}
    &\sum_{\sigma\in\mathfrak{S}_{S}}(-1)^{\epsilon_\sigma}f_\sigma
    =
    \sum_{\sigma_{s}\in {S}}(-1)^{r(\sigma)-1}\sum_{\sigma_{<s}\in\mathfrak{S}_{S(\sigma_{s})}}(-1)^{\epsilon_\sigma}
    f_\sigma
    \\
    &=\sum_{\sigma_{s}\in {S}}(-1)^{r(\sigma)-1}\Delta(y_{S(\sigma_s)})g_{\lambda_{>(k-s+1)}}(y_{S(\sigma_{s})^c})e_{\bar\lambda_{k-s+1}}(-y_{i_{\sigma_1}},\ldots,-y_{i_{\sigma_{s-1}}},\widehat{y_{i_{\sigma_{s}}}},y_{i'_1}\ldots,y_{i'_{a-s}}),
\end{align*}
where $\sigma_{<s}=(\sigma_1,\cdots,\sigma_{s-1})$. 
Using the well-known fact about the elementary symmetric polynomials  
\begin{equation}
\label{er-sum}    
e_r(x_1,\ldots,x_k,x_{k+1},\ldots,x_a)=\sum_{i=0}^{min\{k,r\}}e_i(x_1,\ldots,x_k)e_{r-i}(x_{k+1},\ldots,x_a),
\end{equation} we have 
$g_{\lambda_{>(k-s+1)}}(y_{S(\sigma_{s})^c})=\sum_{t=1}^{s}y^{s-t}_{i_{\sigma_{s}}}A_{s-t}$  and 
$$
e_{\bar\lambda_{k-s+1}}(-y_{i_{\sigma_1}},\ldots,-y_{i_{\sigma_{s-1}}},\widehat{y_{i_{\sigma_{s}}}},y_{i'_1}\ldots,y_{i'_{a-s}})
=\sum_{h=1}^{s}e_{h-1}(-y_{i_{\sigma_1}},\ldots,-y_{i_{\sigma_{s-1}}})e'_{\bar\lambda_{k-s+1}-h+1},
$$ 
where $A_t, e'_s$ are symmetric polynomials of $y_{i'_1},\ldots,y_{i'_{a-s}}$ and do not depends on the choice of $\sigma\in\mathfrak{S}_S$. 

To compute $\Delta(y_{i_{\sigma_1}},\ldots,y_{i_{\sigma_{s-1}}})g_{\lambda_{>(k-s+1)}}(y_{S(\sigma_{s})^c)})$, we 
    define the  matrix $B=\{e_{j-1}(\hat{x_i})\}_{i,j=1,\ldots,s}$, where $e_{j-1}(\hat{x_i})$ is the $(j-1)$th elementary symmetric polynomial of $\{x_1,\ldots,\hat{x_i},\ldots,x_{s}\}$. Thanks to \cite{STA}, we have 
    \begin{equation}
    \label{equ:determi}
      \Delta(x_1,\ldots,x_{s})=\det(B).  
    \end{equation}
Direct computation by using \eqref{equ:determi} and \eqref{er-sum}  yields the following equality for     
    the  $(i,j)$-minor $M_{i,j}$ of $B$: 
\begin{equation}
\label{equ:minorM}
   M_{i,j}=x_i^{s-j}\prod_{\substack{1\le b<c\le s\\b,c\ne i}}(x_b-x_c)=x_i^{s-j}\Delta(x_1,\ldots,\hat{x_i},\ldots,x_{s}).  
\end{equation}

Let $B'$ be the matrix obtained from matrix $B$ by replacing all $x_j$ with $-y_{i_{s_j}}$, and $M'_{i,j}$ be the $(i,j)$-minor of $B'$. 
Then
\begin{equation}
\label{induction by minor}
\begin{aligned}
\sum_{\sigma\in\mathfrak{S}_S}(-1)^{\epsilon_\sigma}f_\sigma
&=
\sum_{\sigma_{s}\in S}(-1)^{r-1}
\sum_{t,h=1}^{s}y_{i_{\sigma_{s}}}^{s-t}\Delta(y_{i_{\sigma_1}},\ldots,y_{i_{\sigma_{s-1}}})e_{h-1}(-y_{i_{\sigma_1}},\ldots,-y_{i_{\sigma_{s-1}}})A_{s-t}e'_{\bar\lambda_{k-s+1}-h+1}\\
&\overset{\eqref{equ:determi}}{\underset{\eqref{equ:minorM}}{=}}\sum_{\sigma_{s}\in S}(-1)^{r(\sigma)-1}\sum_{t,h=1}^{s} (-1)^{s-t}(-1)^{(s-1)(s-2)/2}B'_{r(\sigma),h}M'_{r(\sigma),t}A_{s-t}e'_{\bar\lambda_{k-s+1}-h+1}\\
&=\sum_{t,h=1}^{s}(-1)^{s(s-1)/2}(\sum_{\sigma_{s}\in S} (-1)^{r(\sigma)+t}B'_{r(\sigma),h}M'_{r(\sigma),t})A_{s-t}e'_{\bar\lambda_{k-s+1}-h+1}\\
&=\sum_{t=1}^{s}(-1)^{s(s-1)/2}\Delta(-y_{i_{c_1}},\ldots,-y_{i_{c_{s}}})A_{s-t}e'_{\bar\lambda_{k-s+1}-t+1}\\
&=\Delta(y_{i_{c_1}},\ldots,y_{i_{c_{s}}})\sum_{t=1}^{s}A_{s-t}e'_{\bar\lambda_{k-s+1}-t+1},
\end{aligned}
\end{equation}
where the forth identity is due to the identity  $\sum_{r(\sigma)=1}^{s}(-1)^{r(\sigma)+t}B'_{r(\sigma),h}M'_{r(\sigma),t}=\delta_{h,t}\det B'$.

Then it suffices to prove $\sum_{t=1}^{s}A_{s-t}e'_{\bar\lambda_{k-s+1}-t+1}$ satisfies the form in the RHS of \eqref{bound condition}. Furthermore, we only need to prove the special $S=I_r, \text{ where }r=1,\cdots,k$, as the following recursive proof is essentially an induction of the cardinality of set $S$, independent of the selection of elements in $S$. Note that \eqref{bound condition} equals to \eqref{lead-terms-glambda} when $S=I_k$.
The computation in \eqref{induction by minor} implies an algorithm for calculating $g_r:=g_{\lambda>(k-|I_r|)}(y_{(I_r)^c})$ as follows.  Let 
 $g_1=e_{\bar\lambda_k}(\hat{y_{i_1}},y_{i_2},\ldots,y_{i_2})$.
 For any $1\le i\le k-1$,
 we define $A_0,A_{1},\ldots,A_{i-1}$ such that $g_r=\sum_{j=0}^{r}y_{i_{r+1}}^{j}A_{j}$ and let $g_{r+1}=\sum_{j=0}^{r}A_{j}e_{\bar\lambda_{k-r}+j-r}(y_{i_{r+2}},\ldots,y_{i_a})$, where $2\le r\le k$.
  Then by  \eqref{induction by minor} we have  $g_\lambda=g_k$.

We begin to show \eqref{lead-terms-glambda} by using the above algorithm to compute $g_r,\:r=2,\cdots,k$. 
For each pair of integers $i,j$ such that $1\le i<j\le k$ define the  raising operators
\begin{equation}
\label{equ:rasingop}
\begin{aligned}
    R_{i,j}:\mathbb{Z}^r&\rightarrow\mathbb{Z}^r\\
    (a_1,\ldots,a_i,\ldots,a_j,\ldots, a_r) &\mapsto (a_1,\ldots,a_i+1,\ldots,a_j-1,\ldots,a_r). 
\end{aligned}
\end{equation}
Let $\mathfrak{R}_r=\{\prod_{1\le i<j\le r}R_{i,j}^{r_{i,j}}|r_{i,j}\in\{0,1\}\}$.
Note that \eqref{lead-terms-glambda} follows from the  claim that $g_r=\sum e_\mu(y_{i_{r+1}},\ldots,y_{i_a})$, where the sum runs over all elements of $\mu=R(\bar\lambda_{>(k-r)}-\rho_r), R\in \mathfrak{R}_r$. Here we set  $e_\mu(y_{i_{r+1}},\ldots,y_{i_a})=0$ if $\mu_j<0$ for some $j\in\{1,\ldots,l(\mu)\}$. 

We prove the claim by induction on $r$. The case of $r=1$ is trivial.
Denote $I(j)=\{(\alpha_1,\ldots,\alpha_r)| \alpha_t\in\{0,1\},\forall t \text{ and } \sum_{t=1}^r\alpha_t=j\}$, $\bar\lambda(r):=\bar\lambda_{>(k-r)}=(\bar\lambda_{k+1-r},\ldots,\bar\lambda_k)$.
By induction hypothesis and the above algorithm, we have 
$$
g_r=\sum_{R\in\mathfrak{R}_k}e_{R(\bar\lambda(r)-\rho_r)}(y_{i_{r+1}},\ldots,y_{i_a})\overset{\eqref{er-sum}}{=} \sum_{R\in\mathfrak{R}_k}\sum_{j=0}^{r}\sum_{\alpha\in I(j)} y_{i_{r+1}}^{j}e_{R(\bar\lambda(r)-\rho_r)-\alpha}(y_{i_{r+2}},\ldots,y_{i_a}),
$$
This means $A_{j}=\sum_{R\in\mathfrak{R}_r}\sum_{\alpha\in I(j)} e_{R(\bar\lambda(r)-\rho_r)-\alpha}(y_{i_{r+2}},\ldots,y_{i_a})$.
 Then
    \begin{align*}
        g_{r+1}&=\sum_{j=0}^{r}A_{j}e_{\bar\lambda_{k-r}+j-r}(y_{i_{r+2}},\ldots,y_{i_a})\\
        &=\sum_{j=0}^{r}\big(\sum_{R\in\mathfrak{R}_r}\sum_{\alpha\in I(j)} e_{R(\bar\lambda(r)-\rho_r)-\alpha}(y_{i_{r+2}},\ldots,y_{i_a})\big)e_{\bar\lambda_{k-r}+j-r}(y_{i_{r+2}},\ldots,y_{i_a})\\
        &=\sum_{j=0}^{r}\sum_{R\in\mathfrak{R}_r}\sum_{2\le h_1<\ldots<h_t\le r+1}e_{R_{1,h_1}R_{1,h_2}\ldots R_{1,h_j}\iota (R)(\bar\lambda(r+1)-\rho_{r+1})}(y_{i_{r+2}},\ldots,y_{i_a})\\
        &=\sum_{R\in\mathfrak{R}_{r+1}}e_{R(\bar\lambda(r+1)-\rho_{r+1})}(y_{i_{r+2}},\ldots,y_{i_a}),
    \end{align*}
where   $\iota (R_{i,j})=R_{i+1,j+1}$ for all $1\le i<j\le r-1$.
This completes the proof of the claim and hence \eqref{lead-terms-glambda} and \eqref{bound condition} hold
since  $R(\bar\lambda-\rho_k)\trianglerighteq\bar\lambda-\rho_k,\forall R\in\mathfrak{R}_k$.
\end{proof}
 For $\lambda\in \SPar_{a,k}, \mu\in\Par_a$, the action of $\omega^\circ_{\bar\lambda}\omega_\mu$ can be written as  
 \begin{equation}
\label{action-lambda and mu}
\begin{tikzpicture}[baseline = -1mm,scale=1,color=\clr]
  \draw[-,line width=1.2pt] (0.08,-.5) to (0.08,.5);
  \node at (.08,-.7) {$\scriptstyle a$};
  \draw(0.08,0.2) \bdot;
  \draw(.45,.2)node {$\scriptstyle \omega^\circ_{\bar\lambda}$};
  \draw(0.08,-.2) \bdot;
  \draw(.45,-.2)node {$\scriptstyle \omega_{\mu}$};
\end{tikzpicture}
        (u\otimes h\otimes v_\mathbf i)
        \equiv \sum_{\mathbf j }u\otimes \tilde F_\mathbf j h\otimes v_\mathbf j,
\end{equation}
 where the sum runs over all sequence $\mathbf j=(j_1,\ldots,j_a)\in I^a$ with $j_t\in\{i_t,\bar{i_t}\}$ for all $1\le t\le a$ such that there are  at most $k$ numbers of  $j_t$'s belonging  to $\{\bar{i_1},\ldots,\bar{i_a}\}$  and $\tilde F_\mathbf{j}\in \mathbb{Z}[y_{i_1},\ldots,y_{i_a}]$ with degree $|\bar\lambda|+|\mu|$.

\begin{corollary}
\label{dotpacket-action}
For $\mathbf j=(\bar i_1,\ldots, \bar i_k, i_{k+1}, \ldots, i_a)$ in \eqref{action-lambda and mu}, we have 
   \[\tilde F_\mathbf j=(\sqrt{-1})^k\prod_{t=1}^k(-1)^{\p{i_1}+\ldots+\p{i_t}}\Delta(y_{i_1},\ldots,y_{i_k})g_\lambda e_\mu(y_{i_1},\ldots,y_{i_a}).\]
\end{corollary}
\begin{proof}
   This follows directly from Corollary \ref{actionofelempoly} and Proposition \ref{algorithm}.
\end{proof}

\begin{remark}
\label{remark_atmostk}
    If we define $\psi(\mathbf{i},\mathbf{j})=\sharp\{s|j_s=\bar{i_s},s\in\{1,\ldots,a\}\}$, then \eqref{action-lambda and mu} becomes 
$$
\label{leadterm-action}
   \begin{tikzpicture}[baseline = -1mm,scale=1,color=\clr]
   \draw[-,line width=1.5pt] (0.08,-.5) to (0.08,.5);
   \node at (.08,-.7) {$\scriptstyle a$};
   \draw(0.08,0.2) \bdot;
   \draw(.45,.2)node {$\scriptstyle \omega^\circ_{\bar\lambda}$};
   \draw(0.08,-.2) \bdot;
   \draw(.45,-.2)node {$\scriptstyle \omega_{\mu}$};
\end{tikzpicture}
        (u\otimes h\otimes v_{\mathbf i})
        \equiv \sum_{\underset{\phi(\mathbf{i},\mathbf{j})=k}{\mathbf j}}u\otimes \tilde F_\mathbf j h\otimes v_\mathbf j +\sum_{\underset{\phi(\mathbf{i},\mathbf{j})<k}{\mathbf j} }u\otimes \tilde F_\mathbf j h\otimes v_\mathbf j.
$$
This implies that
$g_{\lambda,\mu}$ with $l(\lambda)\ge k$ can not be expressed as a linear combination of some $g_{\alpha,\beta}$ with $l(\alpha)<k$.
Similarly, in the  proof of Theorem \ref{basis-theorem}, it suffices to analyze the first summand in the above identity and focus on a specific term $\tilde F_{(\bar i_1,\ldots, \bar i_k, i_{k+1}, \ldots, i_a)}$ from Corollary \ref{dotpacket-action}.
\end{remark}

\subsection{A basis theorem for $\QAW$}
We are ready to give a basis theorem for any morphism space of $\QAW$ over $\kk$.
For any $g_{\lambda,\mu}$ in \eqref{dots_data_simplify}, let $\deg^{\circ}g_{\lambda,\mu}$ be the number of white dots on $ g_{\lambda,\mu}$.
\begin{theorem}
\label{basis-theorem}
(A basis theorem for $\Hom_{\QAW}(\mu,\lambda)$)
    For any $\lambda,\mu\in \Lambda_{\text{st}}(m)$, $Hom_\QAW(\mu,\lambda)$ has a $\kk$-basis
    given by $\PMat_{\lambda,\mu}$.
\end{theorem}
\begin{proof}
Thanks to Proposition \ref{prop:spanofaff}, it suffices to prove that $\PMat_{\lambda,\mu}$ is linearly independent.  
    Suppose $\lambda,\mu\in \Lambda_{\text{st}}(m)$ with $l(\lambda)=h$ and $l(\mu)=t$.
Recall for any reduced chicken foot diagram of shape $A \in\Mat_{\lambda,\mu}$, each $\mu_j$ splits first at the bottom to $a_{1j}, \ldots, a_{hj}$ from left to right, then the thin strands $a_{i1},\ldots, a_{it}$ merge back to  $\lambda_i$ at the top.
For any elementary chicken foot diagram of shape $A$, there is an elementary dot packet,  say $g_{\nu_{i,j},\eta_{ij}}, 
\nu_{i,j}\in \SPar_{a_{ij}},\eta_{i,j}\in\Par_{a_{i,j}}$ on each $(i,j)$-leg at the bottom. 
We first assume that $\kk=\C$. Suppose that 
\begin{equation}
    L:=\sum_{A_\wp\in S}c_{A_\wp} A_\wp=0,
\end{equation}
for some finite subset $S\subset \PMat_{\lambda,\mu}$ with each $c_{A_\wp}\in \C^*$. 
Let $\hbar =\max\{\deg (A_\wp) \mid A_\wp\in S\}$ and $S_\hbar=\{A_\wp\in S\mid \deg (A_\wp)=\hbar \}$. 

It remains to prove $S_\hbar=\emptyset$. To prove this, we use the functor $\mathcal F_{M^{\text{gen}}}$ and compare the highest degree terms of $A_\wp\in S_\hbar$ on a certain element of $M^{\text{gen}}\otimes S^\mu V$, where $V$ is the natural representation of $\mathfrak{q}_n$.
We choose $n\gg 0$ (e.g., $n>d$) and $u:=1\otimes 1\otimes 1\otimes w_\mu$, where 
$$w_\mu:=m_1\otimes m_2\otimes\ldots\otimes m_t,$$
and  $m_j=v_{\sum_{i=0}^{j-1}\mu_i+1}\otimes v_{\sum_{i=0}^{j-1}\mu_i+2}\otimes\ldots\otimes v_{\sum_{i=0}^{j}\mu_i}\in S^{\mu_j}V$.
That is, the indices of the vector in   $S ^{\mu_j}V $ is  $I_{\mu_j}=\{\sum_{1\le s\le j-1}\mu_s+1, \ldots, \sum_{1\le s\le j}\mu_s \}$, $1\le j\le l(\mu)=t$. Note that by \eqref{dots-change +-} and corollary \ref{actionofelempoly}, the vector  indices  in $S^{\mu_j}V$ are included in $I_{\mu_j}\cup \bar I_{\mu_j}$ under the elementary dot packet action, where $\bar I_{\mu_j}:=\{\bar j|j\in I_{\mu_j}\}$ and the dot packet $g_{\nu_{i,j},\eta_{i,j}}$ action  changes at most $\deg^\circ g_{\nu_{i,j},\eta_{i,j}}=l(\nu_{i,j})$ indices of $v_i$ from $i$ to $\bar{i}$.
Since $(I_{\mu_j}\cup \bar I_{\mu_j})\cap (I_{\mu_i}\cup \bar I_{\mu_i})=\emptyset$ if $i\ne j$,  this implies that the part of the highest degree component of $F_{M^{\text{gen}}}(A_\wp)(u)$ in $S^{\mu}V $ is different from  $F_{M^{\text{gen}}}(A_\wp')(u)$ if $A\neq A'$, where 
    $A_\wp=(A,P),  A_\wp'=(A',P')\in S_\hbar$.
Thus, By remark \ref{remark_atmostk} we may assume that all $A_\wp$ in $S_\hbar$ have the same shape $A\in \Mat_{\lambda,\mu}$ and the same $\deg^\circ g_{\nu_{i,j}}$ for each $(i,j)$-legs.

Next, we choose a unique element $z_\lambda$ appearing in the action of chicken foot diagram $A$ on $w_\mu$, where $z_\lambda=z_{\lambda_1}\otimes z_{\lambda_2}\otimes\ldots\otimes z_{\lambda_h}$, $z_{\lambda_i}=v_{d'_1}\otimes v_{d'_2}\otimes\ldots\otimes v_{d'_{\lambda_i}}$ with $d'_x\in\{d_x,\bar{d_x}\}$, $(d_1,\ldots,d_{\lambda_i})=(\mathbf i_1, \ldots, \mathbf i_t )$, where 
$$
\mathbf i_j=(p_{i,j}+1,\ldots,p_{i,j}+a_{i,j}), p_{i,j}=\sum_{1\le m\le j-1}\mu_m+\sum_{1\le s\le i-1}a_{sj}
$$
and $d'_{x}=\bar {d_x}$ if $x=p_{i,j}+q,\text{ where } q\in \{1,\ldots,\deg^\circ g_{\nu_{i,j},\eta_{i,j}}\},j\in\{1,\ldots,t\}$, $d'_{x}=d_x$ if $x=p_{i,j}+q,\text{ where } q\in \{\deg^\circ g_{\nu_{i,j},\eta_{i,j}}+1,\ldots,a_{i,j}\},j\in\{1,\ldots,t\}$.

The selection of $z_\lambda$ ensures that the action of dot packet in each leg obtains the similar term corresponding to corollary \ref{dotpacket-action},
then the unique term of degree $\hbar$ component of $\mathcal F_{M^{\text{gen}}}(L)(u)$ with part  $z_\lambda$  must be 
$$
\sum_{B\in S_\hbar}c_{A_\wp}d_{\hbar} \prod_{i,j} \Delta_{i,j} g_{\nu_{i,j}}e_{\eta_{ij}} \big(y_{p_{i,j}+1}, \ldots,y_{p_{i,j}+a_{i,j}} \big) z_\lambda,   
$$ where $d_{\hbar}\in\{\pm1,\pm\sqrt{-1}\}$,$\Delta_{i,j}=\Delta(y_{p_{i,j}+1},\ldots,y_{p_{i,j}+\deg^\circ g_{\nu_{i,j},\eta_{i,j}}})$, $g_{\nu_{i,j}}$ is the polynomial of $\{y_p|p\in\{p_{i,j}+\deg^\circ g_{\nu_{i,j},\eta_{i,j}}+1,\ldots,p_{i,j}+a_{i,j}\}$, as defined in \ref{algorithm} and $\Delta_{i,j}g_{\nu_{i,j}}e_{\eta_{i,j}}$ is obtained by the action of the dot packet $g_{\nu_{i,j},\eta_{i,j}}$ on the $(i,j)$-leg.
This can not be zero if $S_\hbar\neq \emptyset$ by lemma \ref{indepedance of k wdots} and the fact that $\mathbb{Z}[y_1,\ldots,y_{|\mu|}]$ is integral domain. Hence we conclude that $S_\hbar =\emptyset$.

This proves that $\PMat_{\lambda,\mu}$ is linearly independent over $\mathbb C$ and, consequently, over $\Z$. 
Through the standard base change argument from  $\Z$ to $\kk$, the linear independence of  $\PMat_{\lambda,\mu}$ extends to any  $\kk$. This completes the proof.
\end{proof}
\begin{example}
For $\lambda=(2,3), \mu=(3,2)$, we have the following elementary chicken foot diagram:
$$
A=
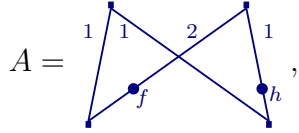
\begin{tikzpicture}[baseline = 19pt,scale=0.3,color=\clr,inner sep=0pt, minimum width=11pt]
\draw[-,line width=.75mm] (-4,0) to (-4,-.3);
\draw[-,line width=.75mm] (4,0) to (4,-.3);
\draw[-,line width=.75mm] (-3,5) to (-3,5.3);
\draw[-,line width=.75mm] (3,5) to (3,5.3);
\draw[-,line width=.25mm] (-4,0) to (-3,5);
\draw[-,line width=.25mm] (-4,0) to (3,5);
\draw[-,line width=.25mm] (4,0) to (-3,5);
\draw[-,line width=.25mm] (4,0) to (3,5);
\draw (-2,1.43)\bdot;
\node at (-1.6,1){$\scriptstyle f$};
\draw (3.7,1.43) \bdot;
\node at (4.3,1.2){$\scriptstyle h$};
\node at (-4,4) {$\scriptstyle 1$};
\node at (-2.4,4) {$\scriptstyle 1$};
\node at (0.6,4) {$\scriptstyle 2$};
\node at (4,4) {$\scriptstyle 1$};
\end{tikzpicture},
$$
where $f=g_{\nu_{2,1},\eta_{2,1}},h=g_{\nu_{2,2},\eta_{2,2}}$ and $\deg^\circ f=2,\deg^\circ h=1$.
Then  
$w_\mu=m_1\otimes m_2=(v_1 v_2 v_3)\otimes (v_4 v_5)$ and
\begin{align*}
\mathcal F_{M^{\text{gen}}}(A)(u)&\equiv d_1 1\otimes1\otimes \Delta_{2,1}g_{\nu_{2,1}}e_{\eta_{2,1}}(y_{2},y_{3})\Delta_{2,2}g_{\nu_{2,2}}e_{\eta_{2,2}}(y_{5}) \otimes v_{1}v_{4}\otimes v_{\bar{2}}v_{\bar{3}}v_{\bar{5}}\\
&+d_2 1 \otimes1\otimes \Delta_{2,1}g_{\nu_{2,1}}e_{\eta_{2,1}}(y_{1},y_{3})\Delta_{2,2}g_{\nu_{2,2}}e_{\eta_{2,2}}(y_{5})  \otimes v_{2}v_{4}\otimes v_{\bar{1}}v_{\bar 3}v_{\bar{5}}\\
&+d_3 1 \otimes1\otimes \Delta_{2,1}g_{\nu_{2,1}}e_{\eta_{2,1}}(y_{1},y_{2})\Delta_{2,2}g_{\nu_{2,2}}e_{\eta_{2,2}}(y_{5})  \otimes v_{3}v_{4}\otimes v_{\bar{1}}v_{\bar 2}v_{\bar{5}}\\
&+d_4 1 \otimes1\otimes \Delta_{2,1}g_{\nu_{2,1}}e_{\eta_{2,1}}(y_{2},y_{3})\Delta_{2,2}g_{\nu_{2,2}}e_{\eta_{2,2}}(y_{4})  \otimes v_{1}v_{5}\otimes v_{\bar{2}}v_{\bar 3}v_{\bar{4}}\\
&+d_5  1 \otimes1\otimes \Delta_{2,1}g_{\nu_{2,1}}e_{\eta_{2,1}}(y_{1},y_{3})\Delta_{2,2}g_{\nu_{2,2}}e_{\eta_{2,2}}(y_{4})  \otimes v_{2}v_{5}\otimes v_{\bar{1}}v_{\bar 3}v_{\bar{4}}\\
&+d_6 1 \otimes1\otimes \Delta_{2,1}g_{\nu_{2,1}}e_{\eta_{2,1}}(y_{1},y_{2})\Delta_{2,2}g_{\nu_{2,2}}e_{\eta_{2,2}}(y_{4})  \otimes v_{3}v_{5}\otimes v_{\bar{1}}v_{\bar 2}v_{\bar{4}}\\
&+\sum_{\underset{\sharp\{s|j_s=\bar{i}, i\in\{1,\ldots,5\}\}<3}{\{j_1,\ldots,j_5\}=\{1,\ldots,5\}}}d_{j_1,\ldots,j_5}1\otimes 1\otimes  G_{j_1,\ldots,j_5}\otimes v_{j_1}v_{j_2}\otimes v_{j_3} v_{j_4}v_{j_5},
\end{align*}
where $d_i,d_{j_1,\ldots,j_5}\in\{\pm1,\pm\sqrt{-1}\},G_{j_1,\ldots,j_5}\in\mathbb{Z}[y_1,\ldots,y_5]$. The variables of $\Delta_{i,j}g_{\nu_{i,j}}$ in $\Delta_{i,j}g_{\nu_{i,j}} e_{\eta_{i,j}}$ are obvious (cf. Proposition \ref{algorithm}), so we omit it.
In this case $z_\lambda=v_1 v_4\otimes v_{\bar 2}v_{\bar 3} v_{\bar 5}$.
\end{example}
\begin{theorem}
\label{thm:basisofqweb}
    (A basis theorem for $\Hom_{\qW}(\mu,\lambda)$)
    For any $\lambda,\mu\in \Lambda_{\text{st}}(m)$, $Hom_\qW(\mu,\lambda)$ has a $\kk$-basis
    given by 
\begin{equation}
\label{equ:parmatr1}
\PMat^{1}_{\lambda,\mu}:=\{(A,P)\in \PMat_{\lambda,\mu}|P=((k),\emptyset),k=0,1\},
\end{equation}
where the condition $k=0,1$  means that each leg of the diagram can have at most one white dot and no black dots. Moreover, $\qW$ is a subcategory of $\QAW$.
\end{theorem}
\begin{proof}
From the degree zero part of the  proof of proposition \ref{prop:spanofaff}, we observe  that  $\PMat^{1}_{\lambda,\mu}$ spans $\Hom_{\qW}(\lambda,\mu)$. Consider the obvious functor  $\mathcal D$  from $\qW$ to $\QAW$ which maps the generating objects and morphisms in $\qW$ to those of the same name in $\QAW$. 
The linear independence of $\PMat^{1}_{\lambda,\mu}$ then follows from the linear independence of the image under $\mathcal D$, as established in Theorem \ref{basis-theorem}.
This proves the first statement. 
The second then follows from the first,  since 
 $\mathcal D$ sends the basis to a linearly independent subset.
\end{proof}

\begin{remark}
    When $\kk=\C$, the basis $\PMat^{1}_{\lambda,\mu}$ is the same as $\{\theta_{T}|T\in\mathscr{T}_*(\lambda,\mu)\},\forall \lambda,\mu\in\Lambda'(r)$ (up to non-zero scalars) defined in \cite[Theorem 5.4]{Br19}.
\end{remark}

\subsection{Affine Sergeev superalgebra.}
Recall (e.g. \cite[Chap. 14]{Kle05}) the Affine Sergeev superalgebra $A_n$ is the superalgebra generated by even generators $s_1,\ldots,s_{n-1},x_1,\ldots,x_n$, and odd generators $c_1,\ldots,c_n$, subject to the relations for all admissible $i,j,l$:
\begin{align}
    &s^2_i=1,\quad s_is_{i+1}s_i=s_{i+1}s_i s_{i+1}, \quad s_j s_i=s_i s_j \quad (|i-j|>1), \label{ss} \\
    &x_is_i=s_ix_{i+1}+1-c_ic_{i+1}, \quad
    s_ix_i=x_{i+1}s_i+1+c_ic_{i+1} \label{sx},\\
    &x_ix_l=x_lx_i,\quad s_ix_j=x_js_i,\quad (j\ne i,i+1)\label{xx},\\
    &c^2_i=1,\quad c_ix_i=-x_ic_i, \quad c_ic_j=-c_jc_i, \quad 
    c_ix_j=x_jc_i\quad (i\ne j), \label{cx}\\
    &s_ic_j=c_js_i,\quad s_ic_i=c_{i+1}s_i,\quad s_ic_{i+1}=c_is_i, (j\ne i,i+1) \label{sc}.
\end{align}
\begin{corollary}
    We have a superalgebra isomorphism $\End_{\QAW}(1^n)\cong A_n$.
\end{corollary}
\begin{proof}
    We define a superalgebra homomorphism $\phi:A_n \rightarrow \End_{\QAW}(1^n)$ such that 
 \[ x_j\mapsto
\begin{tikzpicture}[baseline = 3pt, scale=0.5, color=\clr]
   \draw[-,line width=1pt] (-1,-.2) to (-1,1.2);
   \node at (-1,-.5) {$\scriptstyle 1$};
   \node at (-.4,.4) {$\ldots$};
   \draw[-,line width=1pt] (.5,-.2) to (.5,1.2);
   \node at (.5,-.5) {$\scriptstyle 1$};
   \draw[-,line width=1pt] (1,-.2) to (1,1.2);
   \draw(1,0.5) \bdot;
   \node at (1,-.5) {$\scriptstyle 1$};
   \draw[-,line width=1pt] (1.5,-.2) to (1.5,1.2);
   \node at (1.5,-.5) {$\scriptstyle 1$};
   \node at (2.1,.4) {$\ldots$};
   \draw[-,line width=1pt] (3,-.2) to (3,1.2);
   \node at (3,-.5) {$\scriptstyle 1$};
   \end{tikzpicture}, \quad 
    c_j\mapsto
\begin{tikzpicture}[baseline = 3pt, scale=0.5, color=\clr]
   \draw[-,line width=1pt] (-1,-.2) to (-1,1.2);
   \node at (-1,-.5) {$\scriptstyle 1$};
   \node at (-.4,.4) {$\ldots$};
   \draw[-,line width=1pt] (.5,-.2) to (.5,1.2);
   \node at (.5,-.5) {$\scriptstyle 1$};
   \draw[-,line width=1pt] (1,-.2) to (1,1.2);
   \draw(1,0.5) \wdot;
   \node at (1,-.5) {$\scriptstyle 1$};
   \draw[-,line width=1pt] (1.5,-.2) to (1.5,1.2);
   \node at (1.5,-.5) {$\scriptstyle 1$};
   \node at (2.1,.4) {$\ldots$};
   \draw[-,line width=1pt] (3,-.2) to (3,1.2);
   \node at (3,-.5) {$\scriptstyle 1$};
   \end{tikzpicture},\quad  
s_i \mapsto 
\begin{tikzpicture}[baseline = 3pt, scale=0.5, color=\clr]
   \draw[-,line width=1pt] (-1,-.2) to (-1,1.2);
   \node at (-1,-.5) {$\scriptstyle 1$};
   \node at (-.4,.4) {$\ldots$};
   \draw[-,line width=1pt] (.5,-.2) to (.5,1.2);
   \node at (.5,-.5) {$\scriptstyle 1$};
   \draw[-,line width=1pt] (1,-.2) to (2,1.2); 
   \node at (1,-.5) {$\scriptstyle 1$};
   \draw[-,line width=1pt] (2,-.2) to (1,1.2);
   \node at (2,-.5) {$\scriptstyle 1$};
   \draw[-,line width=1pt] (2.5,-.2) to (2.5,1.2);
   \node at (2.5,-.5) {$\scriptstyle 1$};
   \node at (3.2,.4) {$\ldots$};
   \draw[-,line width=1pt] (4,-.2) to (4,1.2);
   \node at (4,-.5) {$\scriptstyle 1$};
\end{tikzpicture}, \]
where the black and white dot are on the $j$th strand, and the crossing is between $i$th and $(i+1)$th strands. We claim that the homomorphism $\phi$ is well-defined. Indeed, the relation \eqref{ss} holds by \eqref{symmetric+braid}.
The relations in \eqref{sx} hold by \eqref{dotmovecrossing}, and 
the relations \eqref{xx}--\eqref{cx} hold by \eqref{doublewdot}, \eqref{wdotbdot} and the super-interchange law \eqref{super-interchange}. Additionally, the relation \eqref{sc} holds by \eqref{wdotsmovecrossing}.
Then the result follows because $\phi$ maps the basis elements in
$\{\omega c^{a_1}_1\cdots c^{a_n}_n  x^{b_1}_1\cdots x^{b_n}_n |a_i\in\{0,1\},b_i\in\mathbb{Z}_{\ge 0},\omega\in\mathfrak{S}_n\}$(cf. \cite[Theorem 14.2.2]{Kle05}) to the corresponding basis elements (up to signs) of $\End_{\QAW}(1^n)$ as described in Theorem \ref{basis-theorem}.
\end{proof}
\section{Appendix}
This section includes additional $\equiv$-relations in $\QAW$ and a useful result on the linear independence of certain partially symmetric polynomials. These results are used to prove both the spanning (see Proposition \ref{prop:spanofaff}) and linear independence aspects (see Theorem \ref{basis-theorem}) of the basis theorem for $\QAW$.
\subsection{More $\equiv$-relations}
Let $D_a$ (resp. $E_a$) be the submodule of $\End_{\QAW}(a)$ generated by the elements $\wlambdadota$ (resp. $
\begin{tikzpicture}[baseline = -1mm,scale=1,color=\clr]
   \draw[-,line width=1.5pt] (0.08,-.5) to (0.08,.4);
   \node at (.08,-.6) {$\scriptstyle a$};
   \draw(0.08,-0.2) \bdot;
   \draw(.45,0.2)node {$\scriptstyle \omega^\circ_{\bar\lambda}$};
   \draw(0.08,0.2) \bdot;
   \draw(.45,-0.2)node {$\scriptstyle \omega_{\mu}$};
\end{tikzpicture}$
), for any partition $\lambda, \mu\in \Par_a$. Then  $D_a$ is a  subalgebras of $\End_{\QAW}(a)$ by \eqref{extrarelation}.

 For any submodule $S$ of $\End_{\QAW}(a)$, we write $A\in_{\equiv} S $ if $ A\in S$ modulo lower degree terms.
 \begin{lemma}
\label{lambdainD}
		For any partitions  $\lambda\in \text {Par}_a$,  
\begin{equation}
\label{lambda}
\begin{tikzpicture}[baseline = -.5mm, scale=1., color=\clr]
                   \draw[-, line width=1.5pt] (0,-.5) to (0,-.2);
                   \draw[-, line width=1pt] (0,-.2) to[out=180,in=180] (0,.6);
                   \draw[-, line width=1pt] (0,-.2) to[out=0,in=0] (0,.6);
                   \draw[-, line width=1.5pt] (0,.6) to (0,.8);
                   \draw (-0.25,.2) \bdot;
                   \draw(-.6,.2)node {$\scriptstyle \omega_\lambda$};
                   \draw(-.3,-.1)node {$\scriptstyle a$};
                   \draw(.3,-.1)node {$\scriptstyle b$};
				\node at (1.5,.2){$\in_\equiv D_{a+b} $};
\end{tikzpicture}.
\end{equation}
\end{lemma}
	
\begin{proof}
The proof follows a similar approach to \cite[Lemma 2.13]{SW1}. While $\xdot$ moving across merges and splits generates lower-degree terms involving $\wxdot$, the relation \eqref{dotmovemerge+high} coincide with those in \cite[Lemma 2.12]{SW1} when considering only the highest-degree terms. Consequently, the same argument applies here.
\end{proof}

 For any partition $\lambda=(\lambda_1,\ldots, \lambda_k)$ with $\lambda_k\ne 0$, denote $l(\lambda)=k$         and $\lambda_{< q}=(\lambda_1,\ldots,\lambda_{q-1})$ for $1\le q\le k$. The definition of $\lambda_{>q}$ is similar. For simplicity  we  write
\[

     := 
\omega^\circ_{a,\bar\lambda_1}\omega^\circ_{a,\bar\lambda_2}\cdots\omega^\circ_{a,\bar\lambda_{q-1}}.\]

\begin{lemma}
		For any partitions  $\lambda\in \text{Par}_a$ with $l(\lambda)=k$,  
    \begin{equation}
			\label{lambdawdot}
			%
.   
\end{equation}
\end{lemma}
\begin{proof}
		We proceed by induction on both $b$ and $k$. First, consider the case $b=1$.   If $k=1$, the result is clear. Now, assume $k\ge 2$, and we will perform induction on $\lambda_k$. If $\lambda_k=1$, then

\begin{align*}
 %
 \in_\equiv E_{a+1} 
\end{align*}		  			
    by the inductive hypothesis on 
 $k-1$.
		Now let $\lambda_k\ge 2$. Then we have
\begin{align*}
%
. 
\end{align*}
		The first summand $\in_\equiv E_{a+1}$  by the inductive hypothesis on 
 $k-1$. For the second summand, we have
\begin{align*}
%
.
\end{align*}
		The first (resp. second) summand   $\in_\equiv E_{a+1}$ by inductive hypothesis on $k-1$ (resp. $\lambda_k-1$). The completes the proof for case $b=1$.
		
		Suppose $b\ge 2$. Then 
\[
%
. 
\]
The RHS   $\in_\equiv E_{a+1}$ by  Lemma \ref{lambdainD} and the inductive hypothesis on $ b-1$. This completes the proof.
	\end{proof}

\begin{lemma}
\label{keylemma}
		For any partitions $\lambda,\mu \in \text{Par}_a$, we have
\begin{equation}
%
 \in_\equiv E_{a+1}.
\end{equation}
\end{lemma}

\begin{proof}
Let $l(\lambda)=k$, $l(\mu)=t$. We simultaneously prove these two inclusions by induction on $k$. 
Suppose $k=1$. 
  	 We prove the first inclusion by induction on $t$ and $\bar\lambda_1$. 
  	Suppose  $t=1$. The case  $\bar\lambda_1=0$ is clear since 
  	$ 
%
.
\]
    The first  and the third  summands $\in_\equiv E_{a+1}$ by  \eqref{lambdawdot}. For the second summand, we have
    $$
%
.
    $$
    The first summand in $E_{a+1}$ is clear and  the second summand  $\in_\equiv E_{a+1}$
     by the inductive hypothesis on $\bar\lambda_1-1$. 
    
    Suppose $t\ge 2$, we proceed by induction on $\mu_t$. For $\mu_t=1$, we have
\begin{align*}
%
.
\end{align*}
    The first and third summands  $\in_\equiv E_a+1$   by inductive hypothesis on $t-1$. The second and fourth summands  $\in_\equiv E_a+1$   by  \eqref{lambda} and  \eqref{lambdawdot}, respectively.   Suppose $\mu_t\ge 2$.  We have
\begin{align*}
%
.    
\end{align*}
    The first and the third summands $\in_\equiv E_{a+1}$  by inductive hypothesis on $t-1$ and on $\mu_t-1 $ respectively. The second and fourth summands  $\in_\equiv E_a+1$  by   \eqref{lambda} and  \eqref{lambdawdot}, respectively. This completes the proof of the first inclusion when $k=1$. Similarly, when $k=1$, we can prove the second inclusion by induction $\bar\lambda_1$ with \eqref{doublewdot}, \eqref{wdotbdot}, \eqref{dotmovemerge+high}, \eqref{ballmerge}, \eqref{lambda} and  \eqref{lambdawdot}. We omit this proof for simplicity.  

    
    Suppose $k\ge 2$, we only prove the first inclusion since the second one is similar. We proceed by induction on $\bar\lambda_1$, the case of $\bar\lambda_1=0$  holds by
    $$
%
    $$
 since the RHS   $\in_\equiv E_{a+1}$   by inductive hypothesis on $k-1$. When $\bar\lambda_1\ge 1$, we have
\begin{align*}
%
. 
\end{align*}
    The first three summands   $\in_\equiv E_{a+1}$    by inductive hypothesis on $k-1$ and   $\bar\lambda_1-1$. Finally, for the last summand we have 
\[
%
    	).
\]
    For each summand on the RHS of the above identity, we can move $%
  $ downwards  using \ref{extrarelation} and \eqref{wdotballcommute}. Therefore, we deduce that   the RHS   $\in_\equiv E_{a+1}$   by inductive hypothesis on $k-1 $.
    This completes the proof.
\end{proof}
		
\begin{lemma}
\label{doublestring-g}
		For any  $\lambda, \mu \in \text{Par}_a$ and $\theta, \nu \in \text{Par}_b$, we have
$
%
\in_\equiv E_{a+b}$. 
\end{lemma}
\begin{proof}
	 We proceed by double induction: first on $n=l(\theta) $ and then on $\theta_1$. 
     Suppose $n=0$, we first prove the special case
\begin{equation}
\label{left-dotpacket}
%
\in_{\equiv} E_{a+b}.
\end{equation}
This inclusion is proved by induction on $b$ and on $l(\lambda)=k$. The case $b=1$ (resp. $k=0$) is already proved in   Lemma \ref{keylemma} (resp. Lemma \ref{lambdainD}). Now, suppose $b\ge 2$  and $k\ge 1$, then we proceed by induction on $\bar\lambda_1$. For $\bar\lambda_1=0$, we have    
$$
%
			\overset{\eqref{wdot-move-split and merge}}{\equiv}
%
.
$$
The first summand $\in_\equiv E_{a+b}$ by inductive hypothesis on $k-1$. The second summand  $\in_\equiv E_{a+b}$ by inductive hypothesis on $b-1$ and \eqref{keylemma} since 
$$
%
.
\end{equation}
		 The first summand   $\in_\equiv E_{a+b}$   by inductive hypothesis on $k-1$. For each summand in the first summation, we have
\begin{equation}
\label{movedowndot}
%
,
\end{equation} 
		 and hence $\in_\equiv E_{a+b}$ by   induction on $d$ and the inductive hypothesis on $c<\bar\lambda_1$. 
         For each summand in the second summation in \eqref{twosum}, there are some $\nu\in\text{Par}_a, \eta\in\text{Par}$ with $l(\nu)=k-1$ such that 
$$
%
\overset{\eqref{wdotballcommute}}{\equiv}
\text{the linear combination
 of }
%
,
$$    
and hence $\in_\equiv E_{a+b}$ by the induction on $d$ and the inductive hypothesis on $k-1<k$ and Lemma \ref{keylemma}. 

By the similar method of \cite[Lemma 2.13]{SW1}, The case of $n=0$ can be directly derived from \eqref{left-dotpacket}.
     Suppose  $n\ge 1$. We proceed by induction on $\theta_1$. If  $\theta_1=0$, we can move the white dot on the right strand to the left and top strands using \eqref{wdot-move-split and merge}, and hence the result follows from the case   $n=0$.

     Now, 
    suppose $\theta_1\ge 2$. We have
$$
%
.
$$
The first summand $\in_\equiv E_{a+b}$ by the inductive hypothesis on $n-1$. Moreover, note that the right strand of each term in the second summand $\in_\equiv E_b$ and the left strand of each term in the third summand  $\in_\equiv E_a$ by \eqref{wdotballcommute}, thus the second and the third summations   $\in_\equiv E_{a+b}$ by the inductive hypothesis on $n-1<n$ and on $d<\theta_1$, respectively. The lemma is proved.
	\end{proof}
	
\begin{lemma}
\label{rrball}
For any $a\in \N$ and $r\le a-1$ we have 
		$
%
.\]
	The third summand   $\in_\equiv D_a$  by Lemma \ref{lambdainD}.	
	 In the following       we   prove $A_{r,0}$ and $B_{r,1}$   $\in_\equiv Z_a$ by reverse induction on $t$. If $r<a-2-r$, then $A_{r,r}=0$ by \eqref{equ:twowhite=0}. Suppose $r\ge a-2-r$. We have  
$$
\begin{aligned}
A_{r,a-2-r}&=
%
\in_\equiv Z_a
\end{aligned}
$$
       since $a-1\neq 2r-a+1$ and the last diagram is in $D_a$ by Lemma  \ref{lambdainD}.
		  Suppose $t<\min\{r,a-2-r\}$.  We have
\begin{align*}
A_{r,t}&\equiv
%
 \in_\equiv Z_a 
\end{align*}
		  by the inductive hypothesis on $t+1$. Thus $A_{r,0}\in_\equiv Z_a$. 
          
          Finally,  the argument for  $B_{r,1}$ is similar.  
\end{proof}

\begin{lemma}
\label{strictpartition}
    For  any $\lambda,\mu\in \Par_a$, we have $g_{\lambda,\mu}\equiv\sum_id_ig_{\alpha_i,\beta_i}$, for some  $\alpha_i\in \SPar_{a},\beta_i\in\Par_a$ and $ d_i\in\kk$. 
\end{lemma}
\begin{proof} 
Suppose $\lambda\in \Par(m)$ for some $m\in \N$.
    We prove by induction on  $\deg^\circ g_{\lambda,\mu}=k$ and by inverse induction on  the dominance order $\rhd$ for $\lambda$. The base case $k=1$ is trivial, since  $\lambda=\lambda_1$ is clearly a strict partition. Suppose $k\ge 2$ and that  $\lambda$ is maximal with respect to $\rhd$.
If $m\le a$, then $\lambda=(m-1,1),\:k=2$ and it is already strict.
    If $a+1\le m< 2a$, then $\lambda=(a,m-a)$, which is also strict.
   If $m\ge 2a$, then 
 $\lambda_1=\lambda_2=a$.
We have 
\begin{equation}
\label{doublea-1}
\begin{aligned}
\begin{tikzpicture}[baseline = 1.5mm, scale=1, color=\clr]
	\draw[-,line width=2pt] (0,-0.4) to[out=up, in=down] (0,.7);
	\draw(0,-0.05) \bdot; 
        \node at (-.5,0.35) {$\scriptstyle \omega^\circ_{a-1}$};
	\draw(0,0.35) \bdot;
	\node at (-.5,-0.05) {$\scriptstyle \omega^\circ_{a-1}$};
	\node at (0,-.6) {$\scriptstyle a$};
\end{tikzpicture}
	&\equiv~
\begin{tikzpicture}[baseline = 1.5mm, scale=1, color=\clr]
	\draw[-,line width=2pt] (0,-0.4) to(0,-0.3);
	\draw[-,thick] (0,-.3) to [out=180,in=-90] (-.2,.-.1);
	\draw[-,thick] (0,-.3) to [out=0,in=-90] (.2,-.1);
	\draw[-,thick] (-.2,-.1) to (-.2,.4);
	\draw[-,thick] (.2,-.1) to (0.2,.4);
	\draw[-,thick] (-.2,.4) to [out=90,in=-180] (0,.6);
	\draw[-,thick] (.2,.4) to [out=90,in=0] (0,.6);
	\draw[-,line width=2pt] (0,0.6) to(0,.7);
	\draw(-0.2,-.05) \bdot;
	\draw(-0.2,0.35) \bdot;
        \draw(0.2,-.05) \wdot;
        \draw(0.2,0.35) \bdot;
        \draw (-0.7,0.35) node {$\scriptstyle \omega^\circ_{a-2}$};
	\draw (-0.7,-.05) node {$\scriptstyle \omega_{a-1}$};
	\draw (0.3,-.3) node {$\scriptstyle 1$};
	\draw (0,-0.6) node {$\scriptstyle a$};
\end{tikzpicture}
	+
\begin{tikzpicture}[baseline = 1.5mm, scale=1, color=\clr]
	\draw[-,line width=2pt] (0,-0.4) to(0,-0.3);
	\draw[-,thick] (0,-.3) to [out=180,in=-90] (-.2,.-.1);
	\draw[-,thick] (0,-.3) to [out=0,in=-90] (.2,-.1);
	\draw[-,thick] (-.2,-.1) to (-.2,.4);
	\draw[-,thick] (.2,-.1) to (0.2,.4);
	\draw[-,thick] (-.2,.4) to [out=90,in=-180] (0,.6);
	\draw[-,thick] (.2,.4) to [out=90,in=0] (0,.6);
	\draw[-,line width=2pt] (0,0.6) to(0,.7);
	\draw(-0.2,-.05) \bdot;
	\draw(-0.2,0.35) \bdot;
        \draw(0.2,-.05) \wdot;
        \draw(0.2,0.35) \wdot;
        \draw (-0.7,0.35) node {$\scriptstyle \omega_{a-1}$};
	\draw (-0.7,-.05) node {$\scriptstyle \omega_{a-1}$};
	\draw (0.3,-.3) node {$\scriptstyle 1$};
	\draw (0,-0.6) node {$\scriptstyle a$};
\end{tikzpicture}
\equiv-
\begin{tikzpicture}[baseline = 1.5mm, scale=1, color=\clr]
	\draw[-,line width=2pt] (0,-0.4) to(0,-0.3);
	\draw[-,thick] (0,-.3) to [out=180,in=-90] (-.2,.-.1);
	\draw[-,thick] (0,-.3) to [out=0,in=-90] (.2,-.1);
	\draw[-,thick] (-.2,-.1) to (-.2,.4);
	\draw[-,thick] (.2,-.1) to (0.2,.4);
	\draw[-,thick] (-.2,.4) to [out=90,in=-180] (0,.6);
	\draw[-,thick] (.2,.4) to [out=90,in=0] (0,.6);
	\draw[-,line width=2pt] (0,0.6) to(0,.7);
	\draw(-0.2,-.05) \bdot;
	\draw(-0.2,0.35) \bdot;
        \draw(0.2,-.05) \bdot;
        \draw(0.2,0.35) \wdot;
        \draw (-0.7,0.35) node {$\scriptstyle \omega^\circ_{a-2}$};
	\draw (-0.7,-.05) node {$\scriptstyle \omega_{a-1}$};
	\draw (0.3,-.3) node {$\scriptstyle 1$};
	\draw (0,-0.6) node {$\scriptstyle a$};
\end{tikzpicture}
				+
\begin{tikzpicture}[baseline = 1.5mm, scale=1, color=\clr]
	\draw[-,line width=2pt] (0,-0.4) to(0,-0.3);
	\draw[-,thick] (0,-.3) to [out=180,in=-90] (-.2,.-.1);
	\draw[-,thick] (0,-.3) to [out=0,in=-90] (.2,-.1);
	\draw[-,thick] (-.2,-.1) to (-.2,.4);
	\draw[-,thick] (.2,-.1) to (0.2,.4);
	\draw[-,thick] (-.2,.4) to [out=90,in=-180] (0,.6);
	\draw[-,thick] (.2,.4) to [out=90,in=0] (0,.6);
	\draw[-,line width=2pt] (0,0.6) to(0,.7);
	\draw(-0.2,-.05) \bdot;
	\draw(-0.2,0.35) \bdot;
        \draw (-0.7,0.35) node {$\scriptstyle \omega_{a-1}$};
	\draw (-0.7,-.05) node {$\scriptstyle \omega_{a-1}$};
	\draw (0.3,-.3) node {$\scriptstyle 1$};
	\draw (0,-0.6) node {$\scriptstyle a$};
\end{tikzpicture}
\\
	&=
\begin{tikzpicture}[baseline = 1.5mm, scale=1, color=\clr]
	\draw[-,line width=2pt] (0,-0.4) to(0,0);
	\draw[-,thick] (0,0) to [out=180,in=-180] (0,.5);
	\draw[-,thick] (0,0) to [out=0,in=0] (0,.5);
	\draw[-,line width=2pt] (0,0.5) to(0,.7);
	\draw(0,-0.2) \bdot; 
	\draw(-0.15,0.25) \bdot; 
        \draw(0.15,0.25) \wdot; 
	\draw (-0.6,0.25) node {$\scriptstyle \omega^\circ_{a-2}$};
	\draw (0,-0.6) node {$\scriptstyle a$};
\end{tikzpicture}
		+
\begin{tikzpicture}[baseline = 1.5mm, scale=1, color=\clr]
	\draw[-,line width=2pt] (0,-0.4) to(0,-0.3);
	\draw[-,thick] (0,-.3) to [out=180,in=-90] (-.2,.-.1);
	\draw[-,thick] (0,-.3) to [out=0,in=-90] (.2,-.1);
	\draw[-,thick] (-.2,-.1) to (-.2,.4);
	\draw[-,thick] (.2,-.1) to (0.2,.4);
	\draw[-,thick] (-.2,.4) to [out=90,in=-180] (0,.6);
	\draw[-,thick] (.2,.4) to [out=90,in=0] (0,.6);
	\draw[-,line width=2pt] (0,0.6) to(0,.7);
	\draw(-0.2,-.05) \bdot;
	\draw(-0.2,0.35) \bdot;
        \draw (-0.7,0.35) node {$\scriptstyle \omega_{a-1}$};
	\draw (-0.7,-.05) node {$\scriptstyle \omega_{a-1}$};
	\draw (0.3,-.3) node {$\scriptstyle 1$};
	\draw (0,-0.6) node {$\scriptstyle a$};
\end{tikzpicture}
	\overset{\eqref{equ:twowhite=0}}=
\begin{tikzpicture}[baseline = 1.5mm, scale=1, color=\clr]
	\draw[-,line width=2pt] (0,-0.4) to(0,-0.3);
	\draw[-,thick] (0,-.3) to [out=180,in=-90] (-.2,.-.1);
	\draw[-,thick] (0,-.3) to [out=0,in=-90] (.2,-.1);
	\draw[-,thick] (-.2,-.1) to (-.2,.4);
	\draw[-,thick] (.2,-.1) to (0.2,.4);
	\draw[-,thick] (-.2,.4) to [out=90,in=-180] (0,.6);
	\draw[-,thick] (.2,.4) to [out=90,in=0] (0,.6);
	\draw[-,line width=2pt] (0,0.6) to(0,.7);
	\draw(-0.2,-.05) \bdot;
	\draw(-0.2,0.35) \bdot;
        \draw (-0.7,0.35) node {$\scriptstyle \omega_{a-1}$};
	\draw (-0.7,-.05) node {$\scriptstyle \omega_{a-1}$};
	\draw (0.3,-.3) node {$\scriptstyle 1$};
	\draw (0,-0.6) node {$\scriptstyle a$};
\end{tikzpicture}.
\end{aligned} 
\end{equation}
Then, by \eqref{wdotballcommute} and Lemma \ref{lambdainD}, we have $g_{\lambda,\mu}\equiv\sum_ic_ig_{\alpha_i,\beta_i}$, where  $\deg^\circ g_{\alpha_i,\beta_i}<\deg^\circ g_{\lambda ,\mu}$. Thus, the result  holds  by the induction hypothesis on $k$. 
This completes the proof 
for $\lambda$ maximal.

In general, suppose $\lambda$ is not maximal and not strict. 
 Then there is some $1\le j\le k-1$ such that  $\lambda_j=\lambda_{j+1}=r+1$.
  By Lemma \ref{rrball} we have	 
\begin{equation}
\label{reduceii}
\begin{tikzpicture}[baseline = 1.5mm, scale=1, color=\clr]
	\draw[-,line width=2pt] (0,-0.4) to[out=up, in=down] (0,.7);
	\draw(0,-.05) \bdot; 
        \node at (.3,0.35) {$\scriptstyle \omega^\circ_r$};
	\draw(0,0.35) \bdot;
	\node at (.3,-.05) {$\scriptstyle \omega^\circ_r$};
	\node at (0,-.6) {$\scriptstyle a$};
\end{tikzpicture}
		-\sum_{t=1}^{\min\{r,a-1-r\}}c_i
\begin{tikzpicture}[baseline = 1.5mm, scale=1, color=\clr]
	\draw[-,line width=2pt] (0,-0.4) to[out=up, in=down] (0,.7);
	\draw(0,-.05) \bdot; 
        \node at (.5,0.35) {$\scriptstyle \omega^\circ_{r+t}$};
	\draw(0,0.35) \bdot;
	\node at (.5,-.05) {$\scriptstyle \omega^\circ_{r-t}$};
	\node at (0,-.6) {$\scriptstyle a$};
\end{tikzpicture}
\in_\equiv E^{0,2r}
,
\text{ for } r=1,2,\cdots,a-2. 
\end{equation}
for some  $c_i\in\kk$, where 
        \[E^{0,d}_a
        =\text{span-}\{ g_{\gamma,\nu} \mid \deg^\circ g_{\gamma,\nu}= 0 \text{ and } \deg g_{\gamma,\nu}=d\}.\]
This together with  \eqref{extrarelation} and  \eqref{doublea-1}, implies that   $g_{\lambda^\circ,\mu}\equiv\sum_id_ig_{\theta_i,\gamma_i}+\sum_jd_j'g_{{\theta'_j},\gamma'_j}$, where $\theta_i,\gamma_i,\theta'_j,\gamma'_j\in\Par_a, d_i,d'_i\in\kk$ and $\theta_i \rhd\lambda$,  with $\deg^\circ g_{{\theta'}_j,\gamma'_j}<\deg^\circ g_{\lambda,\mu}$. Now, the result follows from the induction hypothesis on $k$ and the dominance order. 
\end{proof}

\subsection{A useful fact on partially   symmetric polynomials}\label{subse:partiallysym}
For any composition $\alpha$, let $[\alpha]$ denote the partition obtained by reordering the entries of $\alpha$.
By \cite[Chapte I, Example 16]{Mac}, 
we have 
\begin{equation}
\label{order-union}
\lambda\cup\mu\trianglerighteq\nu\cup\pi
\end{equation}
for any  $\lambda,\mu,\nu,\pi\in\Par$ such that $\lambda\trianglerighteq\nu$ and $\mu\trianglerighteq\pi$, where
 $\lambda\cup\mu:=[\lambda,\mu]$. 
 
For any $k\in \mathbb{Z}_{\ge 1}$ and $\mu\in \Par$, let 
\begin{equation}
\label{equ:defofIkmu}
 I_{k,\mu}=\{(\alpha_1,\cdots,\alpha_{l(\mu)})|0\le \alpha_j\le \min(k,\mu_j),1\le j\le l(\mu)\}.   
\end{equation}

\begin{lemma}
\label{lem:oderminus}
    For any $\alpha\in I_{k,\mu}$ with $[\alpha]=\gamma$, we have $[\mu-\gamma]\trianglelefteq [\mu-\alpha]$.
\end{lemma}
\begin{proof}
Let 
    $\gamma=(\gamma_1,\cdots,\gamma_{l(\mu)})$, where some entries of $\gamma$ may be zero. Then, there exsits some $w\in \mathfrak{S}_{l(\mu)}$ such that 
    $\gamma\cdot w=\alpha$.
We prove  by induction on $\ell(w)=r$.
    Fix a reduced expression $w=s_{i_1}s_{i_2}\cdots s_{i_r}$. 
   Write $w'=s_{i_1}\cdots s_{i_{r-1}}$ and $\alpha'=\gamma \cdot w'$.
   Then,  there exsit indices  $p<q$ such that 
   $\alpha'_p=\gamma_{i_r}$ and 
   $\alpha'_q=\gamma_{i_{r+1}}$
   with $\gamma_{i_r}\ge \gamma_{i_r+1}$. Moreover,  we have the following 
\begin{align*}
   \mu-\alpha'&=(\mu_1-\alpha'_1, \ldots,\mu_p-\gamma_{i_r}, \ldots, \mu_q-\gamma_{i_{r}+1},\ldots) \\
\mu-\alpha&=(\mu_1-\alpha'_1, \ldots,\mu_p-\gamma_{i_r+1}, \ldots, \mu_q-\gamma_{i_{r}},\ldots) 
\end{align*}
  Note that  $\mu_p-\gamma_{i_r+1}\ge\mu_p-\gamma_{i_r}$, which implies that  
  $[\mu_p-\gamma_{i_r+1},\mu_q-\gamma_{i_r}]\trianglerighteq[\mu_p-\gamma_{i_r},\mu_q-\gamma_{i_r+1}]$ and hence,  by \eqref{order-union}, we conclude that 
$[\mu-\alpha']\trianglelefteq[\mu-\alpha]$. This completes the proof by the induction hypothesis. 
\end{proof}

For any $k\in \Z_{\ge1}$ and  $d\in \N$, 
we define the set  
\[C_{k,d}=\{(\lambda,\mu)|\lambda\in\SPar_{a,k},\mu\in\Par_a,|\bar\lambda|+|\mu|=d\}.\]
Next, we define the following partial order $\le$ on the pairs in  $C_{k,d}$ as follows:
\begin{equation}
\label{order-on-pairs}
    (\alpha,\beta)<(\lambda,\mu) \iff \bar\alpha\cup\beta\lhd\bar\lambda\cup\mu \text{ or } \bar\alpha\cup\beta=\bar\lambda\cup\mu \text{ and } \alpha<\lambda,
\end{equation}
where $\alpha<\lambda$ is the lexicographical order. 
 
For any $\alpha, \mu\in \Par$, we shall write $\alpha\subset\mu$ to mean that $\alpha_i\le\mu_i, \forall i\ge1$. In the following, we fix  an element  $(\lambda,\mu)\in C_{k,d}$ and 
 define 
\[\gamma_{\lambda,\mu}=(k^{d_0},(k-1)^{d_1},\ldots, (k-i)^{d_i},\ldots, 1^{d_{k-1}},0^{d_k}),\]
where  
 \[d_i=\sharp\{\mu_j \mid \bar\lambda_{i+1}<\mu_j\le \bar\lambda_i, 1\le j\le l(\mu) \}, \quad 0\le i\le k. \]
 Here,  $\bar\lambda_0=\infty$, $\bar\lambda_{k+1}=-1$ by convention. Then, by definition,  we have 
 $\gamma_{\lambda, \mu}\subset \mu$, and moreover,
 $\mu-\gamma_{\lambda,\mu}\in \Par$. 
 
Next, we define  $\tilde\lambda=\bar\lambda-\rho_k$ for $\lambda\in\SPar_{a,k}$, where $\rho_k$ is given in \eqref{equ:defofrhok}.
Let  
 $\nu:=\tilde \lambda\cup (\mu-\gamma_{\lambda,\mu})$.
 Then, $\nu$ is given by  
 \begin{equation}
     \label{nu}
    \nu=(\mathbf p_0, \tilde\lambda_1, \mathbf p_1, \tilde \lambda_2,\mathbf p_2 \ldots, \mathbf p_{k-1},\tilde \lambda_k, \mathbf p_{k} ),
 \end{equation}
 where $\mathbf p_i=(\mu_{p_{i-1}+1}-k+i, \ldots, \mu_{p_i}-k+i )$, $p_i=\sum_{j\le i}d_j$ for $0\le i\le k$ with $p_{-1}:=0$.
 
 Note that for  $z\in\{1,\cdots,k\}$
 \begin{equation}
 \label{equ:defofnu}
    \nu_r=\begin{cases}
       \tilde\lambda_z & \text{ if $r=p_{z-1}+z$}\\
        \mu_{r-z}-(k-z) & \text{if $p_{z-1}+z<   r<p_{z}+z+1$.}
    \end{cases}. 
 \end{equation}   
For any $(\lambda,\mu)\in C_{k,d}$ and  any $\alpha\in \SPar_{a,k}$, we define $C(\lambda,\mu)=\bigsqcup_{\alpha} C(\lambda,\mu)_\alpha$ with 
    \[C(\lambda,\mu)_\alpha:=\{(\alpha,\beta)\in C_{k,d}|\tilde\alpha\cup[\beta-\gamma_{\lambda,\mu}]=\nu, \gamma_{\lambda,\mu}\subset \beta\}\]

Note that for any $(\alpha,\beta)\in C(\lambda,\mu) $, we have $\tilde\alpha\subset_p \nu$, where $ \subset_p$ means $\tilde \alpha$ is obtained from $\nu$ by deleting some components.
Let $\nu\setminus \tilde\alpha$ be the partition obtained from $\nu$ by deleting the parts corresponding to $\tilde\alpha$. Define $\beta_\alpha=\nu\setminus \tilde\alpha+ \gamma_{\lambda,\mu}$.
Then, $(\alpha,\beta_\alpha)\in C(\lambda,\mu)_\alpha$ and $\beta_{\lambda}=\mu$.
This also implies that  $C(\lambda,\mu)_\alpha\neq \emptyset  $ if and only if 
$\tilde \alpha\subset_p \nu$.  

Finally, we define $V_{\lambda,\mu}:=\{\alpha\in \SPar_{a,k}|\tilde\alpha\subset_p\nu\}$.

\begin{lemma}
\label{Lem:maximalalpha}
Suppose  $\alpha\in V_{\lambda,\mu}
$. Then  for  all $(\alpha,\beta)\in C(\lambda,\mu)_\alpha$, we have
$(\alpha,\beta)\le(\alpha,\beta_\alpha)$.
\end{lemma}

\begin{proof}
    For any $(\alpha,\beta)\in C(\lambda,\mu)$, the definition immediately yields 
    $\tilde\alpha\cup[\beta-\gamma_{\lambda,\mu}]=\nu$.
  This implies that 
    $[\beta-\gamma_{\lambda,\mu}]=\nu\setminus \tilde\alpha =\beta_\alpha-\gamma_{\lambda,\mu}$. 
   From this, we conclude that   $\beta\trianglelefteq\beta_\alpha$ by \cite[Chapte I, (1.12)]{Mac}. Combining this with \eqref{order-union}, we obtain  $\bar\alpha\cup\beta\trianglelefteq \bar\alpha\cup\beta_\alpha$. Thus,     $(\alpha,\beta)\le (\alpha,\beta_\alpha)$ as desired. 
\end{proof}

 We define an undirected graph structure on the set $V_{\lambda,\mu}$ as follows.  
 We attach an edge colored by $i$ between $\alpha$ and $\alpha'$,  denoted by  $\alpha \overset{i}\sim \alpha' $, 
if the following conditions hold:
\begin{enumerate}
    \item $\tilde \alpha'$ is obtained from $ \tilde \alpha$ by replacing some $\tilde\alpha_i=\nu_r$ with 
$\nu_s$ such that $\nu_r\neq \nu_s$.
\item There is no  $\nu_i$ such that $\nu_r<\nu_i<\nu_s$ or $\nu_r>\nu_i>\nu_s$.
\end{enumerate}

\begin{lemma}
For any $\alpha\in V_{\lambda,\mu}$, there is a path from $\alpha$ to $\lambda$. In particular, 
    $V_{\lambda,\mu}$ is a connected undirected graph.  
\end{lemma}
\begin{proof}
     We construct such a path in two steps.
    
\textbf{Step (1):} Suppose there is some $i\in\{1,\cdots,k\}$ such that $\tilde\alpha_i<\tilde\lambda_i$ and $\tilde\alpha_j\ge \tilde\lambda_j$, for all $ 1\le j\le i-1$.  Let $M:=\{h|\tilde\alpha_{i-1}\ge\tilde\alpha_h>\tilde\alpha_i,1\le h\le k\}$ 
and $N:=\{h|\tilde\alpha_{i-1}\ge\nu_h>\tilde\alpha_i,h\ge 1\}$.
We claim $\sharp M<\sharp N$.
In fact, for any $ h\in M$, we must have  $h\in\{1,\cdots,i-1\}$.
Then, 
\[\tilde\alpha_{i-1}\ge\tilde\alpha_h\ge\tilde\lambda_h\ge\tilde\lambda_i>\tilde\alpha_i.\]
Recall that $\nu_{p_{h-1}+h}=\tilde \lambda_h$, so  $p_{h-1}+h\in N$. On the other hand, we have 
$p_{i-1}+i\in N$ since 
 $\tilde\alpha_{i-1}\ge\tilde\lambda_{i-1}\ge \tilde\lambda_i>\tilde\alpha_i$.
 Thus, $\sharp M<\sharp N$ as $N$ contains at least one more element than $M$.

 By the claim, we 
may find a $\nu_r$,   the smallest element  of $\{\nu_1,\cdots,\nu_{k+l(\mu)}\}$ such that $r\in N$ and 
$\tilde\alpha^1=(\tilde\alpha_1,\cdots,\tilde\alpha_{i-1},\nu_r,\tilde\alpha_{i+1},\cdots,\tilde\alpha_k)\in V_{\lambda,\mu}$.  
Moreover, $\alpha\overset{i}{\backsim}\alpha^1$ and    $\tilde\alpha_{i-1}\ge \nu_r>\tilde\alpha_{i}$. 
By repeating the above argument,  we can find a path from $\alpha$ to $\alpha^{m}$ such that
\begin{equation}
\label{equ:stepone1}
\alpha\overset{i}{\backsim}\alpha^{1}\overset{i}{\backsim}\alpha^{2}\overset{i}{\backsim}\cdots\overset{i}{\backsim}\alpha^{m}
\end{equation}
and $\tilde\alpha^{j}=(\tilde\alpha_1,\cdots,\tilde\alpha_{i-1},x(j),\tilde\alpha_{i+1},\cdots,\tilde\alpha_k)$, where $x(j)<x(j+1),\forall j$ and $x(m)=\tilde\lambda_i$.  

Now repeating the construction of the path (with each edge colored by the some $i$) in this fashion,  we can find a longer  path from $\alpha$ to $\theta$ such that
$$
\alpha\overset{i_1}{\backsim}f^{1}\overset{i_2}{\backsim}f^{2}\overset{i_3}{\backsim}\cdots\overset{i_s}{\backsim}f^{s}=\theta,
$$ 
where the resulting $\theta\in V_{\lambda,\mu}$ satisfies  $\tilde\theta_i\ge\tilde\lambda_i, \forall i=1,\cdots,k$, and  each component $f^j\overset{i_j}\sim f^{j+1}$ is a path constructed as in \eqref{equ:stepone1}, labeled by the same color $i_j$, with $i_1\le i_2\le\cdots\le i_s$.

\textbf{Step (2):} Next,  we  construct a path from $\theta$ to $\lambda $ where 
$\tilde \theta_i\ge \tilde \lambda_i$, for all $1\le i\le k$.
Suppose that there is  some $s\in\{1,\cdots,k\}$ 
 such that  $\tilde\theta_s>\tilde\lambda_s$ and $\tilde\theta_r=\tilde\lambda_r$ for $ r\ge s+1$.
Let $M':=\{ h|\tilde\theta_s>\tilde\theta_h\ge\tilde\theta_{s+1},1\le h\le k\}$   and $N':=\{ h|\tilde\theta_s>\nu_h\ge\tilde\theta_{s+1},h\ge 1\}$.
Using the same argument as that for the claim in Step (1) we see  that  $\sharp M'<\sharp N'$.
Thus, as in Step (1), we
  can find a path from $\theta$ to $\theta^{m'}$ such that
\begin{equation}
\label{equ:steptwo1}
 \theta\overset{s}{\backsim}\theta^1\overset{s}{\backsim}\theta^2\overset{s}{\backsim}\cdots\overset{s}{\backsim}\theta^{m'}   
\end{equation} 
with  $\tilde\theta^j=(\tilde\theta_1,\cdots,\tilde\theta_{s-1},y(j),\tilde\theta_{s+1},\cdots,\tilde\theta_k)$, where $y(j)>y(j+1),\forall j$ and $y(m')=\tilde\lambda_s$. 
Repeating this procedure,  we can find a path from $\theta$ to $\lambda$ such that
$$
\theta\overset{s_1}{\backsim}p^{1}\overset{s_2}{\backsim}p^{2}\overset{s_3}{\backsim}\cdots\overset{s_n}{\backsim}p^{n'}=\lambda, s_1\ge s_2\ge\cdots\ge s_{n'}.
$$

 Finally, the required path from 
 $\alpha$ to $\lambda$ is obtained by combining the path in
Step (1) first from $\alpha$ to $\theta$ and then the path in Step (2)  from $\theta$ to $\lambda$. This shows that $V_{\lambda,\mu}$  is a connected undirected graph.
\end{proof}

\begin{lemma}
\label{Lem:maxmalofC}
    The element $(\lambda,\mu)$
    is the maximal element in $C(\lambda,\mu)$ with respect to the order $\le $ in \eqref{order-on-pairs}.
\end{lemma}
\begin{proof}
Thanks to Lemma \ref{Lem:maximalalpha}, 
    it suffices to prove that $(\alpha,\beta_{\alpha})<(\alpha^1,\beta_{\alpha^1})<\ldots<(\alpha^m,\beta_{\alpha^m})$ and $(\theta,\beta_{\theta})<(\theta^1,\beta_{\theta^1})<\ldots<(\theta^{m'},\beta_{\theta^{m'}})$ corresponding to \eqref{equ:stepone1} and \eqref{equ:steptwo1}. Furthermore, it suffices to show $(\alpha,\beta_{\alpha})<(\alpha^1,\beta_{\alpha^{1}})$ and $(\theta,\beta_{\theta})<(\theta
    ^1,\beta_{\theta^{1}})$ since the others may be proved by the same arguments.
   
    Recall $\beta_\alpha=\nu\setminus \tilde\alpha+\gamma_{\lambda,\mu}$ and $\alpha \overset{i}\sim \alpha^1$. Let $h=\sharp\{\nu_j|\nu_j\ge\nu_r\}-i$. We may write 
    $\beta_\alpha-\gamma_{\lambda,\mu}=\nu\setminus \tilde \alpha$   as $(x_1,\ldots,x_h,\nu_r,x_{h+2},\ldots)$, where  $x_j$'s are components of $\nu$ and  
     $x_1\ge x_2\ge\ldots\ge  x_h\ge\nu_r>\tilde\alpha_i\ge x_{h+2}$.
     Moreover, $\beta_{\alpha^1}-\gamma_{\lambda,\mu}= \nu \setminus \tilde \alpha^1=(x_1,\ldots,x_h,\tilde\alpha_i,x_{h+2},\ldots)$. 
    Thus, there are only two different parts between 
    $\bar \alpha\cup \beta_\alpha$ and  $\bar \alpha^1\cup \beta_{\alpha^1}$. More explicitly,    
\[
\bar\alpha\cup\beta_{\alpha}=[\tilde\alpha_i+(k-i),\nu_r+\gamma_{h+1}]\cup\sigma,\quad 
\bar\alpha^1\cup\beta_{\alpha^1}=[\tilde\alpha_i+\gamma_{h+1},\nu_r+(k-i)]\cup\sigma,
\]
where $\sigma$ represents the partition consisting of the other common elements.
Note that $\nu_r\le\tilde\lambda_i$ implies $\#\{\nu_j|\nu_j\ge\nu_r\}> p_{i-1}+(i-1)$ by \eqref{nu}, i.e.,    $h+1>p_{i-1}$. Therefore,  $\gamma_{h+1}\le k-i$ and we have
\[[\tilde\alpha_i+(k-i),\nu_r+\gamma_{h+1}]\trianglelefteq [\tilde\alpha_i+\gamma_{h+1},\nu_r+(k-i)]=(\nu_r+(k-i), \tilde\alpha_i+\gamma_{h+1})\]
 The equality holds if and only if  $k-i=\gamma_{h+1}$. In any case, we conclude that $(\alpha,\beta_\alpha)<(\alpha^1,\beta_{\alpha^1})$ since   $\alpha<\alpha^1$.

Next we prove $(\theta,\beta_{\theta})<(\theta
    ^1,\beta_{\theta^{1}})$.
    Recall that $\theta \overset {s}\sim \theta^1$. Let 
$z=\sharp\{\nu_j|\nu_j\ge\tilde\theta_s\}-s$.
Similar to the relation between $\alpha$ and $\alpha^1$, we have 
 \[\bar\theta\cup\beta_{\theta}=[\tilde\theta_s+(k-s),y(1)+\gamma_{z+1}]\cup\eta, \quad 
\bar\theta^1\cup\beta_{\theta^1}=[\tilde\theta_s+\gamma_{z+1},y(1)+(k-z)]\cup\eta,
\]
where $\eta$ represents the partition consisting of the other common elements.
Note that $\tilde\theta_s>\tilde\lambda_s$ implies $\#\{\nu_j|\nu_j\ge\tilde\theta_s\}\le p_{s-1}+(s-1)$ by \eqref{nu}, i.e.,  $z+1\le p_{s-1}$. Therefore,  $\gamma_{z+1}> k-s$ and we have 
$$
[\tilde\theta_s+(k-s),y(1)+\gamma_{z+1}]\lhd [\tilde\theta_s+\gamma_{z+1},y(1)+(k-z)]=(\tilde\theta_s+\gamma_{z+1}, y(1)+(k-z)).
$$ 
This implies that $\bar\theta\cup\beta_{\theta}\lhd \bar\theta^1\cup\beta_{\theta^1}$
by \eqref{order-union},
and hence $(\theta,\beta_\theta)<(\theta^1,\beta_{\theta^1})$. This completes the proof.
\end{proof}

\begin{lemma}
\label{partial-order}
Let $(\lambda,\mu)\in C_{k,d}$. Then
for any $(\alpha,\beta)\in C_{k,d}$ such that  $\tilde\alpha\cup[\beta-\gamma_{\lambda,\mu}]\trianglelefteq\tilde\lambda\cup(\mu-\gamma_{\lambda,\mu})$, we have $(\alpha,\beta)\le (\lambda,\mu)$.
\end{lemma}
\begin{proof}
    We proceed by induction on the dominance order of $\tilde\lambda\cup(\mu-\gamma_{\lambda,\mu})$. If $\tilde\lambda\cup(\mu-\gamma_{\lambda,\mu})$ is minimal or $\tilde\alpha\cup[\beta-\gamma_{\lambda,\mu}]=\tilde\lambda\cup(\mu-\gamma_{\lambda,\mu})$, then $ (\alpha,\beta)\in C(\lambda,\mu)$ and the result holds by Lemma \ref{Lem:maxmalofC}.
   
   Now, we assume that $\tilde\alpha\cup[\beta-\gamma_{\lambda,\mu}]\lhd\tilde\lambda\cup(\mu-\gamma_{\lambda,\mu})$.  
    By  \cite[(1.16)]{Mac},  
    there exists some partition  $\theta$  such that
$\tilde\alpha\cup[\beta-\gamma_{\lambda,\mu}]\trianglelefteq\theta
$ and     
     $\tilde\lambda\cup(\mu-\gamma_{\lambda,\mu})=R_{{i,j}}\theta$,
     where $R_{{i,j}}$ is a raising operator defined in \eqref{equ:rasingop}.
    We claim that there exists some $(\alpha',\beta_{\alpha'})\in C_{k,d}$ such that 
    $\theta=\tilde \alpha' \cup [\beta_{\alpha'}-\gamma_{\lambda,\mu}]$, $\gamma_{\alpha',\beta_{\alpha'}}=\gamma_{\lambda,\mu}$ and $(\alpha', \beta_{\alpha'})\le (\lambda, \mu)$. If the claim holds, the induction hypothesis yields $(\alpha,\beta)\le (\alpha',\beta_{\alpha'})$, and the result follows.

    It remains to prove the claim. Write 
    $\theta=(\cdots,\nu_i-1,\cdots,\nu_j+1,\cdots)$ for some $1\le i<j\le k+l(\mu)$. 
  We choose $\alpha'\in\SPar_{a,k}$ such that  $\tilde\alpha'=(\theta_{p_0+1},\theta_{p_1+2},\cdots,\theta_{p_{k-1}+k})$. Let $\beta'= \theta\setminus \tilde\alpha'+\gamma_{\lambda,\mu}$. By \eqref{equ:defofnu}, this ensures $\gamma_{\alpha',\beta'}=\gamma_{\lambda,\mu}$ and $\beta'=\beta_{\alpha'}$.

 Let $S:=\{p_0+1,p_1+2,\cdots,p_z+z+1,\cdots,p_k+k+1\}$. We divide the proof of $(\alpha',\beta_{\alpha'})< (\lambda,\mu)$ into the following four cases.
    \begin{itemize}
        \item[(1)] $i,j\in S$:  Suppose  $i=p_{i'}+i'+1$ and $j=p_{j'}+j'+1$ for some $ 0\le i'<j'\le k-1$. By \eqref{equ:defofnu}, we have 
        $\bar\lambda =R_{i',j'}\bar\alpha'$ and $\beta_{\alpha'}=\mu$.
        Then, $\bar\alpha'\lhd\bar\lambda$, and hence $\bar\alpha'\cup \beta\lhd\bar\lambda\cup \mu$ by \eqref{order-union}, proving $(\alpha',\beta_{\alpha'})<(\lambda,\mu)$. 
    \item[(2)]  $i,j\notin S$: By \eqref{equ:defofnu}, we have $\bar \alpha'=\bar\lambda$ and $\mu=R_{i'',j''}\beta_{\alpha'} $ for some $i''<j''$. This implies  $(\alpha',\beta_{\alpha'})<(\lambda,\mu)$.
        \item[(3)] $i\in S,j\notin S$:  
       Suppose $i=p_{z-1}+z$ and $p_{z'-1}+z'<j<p_{z'}+z'+1$ for some $0\le z\le z'\le k-1$.
       This means $\nu_i=\bar\lambda_z+k-z$
       and $\nu_j=\mu_{j-z'}-k+z'$.
       Then 
       \[\bar\alpha'\cup\beta_{\alpha'}=[\nu_i+k-z-1,\nu_j+k-z'+1]\cup \sigma, \quad \bar\lambda\cup \mu =[\nu_i+k-z, \nu_j+k-z' ]\cup \sigma \]
       where $\sigma$ is the partition consisting of other common numbers.
       Since $\nu_i-1\ge \nu_j+1$, we have 
       $\nu_i+k-z-1\ge \nu_j+k-z+1\ge \nu_j+k-z'+1$
       and hence  $\bar\alpha'\cup\beta_{\alpha'}\lhd\bar\lambda\cup\mu$, yielding $(\alpha',\beta_{\alpha'})<(\lambda,\mu)$.
        \item[(4)]  $i\notin S,j\in S$: Similar argument as case (3) shows that    $\bar\alpha'\cup\beta_{\alpha'}\lhd\bar\lambda\cup\mu$ and hence $(\alpha',\beta_{\alpha'})<(\lambda,\mu)$.
    \end{itemize}
 This completes the proof of the claim and hence the lemma.
\end{proof}

\begin{lemma}
\label{indepedance of k wdots}
For any fixed $k\in \mathbb N$, 
the set $\{g_\lambda(y_{k+1},\cdots,y_a) e_\mu(y_1,\ldots,y_a)\mid \lambda\in \SPar_{a,k},\mu\in \Par_a\}$ is linearly independent in
     $\Z[y_1,\cdots,y_a]$, where 
     $g_\lambda(y_{k+1},\cdots,y_a)$ in \eqref{lead-terms-glambda}. 
\end{lemma}
\begin{proof}By degree considerations, 
it suffices to prove that for fixed $d\in\mathbb{Z}_{\ge 1}$, the set $g_{\lambda} e_\mu, (\lambda,\mu)\in C_{k,d}$ is linearly independent. To do this, 
Suppose that
\begin{equation}
\label{linear-equ}
    \sum_{\lambda,\mu} c_{\lambda,\mu} g_{\lambda} e_\mu=0,
\end{equation}
where the sum runs over all pairs $(\lambda,\mu)\in C_{k,d}$. 
Our goal is to show that $c_{\lambda,\mu}=0$ for all $(\lambda,\mu)$.

Recall $I_{k,\mu}$ in \eqref{equ:defofIkmu}, we have 
   \begin{align*}
    e_\mu(y_1,\cdots,y_a)&\overset{\eqref{er-sum}}=\sum_{i\in I_{\mu,k}}e_{[i]}(y_1,\cdots,y_k)e_{[\mu-i]}(y_{k+1},\cdots,y_a)\\
&=\sum_{\gamma\in\Par}e_\gamma(y_1,\cdots,y_k)\sum_{i\in I_{\mu,k},[i]=\gamma}e_{[\mu-i]}(y_{k+1},\cdots,y_a)\\
&\overset{\text{lem }\ref{lem:oderminus}}{=}\sum_{\gamma\in\Par}e_\gamma(y_1,\cdots,y_k)\big(a_\gamma e_{[\mu-\gamma]}'
+
\sum_{[\mu-i]\rhd[\mu-\gamma],i\in I_{\mu,k},[i]=\gamma}a_{[\mu-i]}e_{[\mu-i]}'\big)
\end{align*}
where $e'$ denotes the elementary symmetric polynomials in $\{y_{k+1},\cdots,y_a\}$, $a_\gamma\in \mathbb{Z}_{>0}$.
Now, for each pair $(\lambda,\mu)$,  we have 
\begin{align*}
    g_{\lambda} e_\mu(y_1,\cdots,y_a)
    &\overset{\eqref{lead-terms-glambda}}{=}\big(e_{\tilde\lambda}(y_{k+1},\cdots,y_a)+\sum_{\beta\rhd\tilde\lambda}d_\beta e_\beta(y_{k+1},\cdots,y_a)\big)e_\mu(y_1,\cdots,y_a)\\
&\overset{\eqref{order-union}}=\sum_{\gamma\in\Par}e_\gamma(y_1,\cdots,y_k)\big(a_\gamma e_{\tilde\lambda\cup[\mu-\gamma]}'
+
\sum_{\theta\rhd\tilde\lambda\cup[\mu-\gamma]}a'_{\theta}e'_\theta \big)
\end{align*}
for some
$d_\beta,a'_\theta\in\mathbb{Z}_{\ge 0}$.
This together with \eqref{linear-equ} and 
the linear independence of $e_\gamma(y_1,\cdots,y_k)$ for distinct $\gamma\in\Par$ implies that for each $\gamma\in \Par$,    
\begin{equation}
\label{sums=0_1}
\sum_{\lambda,\mu}c_{\lambda,\mu}\big(a_\gamma e_{\tilde\lambda\cup[\mu-\gamma]}'
+
\sum_{\theta\rhd\tilde\lambda\cup[\mu-\gamma]}a'_{\theta}e'_\theta \big)=0.
\end{equation}
Now we show that  $c_{\lambda,\mu}=0$ by induction on the order \eqref{order-on-pairs} of $(\lambda,\mu)$. If $(\lambda,\mu)$ is  minimal in this order, then by Lemma \ref{partial-order}, there is no other $(\alpha,\beta)\in C_{k,d}$ such that $\tilde\alpha\cup[\beta-\gamma_{\lambda,\mu}]\trianglelefteq\tilde\lambda\cup[\mu-\gamma_{\lambda,\mu}]$. In this case, we choose $\gamma=\gamma_{\lambda,\mu}$ in \eqref{sums=0_1}, which forces   $c_{\lambda,\mu}=0$.

If $(\lambda,\mu)$ is not minimal, we also choose 
$\gamma=\gamma_{\lambda,\mu}$ in \eqref{sums=0_1}, yielding 
\begin{equation}
\label{sums=0_2}
\sum_{\alpha,\beta}c_{\alpha,\beta}\big(
\sum_{\theta(\alpha,\beta)\trianglerighteq\tilde\alpha\cup[\beta-\gamma_{\lambda,\mu}]}d_{\theta(\alpha,\beta)}e'_{\theta(\alpha,\beta)} \big)=0,
\end{equation}
where $d_{\theta(\alpha,\beta)}\in \mathbb {Z}_{\ge 0}$.  
If $\theta(\alpha,\beta)=\tilde\lambda\cup[\mu-\gamma_{\lambda,\mu}]$, then Lemma \ref{partial-order} implies $(\alpha,\beta)\le (\lambda,\mu)$.
By the induction hypothesis,   $c_{\alpha,\beta}=0$ for $(\alpha,\beta)<(\lambda,\mu)$.  Consequently, the  coefficient of the term $e'_{\tilde\lambda\cup[\mu-\gamma_{\lambda,\mu}]}$ appears in \eqref{sums=0_2} is $c_{\lambda,\mu}a_{\gamma_{\lambda,\mu}}$, forcing   $c_{\lambda,\mu}=0$.
This completes the proof.
\end{proof}
 
\bibliographystyle{alpha}
\bibliography{affineQweb}
\end{document}